%% file: survey-metric-dim-par.tex
\documentclass[a4paper]{article}
\usepackage[utf8]{inputenc}
\usepackage{amsthm}
\usepackage{amsmath}
\usepackage{amssymb}
\usepackage{caption}
\usepackage{cite}
\usepackage{tikz}
\usepackage{subcaption}
\usepackage[labelformat=simple]{subcaption}
\usepackage{a4wide}
\usetikzlibrary{calc}

\usepackage{hyperref}

\usepackage{enumitem}
\setlist[enumerate]{noitemsep,topsep=3pt}
\setlist[itemize]{noitemsep,topsep=3pt}

\theoremstyle{definition}
\newtheorem{definition}{Definition}
\newtheorem{remark}[definition]{Remark}

\theoremstyle{plain}
\newtheorem{lemma}[definition]{Lemma}
\newtheorem{theorem}[definition]{Theorem}
\newtheorem{corollary}[definition]{Corollary}
\newtheorem{proposition}[definition]{Proposition}
\newtheorem{conjecture}[definition]{Conjecture}

\newcommand{\sdim}{{\rm sdim}}
\newcommand{\edim}{{\rm edim}}
\newcommand{\adim}{{\rm adim}}
\newcommand{\res}{{\rm res}}
\newcommand{\mdim}{{\rm mdim}}
\newcommand{\pd}{{\rm pd}}

\newcommand{\msrepr}{\ensuremath{\operatorname{m}}}
\newcommand{\msl}{\{\hspace*{-0.1cm}|}
\newcommand{\msr}{|\hspace*{-0.1cm}\}}

\newcommand{\ter}{\operatorname{ter}}

\author{Dorota Kuziak$^1$ and Ismael G. Yero$^2$\\[0.2cm]
$^1$ Department of Statistics and Research Operations,\\
$^2$ Department of Mathematics,\\
Universidad de C\'{a}diz, 11202 Algeciras, Spain}

\date{}

\title{Metric dimension related parameters in graphs: A survey on combinatorial, computational and applied results}

\begin{document}
\maketitle

\begin{abstract}
Topics concerning metric dimension related invariants in graphs are nowadays intensively studied. This compendium of combinatorial and computational results on this topic is an attempt of surveying those contributions that are of the highest interest for the research community dealing with several variants of metric dimension in graphs.
\end{abstract}

\section{Introduction}

The origin of researches concerning metric dimension of graphs is frequently relatively lost in the literature since such concept has arisen in connection with some other related and/or more general areas than that of graphs. For instance, considering the case of metric spaces in general, the notion of metric dimension dates back to 1953 \cite{Blumenthal-1953}. It is frequently said that metric dimension was independently introduced by Slater, and by Harary and Melter in \cite{Slater-1975}, and \cite{Harary-1976}, respectively. However, metric dimension in graphs actually dates back to \cite{Erdos-1963} and \cite{Erdos-1965}, two works of Erd\"os and others. Moveover, if we indeed continue digging into the databases (for instance MathSciNet), we see that there are other antecessors of such works which we have not been able to find.

Despite the fact it is old concept, their study became much popular just in the last two decades approximately. Before this only a few sporadic works appeared, from which we might remark the complexity result (indeed a comment without proof) given in the book \cite{Garey-1979}, claiming that computing the metric dimension of graphs is an NP-hard problem. The formal proof of this fact was further on given in \cite{Khuller-1996}. It is perhaps the work \cite{Chartrand-2000-a} which produced a significant turning point on the popularity of the topic in the research community.

Since the first results on metric dimension were published, a few interesting applications of the topic to practical problems have been raised up, as well as, a huge number of theoretical results have been published so far. Among the most interesting applications we remark the following ones.

In \cite{Chartrand-2000-a,Chartrand-2000-c}, the authors have described that the structure of some chemical compounds is frequently represented by a labeled graph where the vertex and edge labels specify the atom and bond types, respectively. In this sense, a lot of issues in the field of chemistry are related to obtaining a mathematical representation for chemical compounds, such that each one of these representations leads to different compounds. This means that some theoretical concepts of graph theory are related to such a chemical problem. Therefore, obtaining representations for the vertices of a graph, such that all vertices in it have different representations is ``in some sense'' related to the problem of providing representations of chemical compounds with respect to the others. For instance, the author of \cite{Johnson-1993,Johnson-1998} rediscovered the concepts of resolving sets while he was investigating some aspects about pattern recognition into a chemical compound in a pharmacy company.

Furthermore, metric dimension has some applications to problems of pattern recognition and image processing, some of which involve the use of hierarchical data structures \cite{Melter-1984}. Other applications to navigation of robots in networks and other areas appeared in \cite{Chartrand-2000-a,Hulme-1984,Khuller-1996}. Some interesting connections between resolving sets in Hamming graphs and the Mastermind game or coin weighing have been presented in \cite{Caceres-2007}. Some other applications of parameters related to the metric dimension to combinatorial searching have been presented in \cite{Sebo-2004}. Specifically, there have been analyzed some problems on false coins arising from a connection between information theory and extremal combinatorics. Also, they have dealt with a combinatorial optimization problem related to finding ``connected joins'' in graphs. In such a work, several results about detection of false coins have been used to approximate the value of the strong variant of metric dimension of some specific graphs, like for example the Hamming graphs. Some other more recent applications of metric dimension parameters are given in \cite{Bailey-2019} and \cite{Trujillo-Rasua-2016}, where connections between such parameters with error correcting codes and privacy in social networks, respectively, are demonstrated.

Nowadays, the topic of metric dimension in graphs, in its classical version as well as in all the known variants, is very well studied, although there are still a lot of open problems of interest for the research community. We next give some data which show this. The next table contains the results of several queries made on a number of databases, and the quantity of entries obtained in each query. Note that the results of Scholar Google can be ``significantly inflated''.

\begin{center}
\begin{tabular}{|c|c|c|c|}
  \hline
  Database & Query & Lifetime works  & Last 10 years \\ \hline
  MathSciNet & Metric dimension \& graph & 232  & 193 \\
   & Resolving set \& graph & 18  & 16 \\ \hline
  Web of Science & Metric dimension \& graphs & 228  & 209 \\
   & Resolving set \& graph & 28  & 22 \\ \hline
  Scholar Google & Metric dimension \& graph & 3080  & 2510 \\
   & Resolving set \& graph &  1090 & 955 \\ \hline
  DBLP & Metric dimension \& graph & 146 & 133 \\
   & Resolving set \& graph & 14 & 11 \\
  \hline
\end{tabular}
\end{center}

In connection with the classical version of metric dimension, it has recently appeared the survey \cite{Tillquist-2021}, which is a relatively fairly complete compendium containing the vast majority of mainly the most important results on metric dimension of graphs. However, not much appears there regarding the huge number of versions of metric dimension which are giving more insight into it. Another significant work which also surveys/reviews the literature of metric dimension in graphs, mainly addressed to the existent connections with other topics like group theory and topology is \cite{Bailey-2011}. In view of these facts, we are aimed to present here an extensive survey on the most important results concerning variants of metric dimension in graphs. We remark that in view of \cite{Tillquist-2021}, we shall not recall (unless specifically necessary) any result concerning the classical concept of metric dimension.

\section{The classical version}

From now on in this survey, we consider $G=(V,E)$ is a connected, undirected and simple graph without loops and multiple edges. The set of vertices and of edges of $G$ shall be written as $V(G)$ and $E(G)$, respectively. The \emph{order} and the \emph{size} of $G$ are $n=|V(G)|$ and $m=|E(G)|$. Given two vertices $u,v\in V(G)$, the \emph{distance} between $u$ and $v$ is the length of a shortest path joining $u$ and $v$, and is denoted by $d_G(u,v)$. The two vertices $u,v$ are (\emph{resolved, identified} or \emph{recognized}) by a vertex $x\in V(G)$ if $d_G(u,x)\ne d_G(v,x)$. It is also said that $x$ (\emph{resolves, identifies} or \emph{recognizes}) the pair of vertices $u,v$. Moreover, by $R_G\{x, y\}$ we mean the set of all vertices of $G$ that resolves the pair $x,y$ (note that $x,y\in R_G\{x, y\}$). Whenever possible, the subindex $G$ in the above notations could be removed if the graph $G$ is clear from the context. Based on these definitions, the following concepts are the heart of the exposition.

\begin{definition}
A set of vertices $S\subset V(G)$ is a \emph{resolving set} of $G$ if any two vertices in $V(G)$ are resolved by a vertex of $S$. A resolving set with the smallest possible cardinality is called a \emph{metric basis}. The cardinality of a metric basis in $G$ is the \emph{metric dimension} of $G$, denoted by $\dim(G)$.
\end{definition}

To better get the point in such definition, we use for instance the following example graph given in Figure \ref{Fig_resolving-basis}. There appear a graph $G$ on 12 vertices, which is indeed the so-called grid graph $P_4\Box P_3$ (Cartesian product of the two paths $P_4$ and $P_3$. It can be noted that for any pair of vertices in this grid, there is at least one vertex in red (or in blue) which recognizes such pair. Notice that pairs of vertices formed by vertices which are in the set itself are also recognized by themselves. For the grid graphs, it is well known that any two ``consecutive corner vertices'' are enough to uniquely identify all the vertices of such grids.

\begin{figure}[h]
\centering
\begin{tikzpicture}[scale=.7, transform shape]
\node [draw, shape=circle, fill=red] (a1) at  (0,0) {};
\node [draw, shape=circle, fill=red] (a5) at  (0,1.5) {};
\node [draw, shape=circle, fill=red] (a9) at  (0,3) {};
\node [draw, shape=circle] (a2) at  (3,0) {};
\node [draw, shape=circle] (a6) at  (3,1.5) {};
\node [draw, shape=circle] (a10) at  (3,3) {};
\node [draw, shape=circle] (a3) at  (6,0) {};
\node [draw, shape=circle] (a7) at  (6,1.5) {};
\node [draw, shape=circle] (a11) at  (6,3) {};
\node [draw, shape=circle, fill=blue] (a4) at  (9,0) {};
\node [draw, shape=circle, fill=blue] (a8) at  (9,1.5) {};
\node [draw, shape=circle, fill=blue] (a12) at  (9,3) {};

\draw(a1)--(a2)--(a3)--(a4)--(a8)--(a7)--(a6)--(a5)--(a1);
\draw(a5)--(a9)--(a10)--(a11)--(a12)--(a8);
\draw(a2)--(a6)--(a10);
\draw(a3)--(a7)--(a11);
\end{tikzpicture}
\hspace*{0.6cm}
\begin{tikzpicture}[scale=.7, transform shape]
\node [draw, shape=circle, fill=blue] (a1) at  (0,0) {};
\node [draw, shape=circle] (a5) at  (0,1.5) {};
\node [draw, shape=circle, fill=red] (a9) at  (0,3) {};
\node [draw, shape=circle] (a2) at  (3,0) {};
\node [draw, shape=circle] (a6) at  (3,1.5) {};
\node [draw, shape=circle] (a10) at  (3,3) {};
\node [draw, shape=circle] (a3) at  (6,0) {};
\node [draw, shape=circle] (a7) at  (6,1.5) {};
\node [draw, shape=circle] (a11) at  (6,3) {};
\node [draw, shape=circle, fill=blue] (a4) at  (9,0) {};
\node [draw, shape=circle] (a8) at  (9,1.5) {};
\node [draw, shape=circle, fill=red] (a12) at  (9,3) {};

\draw(a1)--(a2)--(a3)--(a4)--(a8)--(a7)--(a6)--(a5)--(a1);
\draw(a5)--(a9)--(a10)--(a11)--(a12)--(a8);
\draw(a2)--(a6)--(a10);
\draw(a3)--(a7)--(a11);
\end{tikzpicture}
\caption{Two resolving sets $($one in red other in blue$)$ in the left hand side graph, and two metric bases $($one in red other in blue$)$ in the right hand side graph.}\label{Fig_resolving-basis}
\end{figure}
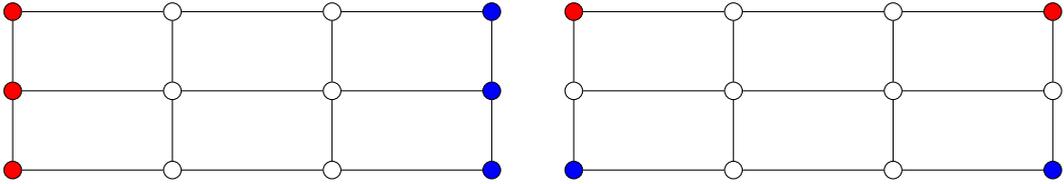

Other useful point of view for defining a resolving set is as follows. For a set $S=\{u_1,\dots,u_r\}\subset V(G)$, and a vertex $x\in V(G)$, the \emph{metric representation} of $x$ with respect to $S$ is the vector $r(x|S)=(d_G(x,u_1),\dots,d_G(x,u_r))$. With this notation in mind, the set $S$ is a resolving set for $G$ if the metric representations of all vertices of $G$ are pairwise different. This indeed means that for every pair of distinct vertices of $G$, their metric representations differ in at least one position.

It is necessary to remark some points regarding the definitions above. Resolving sets have been also called locating sets (in \cite{Slater-1975}) and metric generators (first time in \cite{Sebo-2004}). The metric dimension of graphs has been denoted in at least three different manners, where $\beta(G)$, $\mu(G)$ and $\dim(G)$ are the most common ones. Metric dimension was also called locating number in \cite{Slater-1975}. From our point of view, the notation $\dim(G)$ is the most natural one, since for instance the notation $\beta(G)$ is also used for the independence number, which is a very classical parameter in graph theory. In concordance, from now on we shall use $\dim(G)$ for the whole exposition, and we indeed would suggest researchers to unify their works with this notation.

Several variations of metric dimension in graphs are nowadays more or less well known and studied. In general, the metric dimension parameters can be classified into five types. Notice that we do not mention all the instances of such parameters, but just some of the most remarkable ones, from our point of view.\\

\noindent
1) Resolving sets which also satisfy other properties of the graph:
\begin{itemize}
  \item[--] resolving dominating set \cite{Brigham-2003} - a resolving set which is also a dominating set;
  \item[--] independent resolving set \cite{Chartrand-2003} - a resolving set which is also an independent set;
  \item[--] connected resolving set \cite{Saenpholphat-2003} - a resolving set which is also a connected set.
\end{itemize}

\noindent
2) Resolving sets which have a modified condition of resolvability:
\begin{itemize}
  \item[--] strong resolving set \cite{Sebo-2004} - a set that for any two vertices of the graph, one of them belongs to a shortest path between the other vertex and an extra vertex from such set;
  \item[--] local resolving set \cite{Okamoto-2010} - a set such that every two adjacent vertices of the graph have distinct vectors of distances to the vertices in this set;
  \item[--] adjacency resolving set \cite{Jannesari-2012} - a set such that any two different vertices not belonging to the set have different neighborhood in this set.
\end{itemize}

\noindent
3) Partitions of the vertex set of a graph having some resolving properties:
\begin{itemize}
  \item[--] resolving partition \cite{Chartrand-2000-b} - a partition such that every two different vertices of the graph have distinct vectors of distances to the sets of the partition;
  \item[--] strong resolving partition \cite{Yero-2014-a} - a partition where every two different vertices of the graph belonging to the same set of the partition are strongly resolved by some set of the partition;
  \item[--] metric coloring \cite{Chartrand-2009} - a partition such that every two adjacent vertices of the graph have distinct vectors of distances to the set of the partition.
\end{itemize}

\noindent
4) Resolving sets that are extensions of the classical resolving sets:
\begin{itemize}
  \item[--] $k$-resolving set \cite{Estrada-Moreno-2015} - a set such that any pair of vertices of the graph is distinguished by at least $k$ vertices of this set
  \item[--] simultaneous resolving set \cite{Ramirez-Cruz-2016} - a set which is simultaneously a metric generator for a given family of connected graphs with a common vertex set.
\end{itemize}

\noindent
5) Resolving sets which identify other elements of the graphs:
\begin{itemize}
  \item[--] edge resolving set \cite{Kelenc-2018} - a set such that any pair of edges of the graph is distinguished by the vertices of this set;
  \item[--] mixed resolving set \cite{Kelenc-2017} - a set such that any pair of elements (vertices or edges) of the graph is distinguished by the vertices of this set;
  \item[--] solid resolving set \cite{Hakanen-2020} - a set that uniquely identifies not only pairs of vertices but also pairs of subsets of vertices of the graph.
\end{itemize}

According to the amount of literature concerning all the variants of this topic, and in view of the excellent recent survey \cite{Tillquist-2021}, it is now our goal to present a compendium containing (in our opinion) some of the most interesting contributions published so far on the most important variants of metric dimension in graphs. For each surveyed variant, the most interesting results, research lines and open questions shall be mentioned. In connection with this, we need to remark that the topic of identifying codes (see \cite{Karpovsky-1998}) will be not addressed here since from our point of view, the literature on this topic and related ones is huge and indeed deserves to be surveyed separately. By a similar reason, topics of ``location'' in graphs (locating dominating sets, locating total dominating sets, open locating dominating sets, metric locating dominating sets, locating chromatic partitions, metric locating partitions, resolving dominating partitions, and a long etc.) shall not be addressed in this exposition, since the number of such parameters (many of them quite similar to each other), and results in this direction is also enormous that would deserve a separate survey for themselves.

Another remark we need to say is that, although there have appeared several contributions on metric dimension parameters for the cases of infinite graphs and of digraphs, we shall not mention in this survey any result concerning these parts. That is, all the graphs considered here are finite and non directed.

\subsection{Some terminology and notation}\label{subsec:terminology}

In order to better develop this exposition, we here describe some extra terminology, basic concepts and notations that will be useful for the reader. For any other very basic terminology on graph theory, we refer the reader to the book \cite{West-2001}.

We begin with some concepts defined on trees, where a vertex of degree at least three is called a \emph{major vertex}, and a leaf $u$ is a \emph{terminal vertex} of a major vertex $v$ if $d_T(u, v) < d_T(u, w)$ for every other major vertex $w$. The set of leaves of $T$ shall be denoted as $L(T)$ and we also shall write $l(T)=|L(T)|$. Now, the \emph{terminal degree} $ter_T(v)$ of a major vertex $v$ is the number of terminal vertices of $v$. A major vertex $v$ is called an \emph{exterior major vertex} if $ter_T(v)>0$. Moreover, by $M(G)$ and $ex(T)$ we denote the set and the number of exterior major vertices of $T$, respectively. By $ex_k(T)$, the number of exterior major vertices $u$ for which $ter_T(u)=k$. Note that all the concepts and notations above can be extended to non tree graphs in a natural way.

A shortest path between two vertices $u,v\in V(G)$ is called a \emph{geodesic} and, it is usually written as a $u,v$-geodesic. The \emph{diameter} of $G$ is the largest possible distance between any two vertices of $G$ and is denoted by $D(G)$. The \emph{open neighborhood} of a vertex $v\in V(G)$ is the set of neighbors of $v$ in $G$, and is denoted as $N_G(v)$. The \emph{closed neighborhood} of $v$ is $N_G[v]=N_G(v)\cup \{v\}$. Two distinct vertices are called \emph{true twins} if they have the same closed neighborhood, and they are called \emph{false twins} if they have the same open neighborhood. A vertex $v$ is a \emph{twin} if there exists another vertex $u$ such that $u,v$ are either true twins or false twins. It is well known that the property of being a twin forms an equivalence relation in a graph, and the equivalence classes can be of type true twins class or false twin class or trivial class (this latter one whenever a vertex is not a twin).

A vertex $u$ of $G$ is said to be \emph{maximally distant} from other vertex $v$ if every vertex $w\in N_G(u)$ satisfies that $d_G(v,w)\le d_G(v,u)$. The set of all vertices of $G$ that are maximally distant from some vertex of the graph is called the {\em boundary} of the graph, and is denoted by $\partial(G)$. If a vertex $u$ is maximally distant from other distinct vertex $v$, and $v$ is maximally distant from $u$, then it is said that $u$ and $v$ are \emph{mutually maximally distant}. In the case of complete graphs $K_n$, complete bipartite graphs $K_{r,s}$, cycles $C_n$ and hypercube graphs $Q_k$, the boundary is simply the whole vertex set. On the other hand, the boundary of a tree consists of its leaves. In the grid graph of Figure \ref{Fig_resolving-basis}, any two ``diagonal corner vertices'' are mutually maximally distant.

A set of vertices $S\subset V(G)$ is a \emph{vertex cover set} of $G$, if every edge of $E(G)$ is incident to at least one vertex of $S$. The cardinality of a smallest possible vertex cover set of $G$, is the \emph{vertex cover number} of $G$, denoted by $\alpha(G)$. It is very well known that the complement of a vertex cover set forms an \emph{independent set}, which is an edgeless set of vertices. The \emph{independence number} is the cardinality of a largest independent set of a graph $G$, denoted as $\beta(G)$. Gallai's theorem precisely states that, for any graph $G$ of order $n$, $\alpha(G)+\beta(G)=n$.

Throughout this exposition, several allusions will be made regarding some products of graphs: Cartesian ($\Box$), direct ($\times$), strong ($\boxtimes$), lexicographic ($\circ$) as those called the four standard ones; and other ones like corona ($\odot$), rooted ($\circ_v$, with respect to a vertex $v$), and hierarchical, which are more known as kind of ``operations'' with graphs rather than products. For definitions of these products, and more information about them, we suggest the excellent book \cite{Hammack-2011}.

Whenever $G$ and $H$ are two isomorphic graphs, we shall write $G\cong H$. For a set of vertices $S\subset V(G)$, by $G[S]$ we represent the subgraph of $G$, induced by the set $S$, that is, a graph $G'$ with vertex set $V(G')=S$, and two vertices $u,v\in S$ are adjacent in $G'$ if and only if they are adjacent in $G$. We want to recall that through our exposition, in every notation defined above, whenever there is a subindex with regard to the graph $G$, we could remove such subindex if the graph is clear from the context.


\input{strong-dim.tex}

\input{local-dim.tex}

\input{adjacency.tex}

\input{k-dim.tex}

\input{partition-dim.tex}

\input{edge-dims.tex}

\input{fractional.tex}

\input{k-antidim.tex}

\section{Final remarks}

In this final section, we include several other interesting metric dimension related parameters which are known from the literature. For each of them, we include their definitions, the seminal publication where they appeared, and in some cases, some remarkable publications on each topic.

\subsection*{Doubly resolving sets}

This concept was born in connection with studying the metric dimension of Cartesian product graphs, and was first introduced in \cite{Caceres-2007}. If we consider the Cartesian product graph $G\Box H$, and two resolving sets in $A$ and $B$ of $G$ and $H$, respectively, then the set $A\times B$ is not necessarily a resolving set of $G\Box H$, as it usually happens with many graph parameters. That is, there could exist two vertices $(g,h),(g',h')\in V(G\Box H)$ for which $d_{G\Box H}((g,h),(x,y))=d_{G\Box H}((g',h'),(x,y))$ for every $(x,y)\in A\times B$, even so that there are $x'\in A$ and $y'\in B$ such that $d_G(g,x')\ne d_G(g',x')$ and $d_H(h,y')\ne d_H(h',y')$.

However, if the notion of resolving set is substituted by that of doubly resolving set, then one can prove that the Cartesian product of one doubly resolving set in one factor of the Cartesian product and one resolving set in the second factor, produces a resolving set in the product. It can be indeed proved a stronger result. That is, the metric dimension of any connected Cartesian product graph is bounded above by the sum of the metric dimension of one factor and the smallest cardinality of a doubly resolving set in the second factor minus one. With this motivation in mind the doubly resolving sets in graphs were introduced in \cite{Caceres-2007} as follows.

Two vertices $v, w\in V(G)$ are \emph{doubly resolved} by $x, y \in V(G)$ if $d_G(v, x)-d_G(w, x) \ne d_G(v, y)-d_G(w, y)$. A set of vertices $S\subset V(G)$ is a \emph{doubly resolving set} for $G$, if every pair of distinct vertices $v, w \in V(G)$ are doubly resolved by two
vertices in $S$. The smallest possible cardinality an any doubly resolving set for $G$ was denoted by $\psi(G)$ in \cite{Caceres-2007}, although they did not use any name for this, and we could consider call this cardinality as the \emph{doubly resolving number} of $G$.

Since \cite{Caceres-2007}, the topic of doubly resolving number of graphs became of interest and it has been separately studied in a few interesting investigations. We remark for instance, \cite{Kratica-2009,Mladenovic-2012} where some genetic algorithms were used to compute or approximate the values of such a parameter of graphs.

\subsection*{Multiset versions of metric dimension}

The idea behind this parameter is that of using multisets as the metric representations of the vertices in order to uniquely identify them. That is, given a vertex $u\in V(G)$ and a set $S\subset V(G)$, the \emph{multiset representation} of $u$ with respect to $S$, denoted $\msrepr(u|S)$,
is defined by $$\msrepr(u|S)=\msl d_G(u, w_1), \ldots, d_G(u, w_t) \msr,$$
where $\msl . \msr$ denotes a multiset.

With this definition, there are two variants of multiset versions of metric dimension. In the first one, from \cite{Simanjuntak-2017}, the authors say that a set $S$ of vertices of $G$ is a \emph{multiset resolving set} if for any two distinct vertices $u,v\in V(G)$ it follows that $\msrepr(u|S)\ne \msrepr(v|S)$. The smallest cardinality among all multiset resolving sets of $G$ is the multiset dimension of $G$. The inconvenience of this concept is that there are graphs in which such parameter cannot be computed because there are pairs of vertices that have the same multiset representation with respect to every set of vertices of the graphs (graphs having twins for instance). In such situation, authors of \cite{Simanjuntak-2017} took the agreement that such graph will have multiset dimension infinite, and indeed raised up the following open question.\\

\noindent
\textbf{Open question:} Characterize all the graphs with infinite multiset dimension.\\

In order to not deal with such situation, and have a parameter that can be computed for every graph, in \cite{Gil-Pons-2019} it was defined a related parameter as follows. The set $S\subset V(G)$ is an \emph{outer multiset resolving set} if for any two distinct vertices $u,v\in V(G)\setminus S$ it follows that $\msrepr(u|S)\ne \msrepr(v|S)$. The smallest cardinality among all outer multiset resolving sets of $G$ is the \emph{outer multiset dimension} of $G$. There has been no more studies on this parameter so far.

\subsection*{Resolving number}

The \emph{resolving number} of $G$ is the smallest integer $k$ such that every subset of cardinality $k$ in $G$ is a resolving set of $G$, and is denoted by $\res(G)$. This concept was introduced in \cite{Chartrand-2000-c}.

\subsection*{Simultaneous versions of metric dimension}

This version of metric dimension centers the attention into studying the metric dimension (or some of its variants) of not only graph, but of a family of graphs that can indeed be infinite. That is, consider $\mathcal{F}=G_1, G_2, \dots $ is a family of graphs defined over the same set of vertices $V=\{v_1,\dots,v_n\}$. A set $S\subset V$ is a \emph{simultaneous resolving set} of $\mathcal{F}$ if $S$ is a resolving set for every graph $G_i\in \mathcal{F}$. The smallest cardinality among all simultaneous resolving sets for $\mathcal{F}$ is the \emph{simultaneous metric dimension} of $\mathcal{F}$. These concepts were introduced first in \cite{Ramirez-Cruz-2016}.

Strong, local and adjacency variants of the simultaneous metric dimension have been introduced in \cite{Estrada-Moreno-2016-a}, \cite{Barragan-Ramirez-2017} and \cite{Ramirez-Cruz-2016-a}, respectively, in a natural way. We remark that concerning the simultaneous strong metric dimension of graph families, it was introduced in \cite{Yero-2020} the simultaneous version of the strong resolving graph for graph families. That is, a kind of ``union'' of the strong resolving graphs of all the graphs of the family. This construction allowed to prove, for instance, that computing the simultaneous strong metric dimension is NP-hard, even when restricted to families of graphs composed of only paths or only stars.

\subsection*{Solid metric dimension}

Having in mind the metric representation $r(v|S)$ of a vertex $v\in V(G)$ with respect to a set $S\subset V(G)$, it is generalized this notion to the metric representation $r(X|S)$ of a set $X=\{x_1,\dots,x_r\}\subset V(G)$ with respect to the set $S\subset V(G)$ as follows: $r(X|S)=(d_G(x,1,S),d_G(x_2,S),\dots,d_G(x_r,S))$ (the distance between a vertex and a set was already defined in Subsection \ref{subsec:pd}).

A set $S\subset V(G)$ is a \emph{solid-resolving set} of $G$ if for all vertices $x\in V(G)$ and nonempty subsets $Y\subset V(G)$, $r(x|S) = r(Y|S)$ implies that $Y = \{x\}$. The minimum cardinality among all solid-resolving sets of $G$ is called the \emph{solid-metric dimension} of $G$. Concepts above were introduced in \cite{Hakanen-2020} and further on generalized to $\{\ell\}$-resolving sets and studied in \cite{Hakanen-2020-a,Hakanen-2018}.

\subsection*{Threshold dimension}

The \emph{threshold dimension} of a graph $G$ is the minimum metric dimension among all graphs $H$ having $G$ as a spanning subgraph. In other words, the threshold dimension of $G$ is the minimum metric dimension among all graphs obtained from $G$ by adding edges. This concept was first presented in \cite{Mol-2020}. A strong variant (defined in a natural way with respect to the strong metric dimension) of it was published in \cite{Benakli-2021}.

\subsection*{Upper dimension}

A resolving set $S$ of a graph $G$ is a \emph{minimal resolving set} if no proper subset of $S$ is a resolving set. The maximum possible cardinality among all minimal resolving set of $G$ is the \emph{upper dimension} of $G$, denoted $\dim^+(G)$. These concepts were introduced in \cite{Chartrand-2000-c}. The resolving number, the metric dimension and the upper dimension are clearly related. That is, $1\le \dim(G)\le \dim^+(G) \le \res(G) \le n-1$ for every nontrivial connected graph $G$ of order $n$.

\section*{Acknowledgements}

The authors have been partially supported by the Spanish Ministry of Science and Innovation through the grant PID2019-105824GB-I00.

\input{biblio.tex}
\end{document}

%% file: strong-dim.tex
\section{Strong metric dimension} \label{Sec:strong-dim}

A resolving set for a graph $G$ uniquely distinguishes every vertex of $G$ by means of a vector of distances to such resolving set. The notion of strong resolving sets and strong metric dimension of graphs appears as an attempt to also uniquely distinguish graphs in the following sense. In  \cite{Sebo-2004}, authors discussed the following question: ``\emph{For a given resolving set $T$ of a graph $H$, whenever $H$ is a subgraph of a graph $G$ and the metric vectors of the vertices of $H$ relative to $T$ agree in both $H$ and $G$, is $H$ an isometric subgraph
of $G$}? As stated in \cite{Sebo-2004}, \emph{even though the metric vectors relative to a resolving set of a graph distinguish all pairs of
vertices in the graph, they  do not uniquely determine all distances in a graph.}''\footnote{A sentence from
\cite{Rodriguez-Velazquez-2014-a}.} In connection with this situation, a stronger notion of resolving sets was introduced
in \cite{Sebo-2004}.

Given a connected graph $G$, it is said that a vertex $w\in V(G)$ \emph{strongly resolves} two different vertices $u,v\in V(G)$ if
$d_G(w,u)=d_G(w,v)+d_G(v,u)$ or $d_G(w,v)=d_G(w,u)+d_G(u,v)$. That is, there is some shortest $w-u$ path
that contains $v$ or some shortest $w-v$ path containing $u$. A set $S\subset V(G)$ is a
\emph{strong resolving set} (or a \emph{strong metric generator}) for $G$, if every two vertices of $G$ are strongly resolved by some vertex of $S$. The cardinality of a smallest strong resolving set for $G$ is called the \emph{strong metric dimension} of $G$, denoted by $\sdim(G)$ (in some literature it is also denoted by $\dim_s(G)$, see for instance \cite{Rodriguez-Velazquez-2014-a}). A \emph{strong metric basis} of $G$ is a strong resolving set of cardinality $\sdim(G)$. We must remark that if two vertices are mutually maximally distant, then they are strongly resolved only by themselves. This means that any strong resolving set of a graph $G$ must contain at least one vertex from each pair of mutually maximally distant vertices of $G$.

In Figure \ref{Fig:strong-resolving} appears a graph with a strong metric basis in red color. Notice that for instance the two vertices in the center are mutually maximally distant. Thus, at least one of them must belong to a given strong resolving set. Similarly, the two left hand side vertices together with the two ride hand side ones are also pairwise mutually maximally distant, and so, an strong metric basis must contain at least three of them.

\begin{figure}[ht]
\centering
\begin{tikzpicture}[scale=.7, transform shape]
\node [draw, shape=circle,fill=red] (a1) at  (-4,-1) {1};
\node [draw, shape=circle,fill=red] (a2) at  (-4,1) {2};
\node [draw, shape=circle] (a3) at  (-3,0) {3};
\node [draw, shape=circle] (a4) at  (-1,0) {4};
\node [draw, shape=circle,fill=red] (a5) at  (0,-1) {5};
\node [draw, shape=circle] (a6) at  (0,1) {6};
\node [draw, shape=circle,fill=red] (a10) at  (4,-1) {10};
\node [draw, shape=circle] (a9) at  (4,1) {9};
\node [draw, shape=circle] (a8) at  (3,0) {8};
\node [draw, shape=circle] (a7) at  (1,0) {7};

\draw(a3)--(a1)--(a2)--(a3)--(a4)--(a5)--(a7)--(a8)--(a9)--(a10)--(a8);
\draw(a4)--(a6)--(a7);
\end{tikzpicture}
\caption{A strong resolving set of the smallest possible cardinality appears in red.}\label{Fig:strong-resolving}
\end{figure}
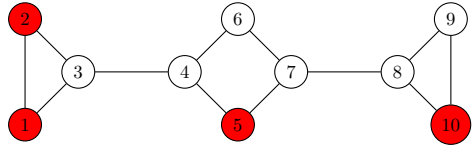

Some literature on the strong metric dimension of graphs was already surveyed in \cite{Kratica-2014}. Also, the Ph. D. dissertation \cite{Kuziak-PhD} contains a nice compendium of results in this topic, with emphasis in the case of product graphs, and in addition, a selection of results appeared in the recent survey \cite{Tillquist-2021}. In connection with these facts, we shall just pay attention here to a few interesting and more recent contributions published after the survey article \cite{Kratica-2014}.

A very significant relationship between the strong metric dimension of a graph $G$ and the vertex cover number of a graph closely related to $G$ was given in \cite{Oellermann-2007}. To see this we need some terminology and notation.

The \emph{strong resolving graph} of $G$, denoted by $G_{SR}$, is a graph that has vertex set
$V(G'_{SR})=V(G)$ and two vertices $u,v$ are adjacent in $G'_{SR}$ if and only if $u$ and $v$ are mutually maximally
distant in $G$. Clearly, if $v\notin \partial(G)$, then $v$ is an isolated vertex in $G_{SR}$. The strong resolving graph of several basic families of graphs can be easily constructed. For instance, $(K_n)_{SR}\cong K_n$; $(C_{2n})_{SR}\cong \bigcup_{i=1}^n K_2$; $(C_{2n+1})_{SR}\cong C_{2n+1}$; $(K_{r,s})_{SR}\cong K_r\cup K_s$; and if $T$ is a tree, then $T_{SR}$ is isomorphic to a graph with a component isomorphic to $K_{|L(T)|}$ and $|V(T)|-|L(T)|$ isolated vertices. As an example, in Figure \ref{Fig:SRG} we have drawn the strong resolving graph of the graph appearing in Figure \ref{Fig:strong-resolving}. We may remark that constructing $G_{SR}$ can be polynomially done. The strong resolving graph $G_{SR}$ turned out to be a remarkable tool while studying the strong metric dimension of graphs based on the following result.

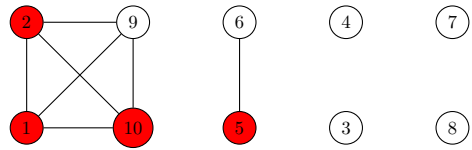
\begin{figure}[ht]
\centering
\begin{tikzpicture}[scale=.7, transform shape]
\node [draw, shape=circle,fill=red] (a1) at  (-4,-1) {1};
\node [draw, shape=circle,fill=red] (a2) at  (-4,1) {2};
\node [draw, shape=circle] (a3) at  (2,-1) {3};
\node [draw, shape=circle] (a4) at  (2,1) {4};
\node [draw, shape=circle,fill=red] (a5) at  (0,-1) {5};
\node [draw, shape=circle] (a6) at  (0,1) {6};
\node [draw, shape=circle,fill=red] (a10) at  (-2,-1) {10};
\node [draw, shape=circle] (a9) at  (-2,1) {9};
\node [draw, shape=circle] (a8) at  (4,-1) {8};
\node [draw, shape=circle] (a7) at  (4,1) {7};

\draw(a1)--(a2)--(a9)--(a10)--(a1)--(a9);
\draw(a2)--(a10);
\draw(a5)--(a6);

\end{tikzpicture}
\caption{The strong resolving graph of the graph in Figure \ref{Fig:strong-resolving}.}\label{Fig:SRG}
\end{figure}

\begin{theorem}{\em \cite{Oellermann-2007}}\label{th:strong-dim-cover}
For any connected graph $G$, $\sdim(G)=\alpha(G_{SR})$.
\end{theorem}

Notice for instance that the set of vertices in red drawn in Figure \ref{Fig:strong-resolving}, which is a strong metric basis of this graph, appears also in red in Figure \ref{Fig:SRG}, and it is vertex cover set of such graph. The result above allowed authors to prove for instance the NP-hardness of the problem of computing the strong metric dimension of graphs, as well as, it allows to claim that finding the strong metric dimension of graphs can be approximated within a constant factor, by using the approximability properties of the vertex cover number. On the other hand, several other conclusions of using this connections can be found in the literature. For more information on the strong resolving graph itself (as a graph transformation independent from the strong metric dimension concept), we suggest the survey \cite{Kuziak-2018} and the work \cite{Lenin-2019}, where an open problem from \cite{Kuziak-2018} was settled. Moreover, another application of the strong resolving graphs of graphs can be found in \cite{Klavzar-2019}.

To prove the NP-completeness of the decision problem related to the strong metric dimension of graphs, in \cite{Oellermann-2007}, the authors proved the following.

\begin{theorem}{\em \cite{Oellermann-2007}}
Let $H$ be any non-complete connected graph. Then there is a graph $G^{(H)}$ such that $H$ is an induced subgraph of $G_{SR}^{(H)}$, and such that $\alpha(G_{SR}^{(H)})=\alpha(H)+ft(H)$, where $ft(H)$ represents the total number of vertices that belong to a non-trivial equivalence class with respect to the false twin\footnote{Two vertices $u,v$ of a graph $G$ are false twins if $N_G(u)=N_G(v)$.} equivalence relation.
\end{theorem}

Clearly, since for any connected graph $H$, it can be polynomially computed $ft(H)$, it follows that $\alpha(G_{SR}^{(H)})$ can be polynomially determined if and only if $\alpha(H)$ can be computed in polynomial time. This is the key of the NP-completeness reduction for deciding whether $\sdim(G)\le k$ for any connected graph $G$ of order $n\ge k+1$.

Although computing the strong metric dimension of graphs was proved to be NP-hard in \cite{Oellermann-2007}, there are some non-trivial families of graphs in which this can be polynomially done. For example, this is the case of trees, which is trivial to deduce. An interesting case regarding this was presented in \cite{Moravcik-2017}, where authors showed that the strong metric dimension of connected split graphs can be found in polynomial time.

Despite the fact that Theorem \ref{th:strong-dim-cover} directly gives an approximation result for the strong metric dimension, some extra work on approximating such parameter was published in \cite{DasGupta-2017}. There was proved, among other results, that the problem of computing the strong metric dimension of a graph of order $n$ admits a $\mathcal{O}(2^{0.287n})$-time exact computation algorithm. Also, it does not admit a polynomial time $2-\epsilon$-approximation algorithm assuming the unique games conjecture is true, as well as, it does not admit a $\mathcal{O}(2^{\mathcal{O}(n)})$-time exact computation algorithm assuming the exponential time hypothesis is true.

It is clear that $\dim(G)\le \sdim(G)$ for any connected graph $G$. In this sense, it is natural to consider the problem of characterizing the graphs attaining an equality in such bound, as well as, in general comparing both parameters for connected graphs. This was the goal of the article \cite{Moravcik-2017} where the following results were presented. Among other things, all trees $T$ for which both invariants achieve the same value were described, thing that happens if and only if $T$ is a path or $T$ has exactly one exterior major vertex. They also determined the class of trees for which the difference between these invariants attains a maximum. In addition, the authors observed that there is no linear upper bound for the strong dimension of trees in terms of the metric dimension. In contrast with this fact, for the case of cographs the following result attracts our attention.

\begin{theorem}{\em \cite{Moravcik-2017}}\label{th:comparing-cograph}
If $G$ is a connected cograph, then $\sdim(G)\le 3\dim(G)$.
\end{theorem}

The bound of Theorem \ref{th:comparing-cograph} has a key point based on the useful result they proved, which states that if $G$ is a connected cograph, then $G_{SR}$ is a cograph as well. As a final contribution of \cite{Moravcik-2017}, regarding comparing $\dim(G)$ and $\sdim(G)$, it was shown that the bound of Theorem \ref{th:comparing-cograph} is indeed asymptotically sharp.

Some interesting Nordhauss-Gaddum results for the strong metric dimension were given in \cite{Yi-2013}. That is, some relationships between the strong metric dimension of a graph and its complement, when both of them are connected. Among other contributions, the next result was proved in \cite{Yi-2013}.

\begin{theorem}{\em \cite{Yi-2013}}
Let $G$ and $\overline{G}$ be connected graphs of order $n\ge 4$. Then $2\le \sdim(G)+sdim(\overline{G}) \le 2(n-2)$. Moreover, $\sdim(G)+\sdim(\overline{G})=2$ if and only if $n=4$. In addition:
\begin{itemize}
  \item If $G$ is a tree and $n\ge 5$, then $3\le \sdim(G)+sdim(\overline{G}) \le 2(n-3)$.
  \item If $G$ is a unicyclic graph and $n\ge 5$, then $4\le \sdim(G)+sdim(\overline{G}) \le 2(n-2)$.
  \item If $G$ is a unicyclic graph and $n\ge 6$, then $4\le \sdim(G)+sdim(\overline{G}) \le 2(n-3)$.
\end{itemize}
\end{theorem}

Characterizations of the families of graphs attaining equality in the bounds given above were also given in \cite{Yi-2013}, which made such work a significant contribution to the topic of Nordhauss-Gaddum results. For some extra information on Nordhauss-Gaddum results in graph theory we suggest the survey \cite{Aouchiche-2013}.

One interesting research line concerning $\sdim(G)$ is connected with the case of product graphs. The goal is to find relationships between the strong metric dimension of a product graph and that of the factors in the product. This kind of investigation allows, among other facts, generating large graphs with prescribed value in the strong metric dimension. Table \ref{tab:strong-dim} summarizes the most remarkable contributions about this. The table include bounds and closed formulae for the strong metric dimension of the four standard products of graphs (as stated in \cite{Hammack-2011}), and the references where they were obtained. The main technique used into obtaining the results of Table \ref{tab:strong-dim} was that of constructing the strong resolving graph of the corresponding product, and then make use of Theorem \ref{th:strong-dim-cover}, which required to compute or bound the vertex cover number of the related strong resolving graph. To better understand the results concerning the structure of the corresponding strong resolving graphs of each case, we suggest \cite{Kuziak-2018}.

\begin{table}[h]
  \centering
  \small{\begin{tabular}{|c|c|c|}
  \hline
  $\mathbf{G*H}$ & $\mathbf{\sdim(G*H)}$ & \textbf{Ref.} \\ \hline \hline
  $G\Box H$ & $=\alpha(G_{SR}\times H_{SR})$ & \cite{Rodriguez-Velazquez-2014-a} \\
  \hline
  $G\Box H$\footnotemark & $=\frac{|V(G)||\partial(H)|}{2}$ & \cite{Rodriguez-Velazquez-2014-a} \\
  \hline
  $G\times K_n$\footnotemark & $=\alpha((G\Box N_n)\sqcup (G_{SR}\circ N_n)\sqcup (N_{|W|}\Box K_n))$\footnotemark & \cite{Kuziak-2017} \\
  \hline
  $P_r\times K_n$ & $=n\left\lceil\frac{r}{2}\right\rceil$ & \cite{Kuziak-2017} \\
  \hline
  $K_{r,t}\times K_n$ & $=n(r+t-1)$ & \cite{Kuziak-2017} \\
  \hline
  $G\boxtimes H$ & $\begin{array}{c}
                                     \ge \max \{|V(H)|\sdim(G),|V(G)|\sdim(H)\} \\ [0.1cm]
                                     \le |V(H)|\sdim(G) + |V(G)|\sdim(H) - \sdim(G)\sdim(H)
                                   \end{array}
  $ & \cite{Kuziak-2015} \\
  \hline
  $G\boxtimes H$ & $= |V(H)|\sdim(G) + |V(G)|\sdim(H) - \sdim(G)\sdim(H)$\footnotemark & \cite{Kuziak-2015} \\
  \hline
  $C_{2r+1}\boxtimes C_{2t+1}$ & $\begin{array}{c}
                                     \ge 3rt + 2r + 2t +1 - \left\lfloor \frac{r}{2} \right\rfloor \\ [0.1cm]
                                     \le 3rt + 2r + 2t +1
                                   \end{array}
  $ & \cite{Kuziak-2015} \\
  \hline
  $G\circ H$\footnotemark & $\begin{array}{c}
                                     =|V(G)|\sdim(H) + |V(H)|\sdim(G)-\dim_s(G)\sdim(H)\footnotemark \\ [0.1cm]
                                     =|V(G)|\sdim(K_1+H) + |V(H)|\sdim(G)-\sdim(G)\sdim(K_1+H)\footnotemark
                                   \end{array}
  $ & \cite{Kuziak-2016} \\
  \hline
  $G\circ K_n$ & $= |V(G)|(n-1)+\sdim(G)$ & \cite{Kuziak-2016} \\
  \hline

\end{tabular}}
  \caption{Bounds and closed formulae for the strong metric dimension of some product graphs.}\label{tab:strong-dim}
\end{table}
\addtocounter{footnote}{-6}
\footnotetext{$G$ be a distance-regular graph and $H$ is a connected graph such that $H_{SR}$ is a regular bipartite graph.}
\addtocounter{footnote}{+1}
\footnotetext{$G$ is a graph containing no pair of mutually maximally distant vertices $u$ and $v$ with $d_G(u,v)=2$.}
\addtocounter{footnote}{+1}
\footnotetext{$W$ is the subset of $V(G)$ which contains all vertices belonging to a triangle in $G$ and $N_{|W|}$ is an edgeless graph with vertex set $W$.}
\addtocounter{footnote}{+1}
\footnotetext{If the vertex set of $G$ can be partitioned into $\beta(G)$ cliques.}
\addtocounter{footnote}{+1}
\footnotetext{$G$ has no true twin vertices.}
\addtocounter{footnote}{+1}
\footnotetext{$G$ has diameter at most 2.}
\addtocounter{footnote}{+1}
\footnotetext{$G$ has diameter larger than 2. $K_1+H$ is the \emph{joint graph} of $K_1$ and $H$.}

We might recall that results of Table \ref{tab:strong-dim} are just the most general ones. By using them, a large number of consequences for the strong metric dimension of product graphs involving particular families of graphs were deduced, that is, products of cycles, paths, complete graphs, complete bipartite graphs, trees, or diameter two graphs, etc. On the other hand, several non standard products have been dealt with in connection with its strong metric dimension. These are the cases of Cartesian sum graphs (\cite{Kuziak-2015-a}); corona graphs and join graphs (\cite{Kuziak-2013}); and rooted product graphs (\cite{Kuziak-2016-a}).

Nearly to the case of product graphs, there exist some graphs operations that are of interest, and they have not escaped of being investigated about its strong metric dimension. A nice case is found in the generalized Sierpi\'nski graphs $S(G,t)$ of a graph $G$. To see more information on the definition and combinatorial properties of Sierpi\'nski graphs we suggest for instance the survey \cite{Hinz-2017} and the work \cite{Gravier-2011} where the generalized version was first introduced. In this direction, the following result from \cite{Estaji-2016} is of interest.

\begin{theorem}{\em \cite{Estaji-2016}}
Let $G$ be a connected graph of order $n$ having $k$ leaves and let $t$ be a positive integer. If every non-leaf vertex of $G$ is a cut vertex, then
$$\sdim(S(G,t))=\frac{k(n^t-2n^{t-1}+1)-n+1}{n-1}.$$
\end{theorem}

The key of the proof of the theorem above relies on the fact that if every internal vertex of $G$ is a cut vertex, then every internal vertex of $S(G, t)$ is a cut vertex (in $S(G, t)$) as well, for any $t\ge 2$. Moreover, by making for example a count on the number of leaves of $S(G, t)$, and using Theorem \ref{th:strong-dim-cover}, one can deduce the formula.

In other direction, the strong metric dimension of the power graph of a finite group was studied in \cite{Ma-2018}. Such study allowed the authors to compute the strong metric dimension of the power graph of some algebraic structures like a cyclic group, an abelian group, a dihedral group and a generalized quaternion group.

We next mention some other classes of graphs for which its strong metric dimension has been computed in the last recent years. This is summarized in Table \ref{tab:strong-extra}

\begin{table}[h]
  \centering
  \begin{tabular}{|c|c|c|c|c|c|}
    \hline
    \textbf{Graph family} & \textbf{Ref.} & \textbf{Graph family} & \textbf{Ref.} & \textbf{Graph family} & \textbf{Ref.} \\ \hline
    Wheel related & \cite{Kusmayadi-2016} & Cozero-divisor graphs & \cite{Nikandish-2021} & Generalized Petersen & \cite{Kratica-2017} \\
    graphs &  & &  &  graphs &  \\ \hline
    Zero-divisor & \cite{Bhat-2019} & Cactus graphs & \cite{Kuziak-2020-a} & Annihilator graphs & \cite{Ebrahimi-2021} \\
    graphs of rings &  &  &  & of commutative rings &  \\
    \hline
  \end{tabular}
  \caption{Some other studies on the strong metric dimension of graphs.}\label{tab:strong-extra}
\end{table}

\subsection{Some open problems}

\begin{itemize}
  \item Since finding the vertex cover number of bipartite graphs can be polynomially done, in view of Theorem \ref{th:strong-dim-cover}, it would be worthwhile to describe all graphs $G$ for which $G_{SR}$ is bipartite.
  \item In view of the bound $\dim(G)\le \sdim(G)$, can we characterize graphs $G$ such that $\dim(G)=\sdim(G)$?
  \item In view of \cite{Moravcik-2017}, where the authors observed there is no linear upper bound for the strong dimension of trees in terms of the metric dimension, can we find some families of graphs $G$, other than cographs, where $\sdim(G)$ is bounded by a constant factor of $\dim(G)$?
\end{itemize} 

%% file: local-dim.tex
\section{Local metric dimension}\label{sec:local}

A set of vertices $S$ of a connected graph $G$ is a \emph{local resolving set} for $G$ if for any two adjacent vertices $u,v\in V(G)$, there is a vertex $x\in S$ such that $d_G(u,x)\ne d_G(v,x)$. In such case, we also say that $x$ locally resolves (identifies or recognizes) $u,v$. A local resolving set of the smallest possible cardinality is a \emph{local metric basis}, and its cardinality is the \emph{local metric dimension} of $G$, denoted by $\dim_{\ell}(G)$. Concepts above were first introduced in \cite{Okamoto-2010} where the notation for $\dim_{\ell}(G)$ was $\mathrm{lmd}(G)$. In order to be consequent with all the notations of this exposition, we propose and use $\dim_{\ell}(G)$. In Figure \ref{Fig:local-res}, two different local metric bases are drawn.

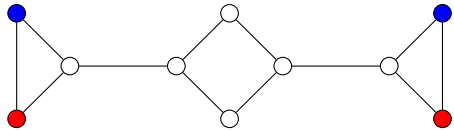
\begin{figure}[ht]
\centering
\begin{tikzpicture}[scale=.7, transform shape]
\node [draw, shape=circle,fill=red] (a1) at  (-4,-1) {};
\node [draw, shape=circle,fill=blue] (a2) at  (-4,1) {};
\node [draw, shape=circle] (a3) at  (-3,0) {};
\node [draw, shape=circle] (a4) at  (-1,0) {};
\node [draw, shape=circle] (a5) at  (0,-1) {};
\node [draw, shape=circle] (a6) at  (0,1) {};
\node [draw, shape=circle,fill=red] (a10) at  (4,-1) {};
\node [draw, shape=circle,fill=blue] (a9) at  (4,1) {};
\node [draw, shape=circle] (a8) at  (3,0) {};
\node [draw, shape=circle] (a7) at  (1,0) {};

\draw(a3)--(a1)--(a2)--(a3)--(a4)--(a5)--(a7)--(a8)--(a9)--(a10)--(a8);
\draw(a4)--(a6)--(a7);
\end{tikzpicture}
\caption{Two local resolving sets of the smallest possible cardinality appear in red and in blue.}\label{Fig:local-res}
\end{figure}

We might notice that this parameter indeed does not require the connectedness of the studied graphs, since only pairs of vertices being adjacent need to be resolved. However, it is not difficult to observe the following result.

\begin{remark}\label{rem_local-not-connected}
Let $G_1,\dots,G_r$, $r\ge 2$, be the subgraphs induced by the components of a disconnected graph $G$. Then $\dim_{\ell}(G)=\sum_{i=1}^{r}\dim_{\ell}(G_i)$.
\end{remark}

In connection with this, we consider only connected graphs as usually made with metric dimension related parameters in graphs. Variations of such parameter are also known in the literature. Among them, we can find for instance the local fractional metric dimension \cite{Aisyah-2019}, the adjacency local metric dimension \cite{Fernau-2014,Fernau-2018}, and the simultaneous local metric dimension \cite{Barragan-Ramirez-2017}. For some more specific information on several results on the local metric dimension of graphs we suggest the Ph. D. dissertation \cite{Barragan-Ramirez-phd}.

We first want to remark an interesting ``application'' of local metric dimension in delivery services that has recently appeared in \cite{Klavzar-2020}. In such work, authors assumed that a given company could need to assign codes to its customers such that the code of a customer will uniquely determine its location. In some sense, it seems to be natural that the company would be interested into making the length of the codes as short as possible. In order to design a graph theory model for this problem, in \cite{Klavzar-2020}, customers are considered as the vertices of a (an edge-weighted) graph $G$. Vertices $u$ and $v$ are declared to be adjacent in $G$ if one of the following conditions is fulfilled: (i) there is no other customer on the $u, v$-geodesics; (ii) the first letters of the family names of the customers $u$ and $v$ are the same. Next, let $S=\{v_1,\dots,v_r\}$ be a local metric basis for $G$. Then, if $F$ is the first letter of the family name of a given customer $v$, then the company allocates the ordered pair $(F,(d_G(v,v_1),\dots,d_G(v,v_r)))$ to the customer $v$ as its code. Note that the second part of such pair is precisely $r(v|S)$, \emph{i.e.}, the metric representation of $v$ with respect to $S$ . The authors of \cite{Klavzar-2020} made a comparison of such model with a related one from \cite{Khuller-1996} (which uses the classical metric dimension), and found out that the local metric dimension model in general behaves much better, and for instance, in cases like bipartite graphs (where the local metric dimension equals 1, and the classical metric dimension could be much larger) the efficiency of the new model is much higher.

In \cite{Salman-2014}, an integer linear programming model for the local metric dimension of graphs was presented. Such formulations follow in general the ideas already known for other related invariants like for instance the classical metric dimension.

Concerning computational aspects of this parameter, it was shown in \cite{Fernau-2014,Fernau-2018} that the decision problem regarding computing the local metric dimension of a graph is NP-complete. Moreover, there was also proved that, assuming ETH, there is no $\mathcal{O}(pol(n+m)2^{o(n)})$-algorithm for solving such problem on graphs of order $n$ and size $m$. In such work, several other computational aspects concerning the NP-completeness of the decision problems related to the adjacency, and the adjacency local metric dimension of graphs were given, even when restricted to planar graphs. However, their reductions were not useful for proving the NP-completeness for the local metric dimension problem of planar graphs. In this sense, authors left an open question on whether finding the local metric dimension of planar graphs is NP-hard. The main ideas behind their proofs are based on some relationships that are given between the studied metric parameters in the corona and strong products of graphs and that of their factors.

With respect to combinatorial results on the local metric dimension, we first immediately notice that if $S$ is an independent set of $G$ of order $n$, then the set $V(G)\setminus S$ is a local resolving set, and so,
\begin{equation}\label{eq:ldim-indep}
  \dim_{\ell}(G)\le n-\beta(G).
\end{equation}
In this sense, it would be interesting to characterize the graphs $G$ for which $\dim_{\ell}(G)= n-\beta(G)$. Clearly, any resolving set is also a local resolving set, and trivial bounds for the local metric dimension are 1 and $n-1$ as well as for the classical metric dimension. That is, for any connected graph $G$,
$$1\le \dim_{\ell}(G)\le \dim(G)\le n-1.$$
Characterizations of the graphs achieving the limit values in the bounds above were given in \cite{Okamoto-2010} as follows. If $G$ is a nontrivial connected graph of order $n$, then $\dim_{\ell}(G) = n-1$ if and only if $G$ is the complete graph $K_n$; and $\dim_{\ell}(G) = 1$ if and only if $G$ is a bipartite graph. This latter result relies on the following property of a bipartite graph $G$. For any vertex $w\in V(G)$, if we consider an edge $uv\in E(G)$, then (w.l.g.) the distance $d_G(u,w)$ is even and the distance $d_G(v,w)$ is odd.  On the other hand, a related characterization for graphs $G$ with $\dim_{\ell}(G) = n-2$ was given in connection with the clique number $\omega(G)$ of $G$ as follows.

\begin{theorem}\emph{\cite{Okamoto-2010}}
A connected graph $G$ of order $n>3$ has local metric dimension $n-2$ if and only if $\omega(G) = n-1$.
\end{theorem}

In addition, the realization of graphs with a given value in the local metric dimension was given in \cite{Okamoto-2010}. That is, for each pair $k, n$ of integers with $1 \le k \le n-1$, there exists a connected graph $G$ of order $n$ with $\dim_{\ell}(G) = k$. Moreover, in connection with the bound in \eqref{eq:ldim-indep}, it was proved in \cite{Okamoto-2010} that for each pair $a, n$ of integers with $1\le a \le n - 1$, there exists
a nontrivial connected graph $G$ of order $n$ and independence number $a$ such that $\dim_{\ell}(G) = n - a$. Some other realizations related to diameter and true twin equivalence classes were also proved in \cite{Okamoto-2010}.

Based on the NP-hardness of finding the local metric dimension of graphs, it is clearly desirable to bound such parameter for general graphs or some specific families, as well as computing its exact value in some particular situations. These two research lines have centered several investigations on this parameter in the same manner as other metric dimension topics. Next two tables summarize those (in our opinion) more interesting results in these directions.

\begin{table}[ht]
  \centering
  \small{\begin{tabular}{|c|c|c|}
    \hline
    \textbf{Graphs} $\mathbf{G}$ & $\mathbf{\dim_{\ell}(G)}$ & \textbf{Reference} \\ \hline
    Graph $G$, clique number $\omega(G)$  & $\ge \log_2\omega(G)$ & \cite{Okamoto-2010} \\ \hline
    Graph $G$, clique number $\omega(G)$ & $\ge |V(G)|-2^{|V(G)|-\omega(G)}$ & \cite{Okamoto-2010} \\ \hline
    Generalized hierarchical products & Some bounds & \cite{Klavzar-2020} \\ \hline
    Graph $G[\mathcal{H}]$,\footnotemark $H=\{H_1,\dots,H_r\}$ & Some bounds & \cite{Rodriguez-Velazquez-2015} \\ \hline
    $\begin{array}{c}
      \mbox{Strong product graphs $G\boxtimes H$} \\
      |V(G)|=r, V(H)=t
    \end{array}$
     & $\begin{array}{c}
         \ge 3 \\
         \le r\dim_{\ell}(H)+t\dim_{\ell}(G)-\dim_{\ell}(G)\dim_{\ell}(H)
       \end{array}$
      & \cite{Barragan-Ramirez-2016} \\ \hline
      $\begin{array}{c}
      \mbox{Strong product graphs $G\boxtimes H$} \\
      \mbox{H, adjacency $k$-resolved graph\footnotemark}\\
      D(G)<k
    \end{array}$
     & $\le |V(H)|\dim_{\ell}(G)$
      & \cite{Barragan-Ramirez-2016} \\ \hline
      $\begin{array}{c}
      \mbox{Strong product graphs $G\boxtimes P_t$} \\
      t\ge 2D(G)+1
    \end{array}$
     & $\ge \left\lceil\frac{t-1}{D(G)}\right\rceil+1$ & \cite{Barragan-Ramirez-2016} \\ \hline
  \end{tabular}}
  \caption{Bounds for the local metric dimension of graphs.}\label{tab:bounds-ldim}
\end{table}
\addtocounter{footnote}{-1}
\footnotetext{$G[\mathcal{H}]$ is a graph obtained by ``point attaching''. Examples of such graphs can be for instance the rooted product graph.}
\addtocounter{footnote}{+1}
\footnotetext{See \cite{Barragan-Ramirez-2016} for this definition.}

\begin{table}[ht]
  \centering
  \small{\begin{tabular}{|c|c|c|}
    \hline
    \textbf{Graphs} $\mathbf{G}$ & $\mathbf{\dim_{\ell}(G)}$ & \textbf{Reference} \\ \hline
    Bipartite graph $G$ & $1$ & \cite{Okamoto-2010} \\ \hline
    Cartesian product graphs $G\Box H$ & $\max\{\dim_{\ell}(G),\dim_{\ell}(H)\}$ & \cite{Okamoto-2010} \\ \hline
    Corona graphs $G\odot H$, $H$ an empty graph & $\dim_{\ell}(G)$ & \cite{Barragan-Ramirez-2014,Rodriguez-Velazquez-2016} \\ \hline
    Corona graphs $G\odot H$, $G$ order n, $H$ not empty\footnotemark & $n\dim_{\ell}(K_1+H)$ & \cite{Barragan-Ramirez-2014,Rodriguez-Velazquez-2016} \\ \hline
    Corona graphs $G\odot H$, $G$ order n, $H$ not empty\footnotemark & $n(\dim_{\ell}(K_1+H-1))$ & \cite{Barragan-Ramirez-2014,Rodriguez-Velazquez-2016} \\ \hline
    Graph $G[\mathcal{H}]$, $H=\{H_1,\dots,H_r\}$ & Some formulas & \cite{Rodriguez-Velazquez-2015} \\ \hline
    Rooted product graphs $G\circ_v H$ & Some formulas & \cite{Susilowati-2015} \\ \hline
    $\begin{array}{c}
      \mbox{Strong product graphs $G\boxtimes P_t$} \\
      \mbox{$G$ connected bipartite }, t\ge 2D(G)+1
    \end{array}$
     & $\left\lceil\frac{t-1}{D(G)}\right\rceil+1$\footnotemark & \cite{Barragan-Ramirez-2016} \\ \hline
    Edge-corona of graphs  & Some formulas & \cite{Suprajitno-2016} \\ \hline
    Some convex polytopes  & Some formulas & \cite{Salman-2014} \\ \hline
    Generalized hierarchical products & Some formulas & \cite{Klavzar-2020} \\ \hline
    Molecular graph $\Gamma_{n,k}$ & $2$ & \cite{Klavzar-2020} \\ \hline
    Lexicographic product graphs & Some formulas & \cite{Barragan-Ramirez-2019} \\ \hline
  \end{tabular}}
  \caption{Exact values of the local metric dimension of some graphs.}\label{tab:formulas-ldim}
\end{table}
\addtocounter{footnote}{-2}
\footnotetext{The vertex of $K_1$ does not belong to any local metric basis for $K_1 + H$.}
\addtocounter{footnote}{+1}
\footnotetext{The vertex of $K_1$ belongs to a local metric basis for $K_1 + H$.}
\addtocounter{footnote}{+1}
\footnotetext{This result allowed to settle a conjecture stated in \cite{Rodriguez-Velazquez-2015-a} concerning the metric dimension of the strong product of two paths $P_r\boxtimes P_t$.}

\subsection{Some open problems}

\begin{itemize}
  \item Characterize the graphs $G$ for which $\dim_{\ell}(G)= n-\beta(G)$.
  \item Is it the case that computing the local metric dimension of planar graphs is NP-hard?
  \item Since the local metric dimension can be understood for non-connected graphs, in contrast with other metric dimension parameters, studying the direct product of graphs (which is frequently not connected) for this parameter seems to be worthwhile.
  \item Characterize the family of graphs $G$ for which $\dim_{\ell}(G)=\dim(G)$.
  \item Since $\dim_{\ell}(G)\le \dim(G)$ for any connected graph $G$, is it possible that $\dim(G)$ will be above by some constant factor of $\dim_{\ell}(G)$.
  \item Is there any relationship between $\dim_{\ell}(G)$ and other metric dimension related parameters of graphs, other than $\dim(G)$?
\end{itemize}

%% file: adjacency.tex
\section{Adjacency dimension}\label{sect:adjacency}

The adjacency dimension of graphs was introduced in \cite{Jannesari-2012} as a tool to study the metric dimension of the lexicographic product of graphs. The necessity of such parameter is based on the following. If we consider the lexicographic product $G\circ H$, then we can notice that any two vertices $(g,h_1),(g,h_2)$ belonging to a same copy of $H$ have the same distance to any other vertex not belonging to the same copy which  $(g,h_1),(g,h_2)$ belong. Moreover, these vertices also have the same distance to every vertex $(g,h)$ in the same copy which they belong, and for which $d_H(h,h_1)\ge 2$ and $d_H(h,h_2)\ge 2$. In this sense, if we want to resolve the pair of vertices $(g,h_1),(g,h_2)$, then we needs to use those vertices $(g,h')$ such that $h'$ is neighbor of either $h_1$ or $h_2$ in $H$. This means that, in order to uniquely recognize all the vertices of $G\circ H$, in each copy of $H$ in $G\circ H$, one needs a set $S$ of vertices that will identifies the vertices of $H$ by means of different neighborhoods in $S$.

With these ideas in mind, authors of \cite{Jannesari-2012} defined the concept of adjacency resolving sets and adjacency metric dimension of graphs and used them into giving some contributions for the metric dimension of the lexicographic product of graphs. After this seminal paper, the parameter became interesting by itself and several other contributions on it have appeared. Similarly to the case of the local metric dimension, this parameter is not influenced by the connectivity of the studied graphs. However, again it can be easily noted an analogous result to Remark \ref{rem_local-not-connected}. In consequence, we shall center our attention in this section on connected graphs, unless we will specifically state the contrary.

A set $S\subset V(G)$ of vertices is said to be an \emph{adjacency resolving set} for $G$ if for every two vertices $x,y\in V\setminus S$ there exists $s\in S$ such that $s$ is adjacent to exactly one of $x$ and $y$. An adjacency resolving set of minimum cardinality is called an \emph{adjacency basis} of $G$, and its cardinality  the \emph{adjacency dimension} of $G$, denoted by $\dim_A(G)$. These concepts are initially introduced in \cite{Jannesari-2012} by using the adjacency representation version (the metric representation modified for this situation) of a vertex $v$; namely $r_A(v|S)$, where the distance between $v$ and any vertex $v_i\in S$ is defined as follows: $$d_G(v,v_i)=\left\{\begin{array}{ll}
                                                                          0, & \mbox{if $v=v_i$}, \\
                                                                          1, & \mbox{if $vv_i\in E(G)$} \\
                                                                          2, & \mbox{otherwise.}
                                                                        \end{array}
\right.$$
In Figure \ref{Fig:adjacency} appears an example of a graph where an adjacency metric basis is drawn in red color. 

\begin{figure}[ht]
\centering
\begin{tikzpicture}[scale=.7, transform shape]
\node [draw, shape=circle,fill=red] (a1) at  (-4,-1) {};
\node [draw, shape=circle,fill=red] (a2) at  (-4,1) {};
\node [draw, shape=circle] (a3) at  (-3,0) {};
\node [draw, shape=circle] (a4) at  (-1,0) {};
\node [draw, shape=circle,fill=red] (a5) at  (0,-1) {};
\node [draw, shape=circle] (a6) at  (0,1) {};
\node [draw, shape=circle,fill=red] (a10) at  (4,-1) {};
\node [draw, shape=circle] (a9) at  (4,1) {};
\node [draw, shape=circle] (a8) at  (3,0) {};
\node [draw, shape=circle,fill=red] (a7) at  (1,0) {};

\draw(a3)--(a1)--(a2)--(a3)--(a4)--(a5)--(a7)--(a8)--(a9)--(a10)--(a8);
\draw(a4)--(a6)--(a7);
\end{tikzpicture}
\caption{An adjacency resolving set of the smallest possible cardinality appears in red.}\label{Fig:adjacency}
\end{figure}
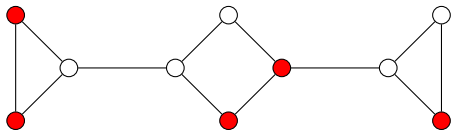

We may remark also that the notion of adjacency resolving sets has an antecedent in \cite{Babai-1980}, were it was studied this idea of adjacency resolving sets for strongly regular graphs, but in the context of the graph isomorphism problem. There, such structures were called distinguishing sets. Moreover, a few other very similar parameters are known in the literature, like for instance, the locating dominating sets \cite{Colbourn-1987}. Moreover, several variants of the adjacency dimension are also known like for instance the local adjacency dimension \cite{Fernau-2018}, the simultaneous adjacency dimension (of graph families) \cite{Ramirez-Cruz-2016-a}, and the $k$-adjacency dimension \cite{Estrada-Moreno-2016-b}. This latter variant has been indeed very well studied in a few articles. To know more on such version, we suggest the Ph. D. dissertation \cite{Estrada-Moreno-PhD}.

It is natural to think that adjacency dimension is related to the classical metric dimension, and it can be easily noted that for any connected graph $G$, $\dim(G)\le \dim_A(G)$. In this sense, characterizing the graphs achieving such equality seems to be a good research line.

One of the first interesting contributions on this topic is noticing that the adjacency dimension of a graph $G$ and that of its complement $\overline{G}$ are indeed equal. This is based on the fact that, if $S$ is any set of vertices of $G$ and $u,v\in V(G)$ satisfy that $N_G(u)\cap S\ne N_G(v)\cap S$, then also $N_{\overline{G}}(u)\cap S\ne N_{\overline{G}}(v)\cap S$. This was a result presented in \cite{Jannesari-2012}.

\begin{theorem}\emph{\cite{Jannesari-2012}}
For every graph $G$, $\dim_A(G)=\dim_A(\overline{G})$.
\end{theorem}

Similarly to other metric dimension parameters, it also happens that $1\le\dim_A(G)\le n-1$ for any graph $G$. Characterizations of graphs achieving the equality in this trivial bounds were given in \cite{Jannesari-2012} as well. In addition to this, all graphs with adjacency dimension 2, and all graphs of order $n$ with adjacency dimension $n-2$ are studied in \cite{Jannesari-2021}.

Computational issues concerning the adjacency dimension of graphs were presented in \cite{Fernau-2018}, where it was proved that computing the adjacency dimension of graphs is NP-hard even when restricted to planar graphs. In addition, an inapproximation result was also given there, which states that, assuming ETH, there is no $\mathcal{O}(pol(n + m)2^{o(n)})$ algorithm solving the decision problem regarding computing the adjacency dimension of graphs, on graphs of order $n$ and size $m$. The reductions used there are similar to that ones we mention in Section \ref{sec:local} for the case of local metric dimension. Some other computational contributions on this parameter were also given in \cite{Fernau-2018}.

With respect to combinatorial results on this parameter, some bounds and exact values for the adjacency dimension of some families of graphs are summarized in Table \ref{tab:adjacency}.

\begin{table}[ht]
  \centering
  \small{\begin{tabular}{|c|c|c|}
    \hline
    \textbf{Graphs} $\mathbf{G}$ & $\mathbf{\dim_A(G)}$ & \textbf{Reference} \\ \hline
    Cycle $C_n$ and path $P_n$ & $=\left\lfloor\frac{2n+2}{5}\right\rfloor$ & \cite{Jannesari-2012} \\ \hline
    Corona graph $G\odot H$ & Some bounds and formulas & \cite{Estrada-Moreno-2016-c,Estrada-Moreno-2016-e} \\ \hline
    Lexicographic product graph $G\circ H$ & Some bounds and formulas & \cite{Estrada-Moreno-2016-b} \\ \hline
    Graph $G$, order $n$, diameter $D$ & $\le n-D-1 + \left\lfloor\frac{2D + 4}{5}\right\rfloor$ & \cite{Jannesari-2021} \\ \hline
  \end{tabular}}
  \caption{Bounds and formulas for the adjacency dimension of some families of graphs.}\label{tab:adjacency}
\end{table}

\subsection{Some open problems}

\begin{itemize}
  \item Characterizing the graphs $G$ for which $\dim(G)=\dim_A(G)$.
  \item Studying the adjacency dimension of Cartesian, strong and direct products of graphs.
  \item Which is the adjacency dimension of grid graphs?
  \item Since $\dim(G)\le \dim_A(G)$ for any graph $G$, can $\dim_A(G)$ be bounded above by a constant factor of $\dim(G)$?
\end{itemize} 

%% file: k-dim.tex
\section{$k$-metric dimension}

The extension of the classical metric dimension of graphs to the $k$-metric dimension was independently introduced in \cite{Adar-2017} and \cite{Estrada-Moreno-2015} (although the latter work was much earlier known), as an attempt of improving the weakness of existing unique vertices in a metric basis that are uniquely recognizing some vertices of the graph. The idea had been previously partially taken into account in the fault tolerant metric dimension of graphs (see \cite{Hernando-2008}), which is indeed considering $k=2$ in the extended version of $k$-metric dimension. However, nothing more had been made so far on this issue. On the other hand, this concept might not be confused with that one, also called $k$-metric dimension, introduced in \cite{Sooryanarayana-2016}, which is in fact a very different parameter.

A set $S\subseteq V(G)$ is said to be a \emph{$k$-resolving set} for a connected graph $G$ if and only if any pair of vertices of $G$ is distinguished by at least $k$ elements of $S$. That is, for any pair of different vertices $u,v\in V(G)$, there are at least $k$ vertices $w_1,w_2,\ldots,w_k\in S$ such that $d_G(u,w_i)\ne d_G(v,w_i)$ for every $i\in \{1,\ldots,k\}$. A $k$-resolving set of minimum cardinality in $G$ is  called a \emph{$k$-metric basis} and its cardinality the \emph{$k$-metric dimension} of $G$, which is denoted by $\dim_{k}(G)$. If $S$ is a $k$-metric basis, then the metric representations of each pair of distinct vertices differ in at least $k$ positions, and moreover, there must exist at least one pair of vertices whose metric representations differ in exactly $k$ positions. The notation is from \cite{Estrada-Moreno-2015}, where indeed $k$-resolving sets were called $k$-metric generators. Clearly, $1$-resolving sets form the standard resolving sets.

As an example, Figure \ref{Fig:k-dim} shows a grid graph and examples of $t$-metric bases for $t\in\{1,\dots,7\}$ are given by the sets $S=\{1,\dots,t\}\cup\{8,\dots,7+t\}$. Notice that each of these $t$-metric bases is having cardinality $2t$, result that was proved in \cite{Bailey-2019}.

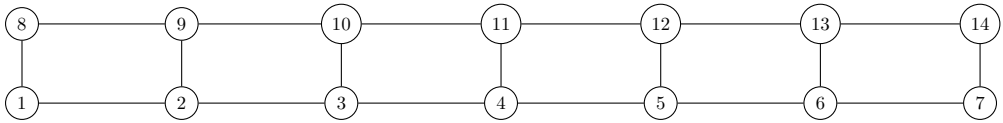
\begin{figure}[h]
\centering
\begin{tikzpicture}[scale=.7, transform shape]
\node [draw, shape=circle] (a1) at  (0,0) {1};
\node [draw, shape=circle] (a11) at  (0,1.5) {8};
\node [draw, shape=circle] (a2) at  (3,0) {2};
\node [draw, shape=circle] (a22) at  (3,1.5) {9};
\node [draw, shape=circle] (a3) at  (6,0) {3};
\node [draw, shape=circle] (a33) at  (6,1.5) {10};
\node [draw, shape=circle] (a4) at  (9,0) {4};
\node [draw, shape=circle] (a44) at  (9,1.5) {11};
\node [draw, shape=circle] (a5) at  (12,0) {5};
\node [draw, shape=circle] (a55) at  (12,1.5) {12};
\node [draw, shape=circle] (a6) at  (15,0) {6};
\node [draw, shape=circle] (a66) at  (15,1.5) {13};
\node [draw, shape=circle] (a7) at  (18,0) {7};
\node [draw, shape=circle] (a77) at  (18,1.5) {14};

\draw(a1)--(a2)--(a3)--(a4)--(a5)--(a6)--(a7);
\draw(a11)--(a22)--(a33)--(a44)--(a55)--(a66)--(a77);
\draw(a1)--(a11);\draw(a2)--(a22);\draw(a3)--(a33);\draw(a4)--(a44);
\draw(a5)--(a55);\draw(a6)--(a66);\draw(a7)--(a77);
\end{tikzpicture}
\caption{The set $S=\{1,\dots,t\}\cup\{8,\dots,7+t\}$, with $1\le t\le 7$ is a $t$-resolving set of the grid graph $P_7\Box P_2$.}\label{Fig:k-dim}
\end{figure}

It must be mentioned that \cite{Adar-2017} introduced two approaches of the concepts above in the following sense. In one hand, there is one approach which is in fact the same as defined above (the same as from \cite{Estrada-Moreno-2015}), which they called the all-pairs model (AP). In a second hand, in \cite{Adar-2017}, authors considered the idea of uniquely identifying not the whole set of vertices of the graph, but only the vertices outside the resolving set in question, and they called this, the non-landarks model (NL). These facts clearly influence much on the conclusions we get for the corresponding $k$-metric dimension problem. We want to also recall that \cite{Adar-2017} contains some results dealing with a weighted version of the problem. However, we shall not include anything about this, since it is not our goal to present any results concerning weighted graphs.

We may notice that for every $k$-resolving set $S$ it follows $|S|\geq k$. Also, if $k>1$, then $S$ is also a $(k-1)$-resolving set. One first observation (in the AP model) is as follows. If $k=1$, then the problem of checking if a given set $S$ is a resolving set is confirmed by only checking whether the vertices $u,v\in V(G)\setminus S$ are uniquely determined by $S$, since every vertex in $S$ is already distinguished by itself. Moreover, if $k=2$, then one needs to only check a similar fact only for those pairs having at most one vertex in $S$, since two vertices of $S$ are already distinguished by themselves. However, if $k\ge 3$, then we are required to check every pair of different vertices of the graph. These facts make much differences while dealing with the $k$-metric dimension of graphs, for the cases $k=1$, $k=2$ and $k\ge 3$. This is also related with the existence of the two approaches presented in \cite{Adar-2017}. From now on, unless specifically stated, whenever we would deal with the $k$-metric dimension problem, it must be understood we are dealing with the AP model.

One first observation on the $k$-metric dimension of graphs is that, there is an upper limit for $k$, making that a given graph does not contain $k'$-resolving sets for every $k'>k$. In this sense, a graph $G$ is called $k$-\emph{metric dimensional}, if $k$ is the largest integer for which $G$ contains a $k$-resolving set. It was shown in \cite{Estrada-Moreno-2015} the following result concerning this fact.

\begin{theorem}{\em \cite{Estrada-Moreno-2015}}
A connected graph  $G$ is $k$-metric dimensional  if and only if $$k=\min\{\vert R_G\{x, y\}\vert\,:\,x,y\in V(G),\; x\ne y\}.$$
\end{theorem}

Based on this result, and taking into account that every pair of distinct vertices is recognized at least by the vertices of the pair, it is clear that every graph $G$ is $k$-metric dimensional for some $k\ge 2$. This immediately opens the question of characterizing the family of all $2$-metric dimensional graphs, as well as, the problem of finding the integer $k$ for which a given graph is $k$-metric dimensional.

Regarding the first problem, it was shown in \cite{Estrada-Moreno-2015} that a connected graph $G$ of order $n\geq 2$ is $2$-metric dimensional if and only if $G$ has twin vertices (for instance, a tree with two leaves having a common neighbor). Concerning the latter problem, from \cite{Yero-2017}, it is known that finding the value $k$ for which a given graph $G$ is $k$-metric dimensional can be efficiently (polynomially) done, and that has order $\mathcal{O}(n^3)$, where $n$ is the order of the graph. Despite this fact, some basic situations can be easily deduced. For instance, it can be readily seen that only  whether $n=2$ it is possible to see a graph that is $n$-metric dimensional (\emph{i.e.}, the graph $K_2$). In consequence, it might be of interest to know those graphs that are $(n-1)$-metric dimensional. This was done in \cite{Estrada-Moreno-2015}.

\begin{theorem}\emph{\cite{Estrada-Moreno-2015}}
A graph $G$ of order $n\ge3$ is $(n-1)$-metric dimensional if and only if $G$ is a path or $G$ is an odd cycle.
\end{theorem}

In order to complete the case of cycles graphs, it is known from \cite{Estrada-Moreno-2015} that any even cycle of order $n$ is $(n-2)$-metric dimensional. Another case, where the time complexity can be lowered is that of trees. In this sense, the following result was given in \cite{Yero-2017}, together with a corresponding algorithm.

\begin{theorem}{\em \cite{Yero-2017}}
The positive integer $k$ for which a tree different from a path is $k$-metric dimensional can be computed in linear time with respect to the order of the tree.
\end{theorem}

In order to also slightly improve the time complexity $\mathcal{O}(n^3)$ of finding the integer $k$ for which a given graph is $k$-metric dimensional, some bounds on such value were given in \cite{Estrada-Moreno-2015}, and further on, in \cite{Corregidor-2021}. Such bounds were given in terms of different parameters or invariants of the graph such as the minimum and maximum degrees, the girth, the diameter, the clique number, and others. In special, the article \cite{Corregidor-2021} gives some bounds which are improving some ones from \cite{Estrada-Moreno-2015}.

\subsection{Computational aspects}

With respect to the computational complexity of the problem of computing the $k$-metric dimension of graphs, it was published a result in \cite{Yero-2017} where was claimed that such problem is NP-hard. However, such result has a significant gap in the proof, which is ``rather impossible'' to avoid. Fortunately, the recent investigation \cite{Schmitz-2021+} (still a manuscript) presents a proof of such result that seems to be correct. The technique (which turned out to be not applicable) applied in \cite{Yero-2017} was highly inspired in the proof from \cite{Khuller-1996}, which uses a reduction from the 3-satisfiability problem (3-SAT). In contrast, \cite{Schmitz-2021+} presents a different technique, while proving a result that go even further.

\begin{theorem}{\em \cite{Schmitz-2021+}}
For any $k\ge 2$, the decision problem regarding computing the $k$-metric dimension ($k$-MD problem for short) of graphs is NP-complete even when restricted to bipartite graphs.
\end{theorem}

To proof the result above the authors are first required to prove the NP-completeness of another problem called the 3-Dimensional $k$-Matching problem, which is stated as follows: Given a set $S\subseteq A\times B\times C$, where $A$, $B$ and $C$ are disjoint sets
of the same size $n$: Does $S$ contain a $k$-matching, {\em i.e.} a subset $M$ of size $kn$, such that each element of $A$, $B$ and $C$ is contained in exactly $k$ triples of $M$? Once proved the NP-completeness of such result, a reduction from the 3-Dimensional $(k-1)$-Matching problem to the $k$-MD problem for bipartite graphs is carried out.

It must be also said that, a partial proof of the NP-completeness of the $k$-MD problem is also known from \cite{Estrada-Moreno-2021+}, where (in a more general setting) it is shown that $k$-MD problem is NP-complete, whenever $k$ is an odd integer. The technique used here is also rather different from the traditional ones, due to the most general setting which is studied in such work.

Despite the NP-hardness of computing the $k$-metric dimension of graphs, there are non trivial families of graphs where such parameter can be efficiently computed. This is the case of a tree for instance. In \cite{Yero-2017} was presented a polynomial algorithms that finds the $k$-metric dimension as well as a $k$-metric basis of any tree in linear time.

\subsection{Combinatorial issues}

Based on the nature of the problem, one might immediately consider analysing a possible monotonicity with respect to $k$ on the $k$-metric dimension. Such property can be relatively clearly deduced since for any $k$-resolving set $B$ and any $x\in B$, all pairs of different vertices in $V(G)$ are distinguished by at least $k$ vertices of $B$, which allows to claim that $B-\{x\}$ will always be a $(k-1)$-resolving set for $G$. This proves the following result.

\begin{theorem}{\em \cite{Estrada-Moreno-2015}}
Let $G$ be a $k$-metric dimensional graph and let $k_1,k_2$ be two integers. If $1\le k_1<k_2\le k$, then $\dim_{k_1}(G)<\dim_{k_2}(G)$.
\end{theorem}

Consequences of the result above are for instance the following ones (published in \cite{Estrada-Moreno-2015}), where we consider $G$ is a $k$-metric dimensional graph.

\begin{itemize}
\item For every $r\in\{1,\ldots, k\}$, $\dim_r(G)\ge \dim(G)+(r-1).$
\item If $G\not\cong P_n$, then  for any $r\in \{1,\ldots,k\}$, $\dim_{r}(G)\geq r+1.$
\end{itemize}

Based on the NP-hardness of the $k$-MD problem, it is clearly desirable to have tight general bounds on the $k$-metric dimension of graphs, as well as, closed formulas for several specific and non-trivial families of graphs. Some of the most remarkable results on this issue are next shown. A first interesting result deals with the $k$-metric dimension of cycles.

\begin{proposition}{\em \cite{Estrada-Moreno-unpublished}}
Let $C_{n}$ be the  cycle graph of order $n$.
\begin{itemize}
\item If $n$ is even, then
\begin{itemize}
\item $\dim_k(C_n)=k+1$ for every $1\le k\le\frac{n}{2}-1$,
\item $\dim_k(C_n)=k+2$ for every $\frac{n}{2}\le k\le n-2$.
\end{itemize}
\item If $n$ is odd, then
\begin{itemize}
\item $\dim_k(C_n)=k+1$ for every $1\le k\le n-1$.
\end{itemize}
\end{itemize}
\end{proposition}

Several other results concerning the $k$-metric dimension of unicyclic graphs are described in \cite{Estrada-Moreno-unpublished}.

\subsection{The case of trees}

The first situation that needs to be considered is that of paths graphs, since it behaves different from the remaining tree structures.

\begin{proposition}{\em \cite{Estrada-Moreno-2015}}
Let $k\geq 3$ be an integer. For any path graph $P_{n}$ of order $n\geq k+1$, $\dim_{k}(P_{n})=k+1.$
\end{proposition}

In order to study the trees different from paths, we need some terminology, which uses some concepts already defined in Subsection \ref{subsec:terminology}. For a graph $G$ and any exterior major vertex $w\in M(G)$, $l(w)$ stands for the smallest distance between $w$ and any of the terminal vertices of $w$. Moreover, if $G$ is $k'$-metric dimensional, then for every $k\le k'$, we let
\[
I_k(w)=\left\{ \begin{array}{ll}
\left(\ter(w)-1\right)\left(k-l(w)\right)+l(w), & \mbox{if } l(w)\le\lfloor\frac{k}{2}\rfloor,\\[0.2cm]
\left(\ter(w)-1\right)\lceil\frac{k}{2}\rceil+\lfloor\frac{k}{2}\rfloor, & \mbox{otherwise.}
\end{array}
\right.
\]
With this notation in mind, the $k$-metric dimension of trees is as follows.

\begin{theorem}{\em \cite{Estrada-Moreno-2015}}
If $T$ is a $k'$-metric dimensional tree which is not a path, then for any $k\in \{1,\ldots, k'\}$, $$\dim_{k}(T)=\sum_{w\in M(T)}I_{k}(w).$$
\end{theorem}

We need to remark that the value of the formula in the result above indeed represents a lower bound for the $k$-metric dimension of connected graphs in general, as also stated in \cite{Estrada-Moreno-2015}.

\subsection{An application of $k$-resolving sets}

A remarkable contribution to this topic is that one presented in \cite{Bailey-2019}. There was described a construction of error-correcting codes from graphs by means of $k$-resolving
sets, and having a decoding algorithm which makes use of covering designs. That is, given a connected graph $G$ on $n$ vertices and diameter $d$, and a $k$-resolving set $S=\{v_1,v_2,\ldots,v_{\ell}\}$ of cardinality $\ell$, the set
$$\mathcal{C}(G,S) = \{(d_G(u,v_1),d_G(u,v_2),\ldots,d_G(u,v_{\ell})) \, : \, u\in V\}$$ was called a {\em $(G,k)$-code} in \cite{Bailey-2019}.

From this definition, it is clear that $\mathcal{C}(G,S)$ is an error-correcting code of length $\ell$, size $n$ and minimum Hamming distance at least $k$, over the alphabet $\{0,\ldots,d\}$, and it can correct $r=\lfloor (k-1)/2 \rfloor$ errors. Note that it must be $k\geq 3$ (for otherwise the code has no sense).

Let $G_j(u)$ denote the subset of vertices of $G$ at distance $j$ from $u$.
Assume that $u\in V$ and $v_u=r(u|S)=(d_G(u,v_1),d_G(u,v_2),\ldots,d_G(u,v_{\ell})) \in \mathcal{C}(G,S)$. We transmit $v_u$ and receive the word $X=(x_1,x_2,\ldots,x_{\ell})$, which is assumed to have at most $r$ errors. Then, we have the following.

\begin{lemma}{\em \cite{Bailey-2019}}\label{lemma:uniqueness}
Let $u$, $v_u$ and $X$ be as above, and suppose that $I$ is an $(\ell-r)$-subset of $\{1,\ldots,\ell\}$.
\begin{itemize}
\item If the received word $X$ contains no errors in the positions indexed by $I$, then
$\bigcap_{i\in I} G_{x_i}(v_i) = \{u\}.$
\item If the received word $X$ does contain an error in a position in $I$, then
$\bigcap_{i\in I} G_{x_i}(v_i) = \varnothing.$
\end{itemize}
\end{lemma}

To see the validity of this, we observe the following. The existence of no errors in the positions indexed by $I$, means both words $X$ and $v_u$ contain the same entries in the corresponding positions. Thus, $u$ is the unique vertex having such distances from the corresponding vertices in the $k$-resolving set. For the contrary, the existence of errors in such positions (in $I$), means that no such vertex can exist. That is, the intersection is empty.

The goal when decoding a $(G,k)$-code is that of finding (efficiently if possible) an $(\ell-r)$-subset of positions for which the intersection $\displaystyle{\bigcap_{i\in I} G_{x_i}(v_i)}$ is not empty. We might successively enumerate all the $(\ell-r)$-subsets of $\{1,\ldots,\ell\}$ to achieve this, but this is clearly slow in practice. By making the assumption that there are at most $r'<r$ errors (for instance $r'=2$ or $r'=3$), the notion of uncovering design can be used. That is, given the integers $\nu$, $\kappa$, $\tau$ such that $\nu\geq \kappa \geq \tau \geq 0$.  A {\em $(\nu,\nu-\kappa,\tau)$-uncovering} is a collection $\mathcal{U}$ of $(\nu-\kappa)$-subsets of $\{1,\ldots,\nu\}$ with the property that any $\tau$-subset of $\{1,\ldots,\nu\}$ is disjoint from at least one member of $\mathcal{U}$. To see more information on uncovering designs we suggest \cite{Bailey-2006,Bailey-2009}.

Now assume we have a $(G,k)$-code of length $\ell$, and we want to correct $r'$ errors. If $\mathcal{U}$ is an $(\ell, \ell-r, r')$-uncovering, then one can proceed as follows: for a received word $X$, consider each $I\in\mathcal{U}$ and search for $\displaystyle{\bigcap_{i\in I} G_{x_i}(v_i)}$. By Lemma~\ref{lemma:uniqueness}, this intersection is either empty or contain the vertex $u$ corresponding to the transmitted word $U$.  If $\mathbf{u}$ contains at most $r'$ errors, then by the definition of uncovering, there exists an $I\in\mathcal{U}$ which is disjoint from the error positions. In consequence, it is guaranteed the possibility of finding the transmitted word.
In order to compute $\displaystyle{\bigcap_{i\in I} G_{x_i}(v_i)}$, one need to consider the matrix $M$ whose rows are indexed by $V$ and whose columns are indexed by the $k$-resolving set $S$, and where the entries are $M_{uv}=d_G(u,v)$. That is, the rows of $M$ are precisely the codewords.  For a given $I\subseteq S$, one needs to examine the rows of the submatrix to find a row which agrees with $X$ in the corresponding positions. If such a row exists, by Lemma~\ref{lemma:uniqueness}, it must be unique and will correspond to the vertex $u$ which we search for.

Clearly, the realization of the error correcting code above relies on the fact that we would need to have the possibility of finding a $k$-resolving set for a given graph that one would pretend to use in the code. However, computing the $k$-metric dimension of graphs is known to be NP-hard. One solution for this can be that of generating graphs with a prescribed value in the $k$-metric dimension, which can be made by using some structures of product graphs. In connection with this, in \cite{Bailey-2019}, was presented a sample of error correcting code that precisely uses the grid graphs $P_s\times P_t$, which is the Cartesian product of two paths. It was first proved there that for any $s,t\ge 2$, the graph $G=P_s\Box P_t$ is $(s+t-2)$-metric dimensional. With this value in mind, the $k$-metric dimension of grid graphs was given.

\begin{theorem}{\em \cite{Bailey-2019}}
For any grid graph $G=P_s\Box P_t$ and every $k\in \{1,\ldots,s+t-2\}$,
$$\dim_k(G)=2k.$$
\end{theorem}

The idea of the proof comes by first noticing that taking $k$ vertices either in each one of the two ``horizontal borders'', or in each one of the two ``vertical borders'' of the grid gives a $k$-resolving set for the grid. In addition, since the two neighbors of any corner vertex are only resolved by vertices from the copies of the path to which the corner belongs, one can see that any $k$-resolving set for the grid must be part of precisely the two ``horizontal borders'', or the two ``vertical borders'' of the grid. These are the key points for proving the result above.

With this in mind, for the grid graph $G=P_s \Box P_t$, and for any $k\leq s+t-2$, the $(G,k)$-code with $n=st$ codewords of length $\ell=2k$ over an alphabet of size $s+t+1$ was  given in \cite{Bailey-2019}. Such a $(G,k)$-code has correction capability $r=\left\lfloor (k-1)/2 \right\rfloor$.

\subsection{$k$-resolving sets versus product graphs}

Following with the last notes from the previous subsection, we now center our attention into the $k$-metric dimension of product graphs. The main contributions to this topic are centered into three products: lexicographic (\cite{Estrada-Moreno-2016-c,Estrada-Moreno-2016-e}), hierarchical (\cite{Klavzar-2021+}) and corona (\cite{Estrada-Moreno-2016-d}) products, although some sporadic results have appeared for some other products. Namely, still several contributions on the $k$-metric dimension of product graphs could enrich this theory.

The corona product graph $G\odot H$ of two graph $G$ and $H$ was studied in \cite{Estrada-Moreno-2016-d}. Indeed, the work \cite{Estrada-Moreno-2016-d} centers its attention in a more general case of corona graphs. That is, when the $n$ copies of the graph $H$ used in the product are replaced with $n$ non necessarily isomorphic graphs $H_1,\dots,H_n$. However, for the purposes of this survey this is not exactly of much interest, since the generalization can be done a relatively natural way. It was first noticed in \cite{Estrada-Moreno-2016-d} that studying such product needs to be separated in two different scenarios: a first one whether the graph $G$ has order larger than one, and a second one whether $G$ is the singleton graph $K_1$, which can be understood as the join of $K_1$ and $H$ (an example of such graph is for instance the wheel graph $W_{1,n}=K_1+C_n=K_1\odot C_n$.

As usual in this topic, the first contributions are centered into finding the value $k$ for which a given corona product graph is $k$-metric dimensional. To see this, the following concept is required. For a connected non-trivial graph $H$, it is defined
$$\mathcal{C}(H)=\min_{x,y\in V(H)}\{\vert N_H(x)\triangledown N_H(y) \cup\{x,y\}\vert\},$$ where $\triangledown$ represents the symmetric difference between two sets. By using this fact, it was proved in \cite{Estrada-Moreno-2016-d} that for any connected non-trivial graph $G$ (different from $K_1$) and any non-trivial graph $H$, $$\mbox{$G\odot H$ is $k$-metric dimensional if and only if $k=\mathcal{C}(H)$.}$$ Once made this, some particular situations are dealt with. Among them, we remark the following one which uses the girth (length of a shortest cycle) of $H$ as a premise.

\begin{corollary}\label{cuello-ge5}{\em\cite{Estrada-Moreno-2016-d}}
Let $G$ be a connected non-trivial graph of order $n$ and let $H$ be a $\delta$-regular graphs with $g(H)\ge 5$. Then
$G\odot H$ is a $2\delta$-metric dimensional graph.
\end{corollary}

For the case of corona product graphs with the first factor isomorphic to $K_1$, the following was deduced in \cite{Estrada-Moreno-2016-d}. For a graph $H$ of order $n'\ge 2$ and maximum degree $\Delta(H)$, the graph $K_1+ H$ is $k$-metric dimensional if and only if $k=\min\{\mathcal{C}(H),n'-\Delta(H)+1\}$. For instance, for any $n\ge 4$, the fan graph  $F_{1,n}$ is $3$-metric dimensional, and for any $n\ge 5$, the wheel graph  $W_{1,n}$ is $4$-metric dimensional.

Once found the limit values for $k$ described above, the $k$-metric dimension of corona product graphs was dealt with in \cite{Estrada-Moreno-2016-d}. Tight general lower and upper bounds for the $k$-metric dimension of such graphs were deduced. We summarize some of the most remarkable contributions in the following table.

\begin{table}[ht]
  \centering
  \begin{tabular}{|c|c|c|}
    \hline
    Connected graph $G$ & Connected non-trivial graph $H$ & $\dim_k(G\odot H)$ \\ \hline
    $G$, order $n\ge 2$ & $H$, order $n'$ & $\begin{array}{l}
                                                                              \ge n\dim_k(H) \\
                                                                              \le nn'
                                                                            \end{array}$
     \\ \hline
    $G$, order $n\ge 2$ & $H$, diameter at most $2$ & $=n\dim_k(H)$ \\ \hline
    $K_1$ & $P_n$ ($n\ge 6$) & $\begin{array}{l}
                      =\left\lceil\frac{n+1}{2}\right\rceil \mbox{(when $k=2$)} \\[0.15cm]
                      =n-\left\lceil\frac{n-4}{5}\right\rceil \mbox{(when $k=3$)}
                    \end{array}$
     \\
    \hline
    $K_1$ & $C_n$ ($n\ge 7$) & $\begin{array}{l}
                      =\left\lceil\frac{n}{2}\right\rceil \mbox{(when $k=2$)} \\[0.15cm]
                      =n-\left\lceil\frac{n}{5}\right\rceil \mbox{(when $k=3$)}\\
                      =n \mbox{(when $k=4$)}
                    \end{array}$
     \\
    \hline
  \end{tabular}
  \caption{Bounds and closed formulae for the $k$-metric dimension of corona product graphs appeared in \cite{Estrada-Moreno-2016-d}.}\label{tab:corona-k-dim}
\end{table}

The next interesting case is the family of lexicographic product graphs. In this direction, the investigation was divided into two cases. One case whether $k=2$ \cite{Estrada-Moreno-2016-c}, since it is known that any graph $G$ is $k$-metric dimensional for some $k\ge 2$, and there are indeed a lot of $2$-metric dimensional graphs. Another situation \cite{Estrada-Moreno-2016-e} was developed for the remaining values $k\ge 3$. As usual in this kind of researches, the first results were as the following one, in which the definitions coming from Section \ref{sect:adjacency} are necessary, and whether the existence or not of (false or true) twins in crucial.

\begin{corollary}{\em\cite{Estrada-Moreno-2016-e}}\label{coro-k-dim-nal-lexico}
Let $G$ be a connected non-trivial graph and let $H$ be a graph of order $n'\ge 2$ and maximum degree $\delta(H)$.
\begin{itemize}
\item If $G$ is twins free, then the graph $G\circ H$ is $k$-metric dimensional $k=\mathcal{C}(H)$.
\item If $G$ contains at least one false twin and one true twin, then the graph $G\circ H$ is $k$-metric dimensional for $k=\min\{2\delta(H)+2,2(n'-\Delta(H)),\mathcal{C}(H)\}$.
\item If $G$ is true twins free and contains at least one false twin, then the graph $G\circ H$ is $k$-metric dimensional for $k=\min\{2\delta(H)+2,\mathcal{C}(H)\}$.
\item If $G$ is false twins free and contains at least one true twin, then the graph $G\circ H$ is $k$-metric dimensional for $k=\min\{2(n'-\Delta(H)),\mathcal{C}(H)\}$.
\end{itemize}
\end{corollary}

Once known these facts, the investigation is then devoted to compute the $k$-metric dimension of lexicographic products for any suitable value of $k$, or finding relationships between this parameter and that of the factors in the product together with other invariants of the factors. It is important to remark three key points in these contributions. One of them is the usefulness of the $k$-adjacency dimension of the graphs in the second factor, the second one, the role which play the (true or false) twins in the first factor of the product, and the third one concern the possibility of knowing not only the $k$-metric dimension of $G\circ H$ but also that of $G\circ \overline{H}$, based on the results given in \ref{sect:adjacency} ($\overline{H}$ is the complement graph of $H$). To this end, the use of the following notations and terminology is required, as well as, some other ones from Section \ref{sect:adjacency}.

A graph $H$ satisfies the \textbf{Property $\mathcal{P}_1$}, if it contains $r$ distinct $k$-adjacency bases $A_{i_1},A_{i_2},\ldots,A_{i_r}$ such that for every $j,l\in\{1,\ldots,r\}$, $j\ne l$, and every $x,y\in V(H)$, $x\ne y$, it follows, $$|(A_{i_j}\cap(V(H)-N_{H}(x)))\cup(A_{i_l}\cap(V(H)-N_{H}(y)))|\ge k.$$
Also, $H$ satisfies the \textbf{Property $\mathcal{P}_2$}, if it contains $r$ distinct $k$-adjacency bases $A_{i_1},A_{i_2},\ldots,A_{i_r}$, such that for every $j,l\in\{1,\ldots,r\}$, $j\ne l$, and every $x,y\in V(H)$, $x\ne y$, it follows, $$|(A_{i_j}\cap N_{H}[x])\cup(A_{i_l}\cap N_{H}[y])|\ge k.$$
Existence of graphs satisfying these properties were given in \cite{Estrada-Moreno-2016-e}. With this in mind, the following formulas (although in a more general setting) were presented in \cite{Estrada-Moreno-2016-e}.

\begin{corollary}{\em \cite{Estrada-Moreno-2016-e}}
Let $G$ be a connected graph of order $n\ge 2$ and let $H$ be a non-trivial graph. Then for any $k$ taken according to Corollary \ref{coro-k-dim-nal-lexico}, the following statements hold.
\begin{itemize}
\item If $G$ is twins free, then $$\dim_k(G\circ H)=\dim_k(G\circ\overline{H})=n\cdot\adim_k(H).$$
\item If $G$ is false twins free and $H$ holds Property $\mathcal{P}_1$, then $$\dim_k(G\circ H)=n\cdot\adim_k(H).$$
\item If $G$ is true twins free and $H$ holds Property $\mathcal{P}_2$, then $$\dim_k(G\circ H)=n\cdot\adim_k(H).$$
\end{itemize}
\end{corollary}

These and some other general results for the $k$-metric dimension of lexicographic product graphs were given in \cite{Estrada-Moreno-2016-e}. In addition, some particular situations involving cycles and/or paths in the second factor were settled.

On the other hand, the case $k=2$ was completely analyzed in \cite{Estrada-Moreno-2016-c}, where formulas for $\dim_2(G\circ H)$ were given based on the fact that for any connected graph $G$ of order $n\ge 2$ and any nontrivial graph $H$, there exists a non-negative integer $f(G,H)$ such that
$$\dim_2(G\circ H)=n\cdot \adim_2(H)+f(G,H).$$
The aim of \cite{Estrada-Moreno-2016-c} was then to characterize all possible values that such $f(G,H)$ can attain. Such characterizations are based on the existence or not of (true or false) twin vertices in $G$, the number of equivalence classes that such vertices define, and some properties which need to be satisfied by the $2$-adjacency bases of $H$.

Finally, to end this subsection, we mention that the $k$-metric dimension of the hierarchical product of graphs was studied in \cite{Klavzar-2021+} together with other less common graph operations called splice and link products. The most remarkable aspect of this work is the application of some integer linear programming formulation while computing the $k$-metric dimension of special cases of the hierarchical product graph $\dim_k(G(U)\sqcap H)$ (with respect to a set $U\subseteq V(G)$), aimed to show the tightness of the main bound of the article, which is next stated. For definitions of $G(U)\sqcap H$, see \cite{Klavzar-2021+} precisely.

\begin{theorem}{\em \cite{Klavzar-2021+}}
Let $G$ be a graph, $U\subseteq V(G)$ and $\dim_k(G(U))=t$. If $H$ is a graph with $\dim_{\lceil k/t\rceil}(H)<\infty$, then
$$\dim_k(G(U)\sqcap H)\le n(H)\dim_k(G(U)).$$
\end{theorem}

We need to finally remark a couple of more general settings in connection with this parameter. In one side, the idea of $k$-metric dimension of graphs is extrapolated to the notion of metric spaces, and some old ideas on the metric dimension from \cite{Blumenthal-1953} are taken into account and generalized to $k$-metric dimension. These ideas were presented in \cite{Beardon-2019,Rodriguez-Velazquez-2020}. On the other hand, the $k$-metric dimension was extended in \cite{Estrada-Moreno-2021+} when the notion the classical vertex distance in graphs is understood in the following way.

Given a positive integer $t$ and the classical distance $d$ in a graph $G=(V,E)$, consider the metric  $d_t:V\times V\rightarrow \mathbb{R}$, defined by
$$d_t(x,y)=\min\{d(x,y),t\}.$$
In this context, the concept of $k$-resolving set is naturally defined by just taking the metric $d_t$ instead of the classical one induced by the standard vertex distance. $k$-resolving sets are called $(k,t)$-resolving sets and the $k$-metric dimension is  called $(k,t)$-metric dimension. The case $t\ge D(G)$, where $D(G)$ denotes the diameter of $G$, corresponds to the original theory of $k$-metric dimension, and  the case $t=2$ corresponds to the theory of $k$-adjacency dimension (see Section \ref{sect:adjacency}). Furthermore, it is pointed out that such approach allows to extend the theory of $k$-metric dimension to the general case of non-necessarily connected graphs, which is a limitation that has the metric dimension theory. In \cite{Estrada-Moreno-2021+}, several computational and combinatorial results for the $(k,t)$-metric dimension of graphs were given. However, nothing more has been done concerning this so far.

\subsection{Some open problems}

\begin{itemize}
  \item In \cite{Diaz-2012}, it was proved that computing the metric dimension of graphs is NP-hard for planar graphs, and in contrast, it is polynomial for outerplanar graphs. Can this be generalized for every suitable values $k\ge 2$ while computing the $k$-metric dimension of graphs?
  \item General studies on the $k$-metric dimension of the four standard product graphs are known for the lexicographic case only. In this sense, it is desirable to continue studying the three remaining ones.
  \item In \cite{Estrada-Moreno-2021+}, it is proved that computing the $(k,t)$-metric dimension of graphs is NP-hard for any odd integer $k$. Can you prove that computing the $(k,t)$-metric dimension of graphs is also NP-hard whenever $k$ is even?
\end{itemize}

%% file: partition-dim.tex
\section{Resolving partitions}

The notion of partition dimension of graphs was born aimed to giving some more insight into the concept of metric dimension in graphs. Such concept was first introduced in \cite{Chartrand-2000-b}, and since then, a lot of investigations on partition dimension of graphs have been developed. Although from our point of view, this topic would deserve itself a separate survey, in this section, we shall try to make a compilation of results on the parameter, emphasizing on those more recent ones that, in our humble opinion, are the most remarkable contributions.

The parameter of partition dimension is now very well studied, and one can find in the literature a lot of different styles of contributions on it. However, a very surprising situation is that there is not any computational aspect of such parameter. This means that it is not known on whether computing such parameter could belong to the NP-hard class, or to any other related complexity class. It could even be possible that computing the partition dimension of trees will be an NP-hard problem.

On the other hand, variations of the parameter itself are also known. For instances, we can mention the connected partition dimension \cite{Saenpholphat-2002}, the strong partition dimension \cite{Yero-2014}, the metric chromatic number \cite{Chartrand-2009}, the locating-chromatic number \cite{Chartrand-2002}, the $k$-partition dimension \cite{Estrada-Moreno-2020}, and the fault-tolerant partition dimension \cite{Javaid-2009} (which is indeed the case $k=2$ in the $k$-partition dimension). We shall also shortly survey a couple of these variants as well.

\subsection{Partition dimension}\label{subsec:pd}

Given a connected graph $G=(V,E)$ and an ordered partition $\Pi =\{P_1,P_2, ...,P_t\}$ of the vertex set of $G$, the {\it partition representation} of a vertex $v\in V$ with respect to the partition $\Pi$ is the vector $r(v|\Pi)=(d_G(v,P_1),d_G(v,P_2),...,d_G(v,P_t))$, where $d_G(v,P_i)$, with $1\leq i\leq t$, represents the distance between the vertex $v$ and the set $P_i$. That is $d_G(v,P_i)=\min_{u\in P_i}\{d_G(v,u)\}$. It is said that $\Pi$ is a {\it resolving partition} for $G$ if for every pair of distinct vertices $u,v\in V$, it follows $r(u|\Pi)\ne r(v|\Pi)$, or equivalently, if for any two distinct vertices $u,v\in V(G)$, there exists a set $P_i\in \Pi$ such that $d_G(u,P_i)\ne d_G(v,P_i)$. In such case, we say that the set $P_i$ \emph{resolves} (\emph{identifies} or \emph{recognizes}) the vertices $u,v$. The {\it partition dimension} of $G$ is the minimum number of sets among all resolving partitions for $G$, and is denoted by $pd(G)$. A resolving partition of cardinality $\pd(G)$ is a \emph{partition basis} for $G$. Concepts above were first introduced in \cite{Chartrand-2000-b}.

Notice that, while checking whether a given partition of $V(G)$ represents a resolving partition, one only needs to check pairs of vertices belonging to a same set, since those vertices belonging to different sets of the partition are already resolved at least by the two sets which they belong. Figure \ref{fig:pd-3} shows an example of graph with partition dimension 3 with vertices colored according to the sets they belong in a partition basis.

\begin{figure}[h]
\centering
\begin{tikzpicture}[scale=.8, transform shape]
\node [draw, shape=circle,fill=red] (a1) at  (0,0) {};
\node [draw, shape=circle,fill=green] (a5) at  (0,1.5) {};
\node [draw, shape=circle,fill=red] (a2) at  (3,0) {};
\node [draw, shape=circle,fill=green] (a6) at  (3,1.5) {};
\node [draw, shape=circle,fill=green] (a3) at  (6,0) {};
\node [draw, shape=circle,fill=green] (a7) at  (6,1.5) {};
\node [draw, shape=circle,fill=blue] (a4) at  (9,0) {};
\node [draw, shape=circle,fill=green] (a8) at  (9,1.5) {};
\node [draw, shape=circle,fill=blue] (a9) at  (12,0) {};
\node [draw, shape=circle,fill=green] (a10) at  (12,1.5) {};

\draw(a1)--(a2)--(a3)--(a4)--(a9)--(a10)--(a8)--(a7)--(a6)--(a5)--(a1);
\draw(a2)--(a6);
\draw(a3)--(a7);
\draw(a4)--(a8);

\end{tikzpicture}
\caption{A graph with partition dimension 3.}
\label{fig:pd-3}
\end{figure}
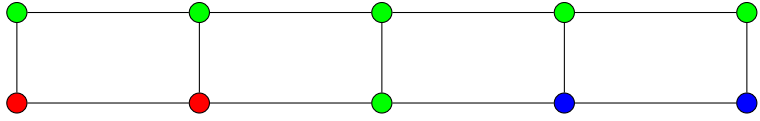

A natural question that one first consider is that of the existence of resolving partitions in graphs, and the answer is positive: we can always construct such a partition for any connected graph $G$. To this end, consider $S=\{v_1,\dots,v_k\}$ is any metric basis for $G$. Now, let $\Pi=\{P_1=\{v_1\},P_2=\{v_2\},\dots,P_k=\{v_k\},V(G)\setminus S\}$. Since $S$ is a metric basis, for any two distinct vertices $u,v\in V(G)$, there is a vertex $v_j\in S$ such that $d_G(u,v_j)\ne d_G(v,v_j)$. Thus, one can readily see that for the vertices $u,v\in V(G)$, the set $P_j=\{v_j\}$ resolves the pair $u,v$, and so, $\Pi$ is a resolving partition. Since every connected graph contains at least one metric basis, it is clear that it also contains at least one resolving partition. By this arguments, the following relationship between metric dimension and partition dimension is deduced. For any connected graph $G$,
\begin{equation}\label{eq:pd-dim}
\pd(G)\le \dim(G)+1\;\;\mbox{\cite{Chartrand-2000-b}}.
\end{equation}
This bound is indeed tight for some classes of graphs including paths, cycles, complete graphs and other ones. However, there are a lot of cases in which it behaves very badly. Now, with respect to the realization of graphs with given values in the metric dimension and partition dimension, the following result is known from \cite{Chartrand-2000-b}.

\begin{theorem}\emph{\cite{Chartrand-2000-b}}
For every pair $a, b$ of positive integers with $\left\lceil\frac{b}{2}\right\rceil+ 1 \le a \le b + 1$, there exists a connected graph $G$ such that $\pd(G)=a$ and $\dim(G) = b$.
\end{theorem}

To prove this result the authors used the complete bipartite graph $K_{s,t}$ with $s = a$ and $t = b-a + 2$. Also, this theorem above immediately raised up one question: Is it the case that $\pd(G)\ge \left\lceil\frac{\dim(G)}{2}\right\rceil+ 1$ for every nontrivial connected graph G? Such question was negatively answered in \cite{Chappell-2008}, where authors showed that for every pair $a, b$ of positive integers with $3 \le a \le b + 1$, there exists a connected graph $G$ such that $\pd(G)=a$ and $\dim(G) = b$.

\subsubsection{Bounding the partition dimension}

Clearly, for any connected graph $G$ of order $n$, it must happen $2\le \pd(G)\le n$, which are the trivial bounds for this parameter. In \cite{Chartrand-2000-b}, graph achieving the equality in such bounds were characterized as follows:

\begin{itemize}
  \item $\pd(G)=2$ if and only if $G$ is the path $P_n$.
  \item $\pd(G)=n$ if and only if $G$ is the complete graph $K_n$.
\end{itemize}

Since the graphs $G$ of order $n$ for which $\pd(G)=n$ are only the complete graphs, authors of \cite{Chartrand-2000-b} also characterized the case $\pd(G)=n-2$, and they proved that if $G$ be a connected graph of order $n\ge 3$, then $\pd(G) = n-1$ if and only if $G$ is one of the graphs $K_{1,n-1}$, $K_n-e$, $K_1 + (K_1 \cup K_{n-2})$. In addition, in \cite{Tomescu-2008}, the graphs of order $n$  having partition dimension $n-2$ were characterized. However, such characterization had some serious gaps and a corrected version of it was further presented in \cite{Hernando-2016}. In this same stye of result, all graphs of order $n\ge 11$ and diameter 2 with partition dimension $n-3$ were characterized in \cite{Baskoro-2020}.

Some general bounds that are known are limiting $\pd(G)$ in terms of some parameter of the graph. Some of them were first given in \cite{Chappell-2008} with corresponding examples of graphs achieving such bounds.

\begin{theorem}\emph{\cite{Chappell-2008}}
For any connected graph $G$ of maximum degree $\Delta$, $\pd(G)\ge 1 + \log_3 (\Delta+1)$.
\end{theorem}

\begin{theorem}\emph{\cite{Chappell-2008}}
For any connected graph $G$ with clique number $\omega$, $\pd(G)\ge 1 + \log_2 \omega$.
\end{theorem}

The example of a family achieving the equality in the bound above is a remarkable one. It is a graph defined as follows. Given positive integers $p, s$, with $p \ge 2$, let $H(p, s)$ be the graph whose vertex set is formed by those vectors in $\{0, 1, 2, \dots, s\}^p$
that contain exactly one zero, and two vertices are adjacent if they differ by at most one in each coordinate. Now, for all positive integers $p, s$, with $p\ge 2$, it was proved in \cite{Chappell-2008} that $\pd(H(p, s)) = p$. A nice drawing of the graph $H(3,6)$ appears in \cite[Page 3]{Chappell-2008}.

Another bounds for the partition dimension of graphs in terms of the chromatic number and metric dimension were also given in \cite{Chappell-2008}. Moreover, some limits on the number of vertices that can have a graph with a prescribed value in the partition dimension and other related invariants were deduced. We remark the next one.

\begin{theorem}\emph{\cite{Chappell-2008}}
If a graph $G$ has order $n$, partition dimension $p$, and diameter $s$, then $n\le ps^{p-1}$.
\end{theorem}

Other similar bounds in the style of the above one were given in \cite{Javaid-2008} while considering special graph families. We remark the following ones.

\begin{theorem}\emph{\cite{Javaid-2008}}
Let $n\ge 2$ and $p$ denote the partition dimension of the Gear graph $G_{2n}$. Then $n < \frac{1}{2}(3p^4(p + 2)2^{p-7}-1)$.
\end{theorem}

\begin{theorem}\emph{\cite{Javaid-2008}}
Let $n\ge 2$ and $p$ denote the partition dimension of the friendship graph $f_{n}$. Then $n \le \binom{p}{2}$.
\end{theorem}

Other bounds are in general obtained whether special families of graphs are considered, or when extra conditions are added. We remark some of the most remarkable bounds that exists for $\pd(G)$, some of them while trying to solve the problem of determining the partition dimension of the incidence graph of a finite projective plane, other ones for bipartite graphs, trees or unicyclic graphs, among other ones.

\begin{theorem}\emph{\cite{Blazsik-2017}}
\begin{itemize}
  \item The partition dimension of the incidence graph of a projective plane of order $q$ is at least of order $(2+o(1))\log_2q$.
  \item The partition dimension of the incidence graph of a projective plane of order $q$ is at most of order $(4+o(1))\log_2q$.
\end{itemize}
\end{theorem}

\begin{theorem}\emph{\cite{Chartrand-2000-b}}
Let $G$ be a connected bipartite graph with partite sets of cardinalities $r$ and $s$. Then,
\begin{itemize}
  \item $\pd(G)\le r + 1$, if $r = s$, and
  \item $\pd(G) \le max\{r, s\}$ if $r\ne s$.
\end{itemize}
Moreover, the equalities holds, if and only if $G$ is a complete bipartite graph.
\end{theorem}

The case of trees and unicyclic graphs have specially centered the attention of some works. To expose this, let $K(G)$ be the set of exterior major vertices of a graph $G$ with terminal degree larger than one, let $\kappa(G)=|K(G)|$, and let $l_i$ be the terminal degree of an exterior major vertex $s_i\in K(G)$. Also, let $\tau(G)=\max_{s_i\in K(G)}\{l_i\}$. With the above notation the following result can be deduced.

\begin{theorem}\emph{\cite{Rodriguez-Velazquez-2014}}
For any tree $T$  which is not a path, $\pd(T)\le \kappa(T)+\tau(T)-1.$
\end{theorem}

To prove such result, the following partition of $V(T)$ was made.  For a terminal vertex
$s_{ij}$ of a major vertex $s_i\in S$, let $S_{ij}$ be the set of vertices of $T$, different from $s_i$, belonging to the $s_i-s_{ij}$ path. If $l_i< \tau(T)-1$, then it is assumed $S_{ij}=\emptyset$ for every $j\in \{l_i+1,\dots,\tau(T)-1 \}$.
For every $j\in \{2,\dots,\tau(T) -1\}$, let $B_j=\cup_{i=1}^{\kappa(T)}S_{ij}$  and, for every $i\in \{1,...,\kappa(T)\}$, let
$A_i=S_{i1}$. Hence, it is shown in \cite{Rodriguez-Velazquez-2014} that $\Pi=\{A,A_1,A_2,\dots,A_{\kappa(T)},B_2,\dots,B_{\tau(T)-1}\}$ is a
resolving partition of $T$, where $A=V-\left(\displaystyle\left(\cup_{i=1}^{\kappa(T)}A_i\right)\cup \left(\cup_{j=2}^{\tau(T)-1}B_j\right)\right)$.

An improvement to the result above can be made if the next notations are considered. Let $\xi(T)$ be the number of support vertices of $T$ and let $\theta(T)$ be the maximum number of leaves adjacent to a support vertex of $T$.

\begin{theorem}\emph{\cite{Rodriguez-Velazquez-2014}}
For any tree $T$  of order $n\ge 2$,  $\pd(T)\le \xi(T)+\theta(T)-1$.
\end{theorem}

Since the number of leaves, $l(T),$ of a tree $T$ is bounded below by $\xi(T)+\theta(T)-1$, the result above leads to the following bound. For any tree $T$  of order $n\ge 2$, $\pd(T)\le l(T)$. In \cite{Rodriguez-Velazquez-2014}, it was obtained that for a tree $T$ with $l(T)\ge 4$ leaves,
$\pd(T)=l(T)$ if and only if $T$ is the star graph.

Some other general results concerning trees and block graphs (called generalized trees) were also obtained in \cite{Rodriguez-Velazquez-2014}. Such work has a ``kind of natural'' continuation in the work \cite{Fernau-2014} concerning unicyclic graphs.

\begin{theorem}\emph{\cite{Fernau-2014}}
Let $G$ be a connected unicyclic graph.
\begin{itemize}
  \item If $G$ is a cycle graph or every exterior major vertex of $G$ has terminal degree  one, then $\pd(G)=3.$
  \item If $G$ contains at least an exterior major vertex of terminal degree greater than one, then  $\pd(G)\le
\kappa(G)+\tau(G)+1.$
\end{itemize}
\end{theorem}

Among other interesting results, the work \cite{Fernau-2014} showed a relationship between the partition dimension of trees and unicyclic graphs.

\begin{proposition}
If $T$ is a spanning tree of a unicyclic graph $G$,  then $\pd(G)\le \pd(T)+3.$
\end{proposition}

Since the authors of \cite{Fernau-2014} were not able to find any unicyclic graph for which the bound above is tight, this work finish with an interesting conjecture which states the following.

\begin{conjecture}
\label{conj:pd-tree-unicyc}
If $T$ is a spanning tree of a unicyclic graph $G$,  then $\pd(G)\le \pd(T)+1.$
\end{conjecture}

Circulant graphs have also attracted the attention of several researches with respect to their partition dimension. There are bounds and closed formulae for this parameter in such graphs classes. Table \ref{tab:pd-values} contains some of these formulae. An interesting bound in this sense, appeared in \cite{Maritz-2018}.

\begin{theorem}\emph{\cite{Maritz-2018}}
If $G$ is the circulant graph $C(n,\pm\{1,2,\dots,j\})$, then if $t \ge 4$ is even, then there exists an infinite set of values of $n$, such that $\pd(C(n,\pm\{1,2,\dots,j\}))\le t/2+4$.
\end{theorem}

The result above disproved a result from \cite{Grigorious-2017} which stated that $\pd(C(n,\pm\{1,2,\dots,j\}))\ge t+1$.

\subsubsection{Formulas for the partition dimension}

In this subsubsection we are mainly interested into summarizing the main formulas that exist for the partition dimension of several significant families of graphs. This is done in Table which follows the ideas of previous ones for some related parameters.

\begin{table}[ht]
  \centering
  \small{
  \begin{tabular}{|c|c|c|}
    \hline
    \textbf{Graphs} $\mathbf{G}$ & $\mathbf{\pd(G)}$ & \textbf{Reference} \\ \hline
    Grid graphs $P_r\Box P_t$ & $3$ & \cite{Yero-2010} \\ \hline
    Circulant graphs $C(n,\{\pm1, \pm2\})$, $n\ge 9$ odd & $3$ & \cite{Salman-2012} \\ \hline
    Circulant graphs $C(n,\{\pm1, \pm2\})$, $n\ge 12$ even & $3$ & \cite{Grigorious-2014} \\ \hline
    $\begin{array}{c}
      \mbox{Circulant graphs $C(n,\pm\{1,2,\dots,j\})$}\\
      \mbox{$1< j< \lfloor n/2 \rfloor$}\\
      \mbox{$n\ge (j+k)(j+1)$ and $n\equiv k$ (mod $2j$)}\\
      \mbox{when $j$ is even and $gcd(k,2j)=1$, or}\\
      \mbox{when $j$ is odd and $k=2m$, $1\le m\le j$}
    \end{array}$
     & $j+1$ & \cite{Grigorious-2017} \\ \hline
     Generalized M\"obius Ladder $M_{m,n}$ & $\begin{array}{ll}
                                                5, & \mbox{$n,m \equiv 1$, (mod 2), $m-n \ge 4$} \\
                                                4, & \mbox{$n \equiv 0$ and $m \equiv 1$, (mod 2)} \\
                                                4, & \mbox{$n \equiv 1$ and $m \equiv 0$, (mod 2)} \\
                                                5, & \mbox{$n,m \equiv 0$, (mod 2), $m-n \ge 4$}
                                              \end{array}
     $ & \cite{Hussain-2018}\\ \hline
     Fullerene graphs $G_1[n]$, $G_2[n]$ & $3$ & \cite{Mehreen-2018} \\ \hline
     Series parallel graphs & several formulas & \cite{Mohan-2019} \\ \hline
  \end{tabular}
  }
  \caption{Graphs for which their partition dimension has been computed.}\label{tab:pd-values}
\end{table}

\subsubsection{Partition dimension versus product graphs}

The partition dimension has been also studied in connection with some product graphs. We can find, for instance, works dealing with the Cartesian \cite{Yero-2014,Yero-2010}, the strong \cite{Yero-2014}, the lexicographic \cite{Campanelli-2017}, the corona \cite{Darmaji-2012,Rodriguez-Velazquez-2016-a}, and the rooted \cite{Monica-2019} products. We next expose some interesting contributions about this.

\begin{table}[ht]
  \centering
  \small{
  \begin{tabular}{|c|c|c|}
    \hline
    \textbf{Product $G*H$} & $\mathbf{\pd(G*H)}$ & \textbf{Reference} \\ \hline
    $G\Box H$ & $\leq \pd(G)+\pd(H)$ & \cite{Yero-2010} \\ \hline
    $G\Box H$ & $\leq \pd(G)+\pd(H)-1$ & \cite{Yero-2014} \\ \hline
    $G\Box H$ & $\le \pd(G)+\dim(H)$ & \cite{Yero-2010} \\ \hline
    $K_n\Box H$ & $\leq \min\left\{\left\lceil\frac{n}{k}\right\rceil(pd(H)-1)+k\,:\,2\le k\le n-1\right\}$ & \cite{Yero-2014} \\ \hline
    $G\circ H$ & $\begin{array}{c}
                    \le |V(H)|\pd(G) \\
                    \ge \pd(H)+1
                  \end{array}
    $ & \cite{Campanelli-2017} \\ \hline
    $G\circ H$ & $\leq |V(G)| \pd(K_1+H)$ & \cite{Campanelli-2017} \\ \hline
    $G\boxtimes H$ & $\begin{array}{c}
                   \leq \pd(G)\cdot \pd(H) \\
                   \ge 4
                 \end{array}$ & \cite{Yero-2014} \\ \hline
    $P_r\boxtimes P_t$ & $=4$ & \cite{Yero-2014} \\ \hline
    $G\boxtimes K_n$ & $\ge n+2$ & \cite{Yero-2014} \\ \hline
    $G\circ_v H$, $|V(G)|,|V(H)|>2$ & $\begin{array}{c}
                   \leq \pd(G)+\pd(H) \\
                   \ge \pd(H)
                 \end{array}$ & \cite{Monica-2019} \\ \hline
    $P_r\circ_v P_t$, $r,t\ge 4$ & $=3$ & \cite{Monica-2019} \\ \hline
    $C_r\circ_v C_t$, $r,t\ge 4$ & $=4$ & \cite{Monica-2019} \\ \hline
    $G\odot H$, $|V(G)|\ge 2|$ & $\le \frac{1}{|V(G)|}\dim(G\odot H)+\pd(G)+1$ & \cite{Rodriguez-Velazquez-2016-a} \\ \hline
    $G\odot H$ & $\ge \pd(H)$ & \cite{Rodriguez-Velazquez-2016-a} \\ \hline
    $G\odot H$, $D(H)\le 2$ & $\leq \pd(G)+\pd(H)$ & \cite{Rodriguez-Velazquez-2016-a} \\ \hline
    $\begin{array}{c}
       G\odot H \\
       \mbox{$\alpha$, non trivial
components of $H$} \\
       \mbox{$\beta$, trivial
components of $H$}
     \end{array}
    $ & $\begin{array}{ll}
\le \pd(G)+|V(H)|-\alpha, & \textrm{$\alpha\ge 1$ and $\beta\ge 1$,}\\
\le \pd(G)+|V(H)|-\alpha+1, & \textrm{$\alpha\ge 1$ and $\beta = 0$,}\\
\le \pd(G)+|V(H)|, & \textrm{$\alpha =0$.}
\end{array}$ & \cite{Rodriguez-Velazquez-2016-a} \\ \hline
  \end{tabular}
  }
  \caption{Relationships on the partition dimension of product graphs and that of their factors together with other parameters.}\label{tab:pd-products}
\end{table}

We would like to recall that results above are just some representative examples (the most general ones) from the cited references. Each of the references, also contains several other bounds and/or closed formulae when specific families of graphs are considered, or whether extra conditions are introduced in the statements.

Concerning the reference \cite{Campanelli-2017}, we want to remark a mistake which appears there. In \cite[Theorem 3.1]{Campanelli-2017} authors claim that if $H$ is any non trivial connected non complete graph, then, for any integer $n\ge 3$,
$$\pd(K_n\circ H)=n\cdot \pd(H+K_1).$$
The result above is not true, since there is a gap in the proof which is indeed impossible to fix. A corrected result is as follows.

\begin{proposition}
Let $H$ be any non trivial connected non complete graph. Then, for any integer $n\ge 3$,
$$n\cdot \pd(H+K_1)-n\le \pd(K_n\circ H)\le n\cdot \pd(H+K_1).$$
\end{proposition}

\begin{proof}
By \cite[Theorem 2.2]{Campanelli-2017} we have that $\pd(K_n\circ H)\le n\cdot \pd(H+K_1)$. On the other hand, let $\Pi_2=\{B_1,B_2,...,B_s\}$ be a partition basis for $K_n\circ H$. Since every two vertices $(u_i,v_j)$ and $(u_i,v_l)$ have the same distance (distance one) to every other vertex $(u_g,v_h)$ where $g\ne i$, it follows that $(u_i,v_j), (u_i,v_l)$ must be resolved by some set $B_f\in \Pi_2$ such that $B_f\cap (\{u_i\}\times V(H))\ne \emptyset$ and $B_f\cap ((V(K_n)-\{u_i\})\times V(H))=\emptyset$. Thus, the restricted partition $\Pi_2^{(i)}$ of $\{u_i\}\times V(H)$ according to $\Pi_2$ is a resolving partition for $^{u_i}H$ such that the resolvability of any two vertices is done throughout different distances which are only taking the values one or two. Thus, the sets in $\Pi_2^{(i)}$ together with a set formed by a singleton vertex adjacent to every vertex of $^{u_i}H$ form a resolving partition for $^{u_i}H+K_1$. As a consequence, $\pd(^{u_i}H+K_1)\le |\Pi_2^{(i)}|+1$. Notice that each extra set in the described resolving partition for $^{u_i}H+K_1$ corresponding to the singleton vertex is counted for other vertex $u_j$ of $K_n$. As a consequence we obtain that
$$\pd(K_n\circ H)=|\Pi_2|=\sum_{i=1}^n\left(|\Pi_2^{(i)}|+1\right)-n\ge \sum_{i=1}^n \pd(H+K_1)-n=n\cdot \pd(H+K_1)-n.$$
Therefore, the proof is completed.
\end{proof}

\subsection{Strong partition dimension}

The notion of the strong metric dimension was born in connection with finding a combination of the two metric parameters partition dimension and strong metric dimension, and in consequence, to give more insight into both parameters.

A set $W$ of vertices of $G$ {\em strongly resolves} two different vertices $x,y\notin W$ if either $d_G(x,W)=d_G(x,y)+d_G(y,W)$ or $d_G(y,W)=d_G(y,x)+d_G(x,W)$.  An ordered vertex partition $\Pi=\left\{U_1,U_2,...,U_k\right\}$ of a graph $G$ is a {\em strong resolving partition} for $G$, if every two distinct vertices of $G$, belonging to the same set of the partition, are strongly resolved by some set of $\Pi$. A strong resolving partition of minimum cardinality is called a {\em strong partition basis}, and its cardinality the \emph{strong partition dimension} of $G$, denoted by $\pd_s(G)$. The strong partition dimension of graphs was introduced in  \cite{Yero-2014-a}. An example of a strong resolving partition in a graph is given in Figure \ref{Fig:strong-resolving-part}.

\begin{figure}[ht]
\centering
\begin{tikzpicture}[scale=.7, transform shape]
\node [draw, shape=circle,fill=red] (a1) at  (-4,-1) {};
\node [draw, shape=circle,fill=blue] (a2) at  (-4,1) {};
\node [draw, shape=circle,fill=blue] (a3) at  (-3,0) {};
\node [draw, shape=circle,fill=blue] (a4) at  (-1,0) {};
\node [draw, shape=circle,fill=blue] (a5) at  (0,-1) {};
\node [draw, shape=circle,fill=orange] (a6) at  (0,1) {};
\node [draw, shape=circle,fill=green] (a10) at  (4,-1) {};
\node [draw, shape=circle,fill=orange] (a9) at  (4,1) {};
\node [draw, shape=circle,fill=orange] (a8) at  (3,0) {};
\node [draw, shape=circle,fill=orange] (a7) at  (1,0) {};

\draw(a3)--(a1)--(a2)--(a3)--(a4)--(a5)--(a7)--(a8)--(a9)--(a10)--(a8);
\draw(a4)--(a6)--(a7);
\end{tikzpicture}
\caption{A strong resolving partition of the smallest possible cardinality is given. Vertices equally colored belong to a same set in the partition.}\label{Fig:strong-resolving-part}
\end{figure}
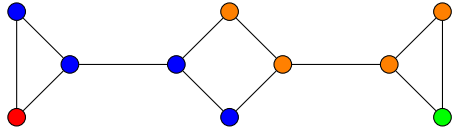

A first natural question on this direction concerns the existence of strong resolving partitions in any connected graphs $G$. This can be easily noted, as in the case of the resolving partitions. If $S=\{v_1,\dots,v_k\}$ is a strong resolving set for $G$, then the partition $\Pi=\{\{v_1\},\dots,\{v_k\},V(G)\setminus S\}$ is readily seen to be a strong resolving partition for $G$. Thus, the following result from \cite{Yero-2014-a} is easy to deduce.

\begin{theorem}{\em \cite{Yero-2014-a}}\label{th:bound-sdim-spd}
For any connected graph $G$, $\pd_s(G)\le \sdim(G)+1$.
\end{theorem}

Similarly to the case of the classical partition dimension, no results concerning its complexity is known, although it seems to be a more tractable problem than that of the partition dimension. For instance, while finding the partition dimension of trees seems to be challenging problem, the strong partition dimension of trees can be easily found, and it is indeed equal to the number of its leaves minus one. This result and several other ones more, that can be easily deduced, are mainly related with the tool of the strong resolving graph of a graph, which is also very powerful for studying the strong metric dimension of graphs, although for the strong partition case, there is not an equality as the one in Theorem \ref{th:strong-dim-cover}. This is based on the following bounds given in \cite{Yero-2014-a} where $\omega(G)$ represents the \emph{clique number} of $G$, \emph{i.e.}, the largest set that induces a complete graph in $G$.

\begin{theorem}{\em \cite{Yero-2014-a}}\label{th:bound-clique-pd-cover}
For any connected graph $G$, $\omega(G_{SR})\le \pd_s(G)\le \alpha(G_{SR})+1$.
\end{theorem}

The upper bound above follows from Theorems \ref{th:strong-dim-cover} and \ref{th:bound-sdim-spd}. Also, the lower bound is obtained since any two mutually maximally distant vertices of a graph $G$ must belong to two different sets in any strong resolving partition for $G$.

Realization results concerning $\sdim(G)$ and $\pd_s(G)$ were first presented in \cite{Yero-2014-a}. There was proved that for any integers $r,t,n$ such that $3\le r\le t\le \frac{n+r-2}{2}$ there exists a connected graph $G$ of order $n$ with $\pd_s(G)=r$ and $\sdim(G)=t$. In connection with this, author of \cite{Yero-2014-a} asked if it is true that $\sdim(G)\le \displaystyle\frac{\pd_s(G)+n-2}{2}$ for every nontrivial connected graph $G$ of order $n$. A negative answer to this question was given in \cite{Kuziak-2020} where it was proved the following.

\begin{theorem}\emph{\cite{Kuziak-2020}}
\label{th:realization-pds-dims}
For any integers $r,t,n$ such that $2\le r\le t\le 2r-3$, there exists a connected graph $G$ of order $n$ with $\pd_s(G)=r$ and $\sdim(G)=t$.
\end{theorem}

To prove the result above authors constructed a graph $G_{a,b,c}$ as follows. Begin with a complete bipartite graph $K_{b,c}$ with bipartition sets $B,C$ such that $|B|=b$ and $|C|=c$. Now, if $a>0$, then construct the graph $G_{a,b,c}$ by adding a path $P_a$ and joining with an edge one of the leaves of $P_a$ with one vertex of the set $C$. If $a=0$, then $G_{a,b,c}$ is simply taken as $K_{b,c}$.

Clearly, Theorem \ref{th:realization-pds-dims} raised up a question on whether $\sdim(G)\le 2\pd_s(G)-3$ for every nontrivial connected graph $G$. Moreover, one could consider the question on whether $\sdim(G)$ can be bounded by a constant factor of $\pd_s(G)$.

The results on the strong partition dimension of graphs are not much, and one can indeed find in the literature only three works dealing with this topic: the seminal article \cite{Yero-2014-a} and \cite{Kuziak-2020,Yero-2016}. The main part of results on this parameter are centered into two directions: bounds and closed formulae for specific families of graph, and the strong partition dimension on some product graphs.

With respect to bounding $\pd_s(G)$, the first bounds that we may consider are the trivial ones. That is, for any connected graph $G$ of order $n$, $2\le \pd_s(G)\le n$. The limit case of this trivial bounds can be easily characterized as follows.

\begin{itemize}
  \item $\pd_s(G)=2$ if and only if $G$ is a path $P_n$.
  \item $\pd_s(G)=n$ if and only if $G$ is a complete graph $K_n$.
\end{itemize}

A characterization of graphs $G$ for which $\pd_s(G)=n$ was given in \cite{Yero-2014-a}. That is, $\pd_s(G)=n-1$ if and only if $G\cong P_3$, $G\cong C_4$, $G\cong K_n-e$\footnote{$K_n-e$ is the graph obtained from $K_n$ by deleting one edge.} or $G\cong K_1+\bigcup_{i}K_{n_i}$, $i>1$, $n_i\ge 1$ for every $i$ and $\sum_{i}n_i=n-1$.

For some other interesting results on $\pd_s(G)$, Table \ref{tab:spd-values-bounds} shows a few of them.

\begin{table}[ht]
  \centering
  \small{
  \begin{tabular}{|c|c|c|}
    \hline
    \textbf{Graphs} $\mathbf{G}$ & $\mathbf{\pd_s(G)}$ & \textbf{Reference} \\ \hline
    Complete graph $K_n$ & $=n$ & \cite{Yero-2014-a} \\ \hline
    Cycle graph $C_n$ & $=3$ & \cite{Yero-2014-a} \\ \hline
    Tree $T$, $l(T)$ leaves & $=l(T)-1$ & \cite{Yero-2014-a} \\ \hline
    Grid graph $P_r\Box P_t$, $r,t\ge 2$ & $=3$ & \cite{Yero-2014-a} \\ \hline
    $\begin{array}{c}
      \mbox{Complete bipartite graph $K_{r,t}$} \\
      1\le r\le t
    \end{array}$ & $\begin{array}{ll}
                         =t, & \mbox{if $r<t$} \\
                         =t+1, & \mbox{if $r=t$}
                       \end{array}$ & \cite{Kuziak-2020} \\ \hline
    Wheel graph $W_{1,r}$, $r\ge 4$ & $\begin{array}{ll}
                         =3, & \mbox{if $r=4$} \\
                         =\left\lceil\frac{r}{2}\right\rceil, & \mbox{if $r\ge 5$}
                       \end{array}$ & \cite{Kuziak-2020} \\ \hline
    Unicyclic graph $G$, $l(G)=1$ & $=3$ & \cite{Yero-2014-a} \\ \hline
    Unicyclic graph $G$, $l(G)\ge 2$ & $\begin{array}{c}
                                                            \ge l(G) \\
                                                            \le l(G)+2
                                                          \end{array}$ & \cite{Yero-2014-a} \\ \hline
    Corona graph $G\odot H$ & $\ge |V(G)| \pd_s(H)$ & \cite{Kuziak-2020} \\ \hline
    Corona graph $G\odot H$, $D(H)=2$ & $=|V(G)| \pd_s(H)$ & \cite{Kuziak-2020} \\ \hline
    Strong product graph $G\boxtimes H$, & $\begin{array}{l}
                                                            \ge \omega(G_{SR})\omega(H_{SR}) \\
                                                            \le |V(G)|\alpha(H_{SR})+|V(H)|\alpha(G_{SR})\\
                                                           \hspace*{0.35cm}-\alpha(G_{SR})\alpha(H_{SR})+1
                                                          \end{array}$ & \cite{Yero-2016} \\ \hline
    Strong product graph $G\boxtimes H$ & $\le \pd_s(G)\pd_s(H)$ & \cite{Yero-2016} \\ \hline
    Strong product of trees $T_1\boxtimes T_2$ & $=l(T_1)l(T_2)$ & \cite{Yero-2016} \\ \hline
    Cartesian product graph $G\Box H$, & $\begin{array}{l}
                                                            \ge \min\{\omega(G_{SR}),\omega(H_{SR})\} \\
                                                            \le \min \{\alpha(G_{SR})|\partial(H)|,|\partial(G)|\alpha(H_{SR})\}+1
                                                          \end{array}$ & \cite{Yero-2016} \\ \hline
  \end{tabular}}
  \caption{Bounds and closed formulae for the strong partition dimension of some graphs.}\label{tab:spd-values-bounds}
\end{table}

\subsection{Local partition dimension}

The local partition dimension of graphs was first introduced under the name of metric chromatic number, and it is indeed the partition version of the local metric dimension of graphs, although the local metric dimension appeared later than the metric chromatic number. It was first introduced in \cite{Chartrand-2009}, and it has been also recently rediscovered under the name of local partition dimension in \cite{Alfarisi-2020}. We need to be careful to not confuse this parameter with a related one called locating-chromatic number first presented in \cite{Chartrand-2002}.

Suppose that $c : V(G) \rightarrow \{1,\dots,k\}$ is a $k$-coloring of a connected graph $G$ for some positive integer $k$ where
adjacent vertices may have the same color and let $V_1,\dots,V_k$ be the resulting color classes. For every vertex $v\in V(G)$, we associate a $k$-vector denoted as $code(v)=(a_1, \dots, a_k)$ and call it the \emph{metric color code} of $v$, such that for every $i\in\{1,\dots,k\}$, we have $a_i = d_G(v, V_i)$. If $code(u)\ne code(v)$ for every two adjacent vertices $u, v\in V(G)$, then $c$ is called a \emph{metric} $k$-\emph{coloring} of $G$. The smallest value $k$ for which $G$ has a metric $k$-coloring is called the \emph{metric chromatic number} of $G$ and was denoted by $\mu(G)$ in \cite{Chartrand-2009}.

Notice that the parameter above can be defined also in the following way. Let $\Pi=\{P_1,\dots,P_k\}$ be an ordered vertex partition of $V(G)$. Then $\Pi$ is a metric $k$-coloring (or a \emph{local resolving partition}) for $G$, if for every pair of adjacent vertices $u,v\in V(G)$ there is a set $P_i\in \Pi$ such that $d_G(u,P_i)\ne d_G(v,P_i)$. In such case, it is said that $u,v$ are \emph{locally resolved} by $P_i$. The cardinality of a smallest local resolving partition for $G$ is the \emph{metric chromatic number} (or the \emph{local partition dimension}) of $G$. In order to be consequent with the remaining terminology on resolving partitions for graphs, from now on, we shall use the terms local resolving partitions and local partition dimension, and denote the parameter as $\pd_l(G)$. Figure \ref{Fig:local-resolving-part} shows a graph with a local resolving partition of minimum cardinality (3), where vertices equally colored belong to a same set of the partition.

\begin{figure}[ht]
\centering
\begin{tikzpicture}[scale=.7, transform shape]
\node [draw, shape=circle,fill=red] (a1) at  (-4,-1) {};
\node [draw, shape=circle,fill=blue] (a2) at  (-4,1) {};
\node [draw, shape=circle,fill=red] (a3) at  (-3,0) {};
\node [draw, shape=circle,fill=blue] (a4) at  (-1,0) {};
\node [draw, shape=circle,fill=blue] (a5) at  (0,-1) {};
\node [draw, shape=circle,fill=blue] (a6) at  (0,1) {};
\node [draw, shape=circle,fill=green] (a10) at  (4,-1) {};
\node [draw, shape=circle,fill=blue] (a9) at  (4,1) {};
\node [draw, shape=circle,fill=green] (a8) at  (3,0) {};
\node [draw, shape=circle,fill=blue] (a7) at  (1,0) {};

\draw(a3)--(a1)--(a2)--(a3)--(a4)--(a5)--(a7)--(a8)--(a9)--(a10)--(a8);
\draw(a4)--(a6)--(a7);
\end{tikzpicture}
\caption{A local resolving partition of the smallest possible cardinality is given. Vertices equally colored belong to a same set in the partition.}\label{Fig:local-resolving-part}
\end{figure}
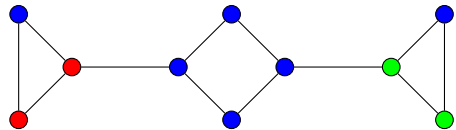

Similarly to the case of the related resolving partitions, one can easily construct a local resolving partition for any connected graph $G$, by just considering any local resolving set $S=\{v_1,\dots,v_k\}$ and making the the partition $\Pi=\{\{v_1\},\dots,\{v_k\},V(G)\setminus S\}$. This clearly leads to an analogous result to \eqref{eq:pd-dim} and Theorem \ref{th:bound-sdim-spd}.

\begin{theorem}
For any connected graph $G$, $\pd_l(G)\le \dim_{\ell}(G)+1$.
\end{theorem}

Thinking into the metric coloring terminology it is natural to consider the relationship between $\pd_l(G)$ and the chromatic number $\chi(G)$. One can readily see that any proper coloring of a graph $G$ induces a local resolving partition since any two adjacent vertices have different colors, and thus the color classes which they belong resolve them. On the other hand, it is clear that any resolving partition is also a local resolving partition. These facts mean that for any connected graph $G$,
\begin{equation}\label{eq:chromatic-lpd}
\pd_l(G)\le \min\{\pd(G),\chi(G)\}\;\;\mbox{\cite{Chartrand-2009}}
\end{equation}
A natural question would be then characterizing the families of graphs $G$ for which $\pd_l(G)=\pd(G)$ or $\pd_l(G)=\chi(G)$. It is also remarkable the fact that no computational results are known for this parameter, although for some basic families of graphs it is known to be easier to manage. For instance, the case of bipartite graphs is an interesting one, since we can take a partition into two independent sets that will clearly form a local resolving partition, and so $\pd_l(G)=2$ for any bipartite graph $G$. Indeed, these are the only graphs having local partition dimension two.  This was proved in \cite{Chartrand-2009}.

\begin{proposition}\emph{\cite{Chartrand-2009}}
\label{prop:pdl-2}
A nontrivial connected graph $G$ has local partition dimension 2 if and only if $G$ is bipartite.
\end{proposition}

An immediate consequence of the Proposition above and \eqref{eq:chromatic-lpd} is as follows. If $G$ is a connected graph with $\chi(G)=3$, then $\pd_l(G)=3$.

Trivial lower and upper bounds for the local partition dimension of a graph $G$ of order $n$ are clearly 2 and $n$, respectively. The graphs with $\pd_l(G)=2$ are characterized in Proposition \ref{prop:pdl-2}, while it is obvious that a connected graph $G$ of order $n$ has local partition dimension $n$ if and only if $G$ is the complete graph $K_n$. In connection with this latter fact, the graphs $G$ for which $\pd_l(G)=n-1$ where also characterized in \cite{Chartrand-2009}. Other bounds for the local partition dimension of graphs are given in Table \ref{tab:bounds-lpd}.

\begin{table}[ht]
  \centering
  \begin{tabular}{|c|c|c|}
    \hline
    \textbf{Graphs} $\mathbf{G}$ & $\mathbf{\pd_l(G)}$ & \textbf{Reference} \\ \hline
    Graph $G$, clique number $\omega(G)$  & $\ge 1+\log_2\omega(G)$ & \cite{Chartrand-2009} \\ \hline
    Graph $G$, order $n$, diameter $d$  & $\le n-d+1$ & \cite{Chartrand-2009} \\ \hline
  \end{tabular}
  \caption{Bounds for the local partition dimension of graphs.}\label{tab:bounds-lpd}
\end{table}

Realization results concerning the existence of graphs with given value of the local partition dimension were presented in \cite{Chartrand-2009}, where was proved for instance that for each pair $k, n$ of integers with $2\le k \le n$, there is a connected graph $G$ of order $n$ with $\pd_l(G) = k$. Another related result from \cite{Chartrand-2009}, which now combines the local partition dimension with the chromatic number is as follows.

\begin{theorem}\emph{\cite{Chartrand-2009}}
For each pair $a, b$ of integers with $2 \le a \le b \le 2^{a-1}$, there exists a connected graph $G$ with $\pd_l(G) = a$ and $\chi(G) = b$.
\end{theorem}

In addition to the above mentioned results, there are some studies on computing the value of the local partition dimension of some families of graphs. Remarkable contributions on this are given in Table \ref{tab:formulas-lpd}.

\begin{table}[ht]
  \centering
  \begin{tabular}{|c|c|c|}
    \hline
    \textbf{Graphs} $\mathbf{G}$ & $\mathbf{\pd_l(G)}$ & \textbf{Reference} \\ \hline
    Cycle graphs $C_n$ & $\begin{array}{cc}
                            2, & \mbox{if $n$ is even}\\
                            3, & \mbox{if $n$ is odd}
                          \end{array}
    $ & \cite{Chartrand-2009} \\ \hline
    Complete $k$-partite graphs $G$, $k\ge 2$ & $k$ & \cite{Chartrand-2009} \\ \hline
  \end{tabular}
  \caption{Exact values of the local partition dimension of some graphs.}\label{tab:formulas-lpd}
\end{table}

\subsection{Some open problems}

\begin{itemize}
  \item Which is the complexity of computing the (strong) partition dimension of graphs?
  \item Is it the case that computing the partition dimension of trees is NP-hard? (notice that the strong partition dimension and the metric chromatic number of trees can be polynomially computed).
  \item Since the partition dimension of Cartesian, lexicographic and strong products of graphs has already been studied, it would be interesting to consider the partition dimension of the last standard product, namely, the direct product of graphs.
  \item Is it true that for every spanning tree $T$ of a unicyclic graph $G$, it follows $\pd(G)\le \pd(T)+1$? (Conjecture \ref{conj:pd-tree-unicyc})
  \item Is it true that $\sdim(G)\le 2\pd_s(G)-3$ for every nontrivial connected graph $G$?
  \item Characterize the families of graphs $G$ for which $\pd_l(G)=\pd(G)$ or $\pd_l(G)=\chi(G)$.
\end{itemize}

%% file: edge-dims.tex
\section{Identification of edges versus metric dimension}

The idea of uniquely identifying the edges of a graph was born in connection with watching the lines that could connect pairs of points in a network. It is widely known (and indeed used) that a resolving set has the capability of uniquely identifying the vertices of a network under a distance vertex framework. In \cite{Kelenc-2018}, authors wondered on whether such resolving sets are also able to uniquely recognizing the connections between the vertices. The answer to such question was negative, and for instance they gave the following example. The graph of Figure \ref{example1} shows a graph, where no metric basis uniquely recognizes all the edges.

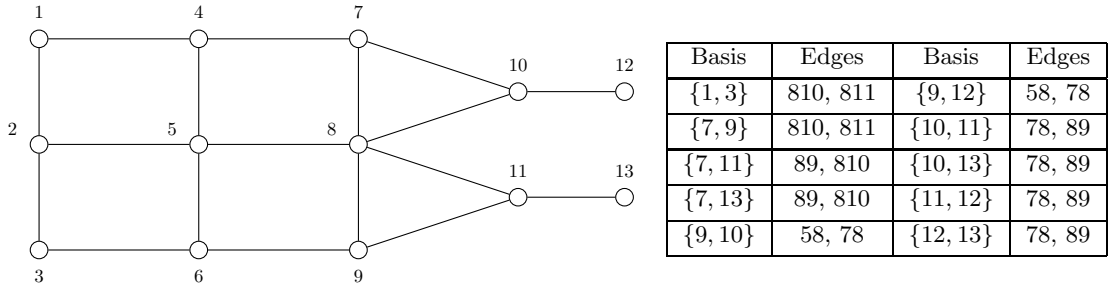
\begin{figure}[ht]
\begin{minipage}[b]{.4\linewidth}
\centering
\begin{tikzpicture}[scale=.7, transform shape]
\node [draw, shape=circle] (a3) at  (0,0) {};
\node [draw, shape=circle] (a2) at  (0,2) {};
\node [draw, shape=circle] (a1) at  (0,4) {};
\node [draw, shape=circle] (a6) at  (3,0) {};
\node [draw, shape=circle] (a5) at  (3,2) {};
\node [draw, shape=circle] (a4) at  (3,4) {};
\node [draw, shape=circle] (a9) at  (6,0) {};
\node [draw, shape=circle] (a8) at  (6,2) {};
\node [draw, shape=circle] (a7) at  (6,4) {};
\node [draw, shape=circle] (a11) at  (9,1) {};
\node [draw, shape=circle] (a10) at  (9,3) {};
\node [draw, shape=circle] (a13) at  (11,1) {};
\node [draw, shape=circle] (a12) at  (11,3) {};

\draw(a1)--(a2)--(a3)--(a6)--(a5)--(a4)--(a7)--(a8)--(a9)--(a11)--(a8)--(a10)--(a7);
\draw(a2)--(a5)--(a8);
\draw(a1)--(a4);
\draw(a6)--(a9);
\draw(a10)--(a12);
\draw(a11)--(a13);

\node at (0,-0.5) {$3$ };
\node at (-0.5,2.3) {$2$ };
\node at (0,4.5) {$1$ };
\node at (3,-0.5) {$6$ };
\node at (2.5,2.3) {$5$ };
\node at (3,4.5) {$4$ };
\node at (6,-0.5) {$9$ };
\node at (5.5,2.3) {$8$ };
\node at (6,4.5) {$7$ };
\node at (9,1.5) {$11$ };
\node at (9,3.5) {$10$ };
\node at (11,1.5) {$13$ };
\node at (11,3.5) {$12$ };
\end{tikzpicture}
\par\vspace{0pt}
\end{minipage}
\begin{minipage}[b]{.75\linewidth}
\centering
\small{
\begin{tabular}{|c|c|c|c|}
  \hline
 Basis & Edges & Basis & Edges \\[2pt]  \hline
 $\{1,3\}$ & $810$, $811$ & $\{9,12\}$ & $58$, $78$ \\[2pt] \hline
 $\{7,9\}$ & $810$, $811$ & $\{10,11\}$ & $78$, $89$ \\[2pt]  \hline
 $\{7,11\}$ & $89$, $810$ & $\{10,13\}$ & $78$, $89$ \\[2pt]  \hline
 $\{7,13\}$ & $89$, $810$ & $\{11,12\}$ & $78$, $89$ \\[2pt]  \hline
 $\{9,10\}$ & $58$, $78$ & $\{12,13\}$ & $78$, $89$ \\[2pt]  \hline
\end{tabular}}
\par\vspace{12pt}
\end{minipage}
\caption{A graph where any metric basis does not recognizes all edges and a table with all metric bases and two edges which are not recognized by the corresponding metric basis.}
\label{example1}
\end{figure}

On the other hand, one could think on the opposite way: that having a set of vertices of a graph that uniquely identifies the edges will also uniquely identify the vertices. However, this is far from reality, and just a simple example of this is the hypercube graph $Q_4$ that can recognize all the edges with just 3 vertices, and whose metric dimension is 4.

In this sense, the idea of identifying the edges of the graphs and the introduction of the concept of edge metric dimension seemed to be worth of studying them, and from the seminal paper \cite{Kelenc-2018} on, the edge metric dimension has indeed become into one of the most popular topics on metric dimension in graphs. In fact, there have appeared several variations of it which has enriched the theory of identifying the edges of graphs. Some examples are for instance, the mixed metric dimension \cite{Kelenc-2017}, the incidence dimension \cite{Bozovic-2018}, the local edge metric dimension \cite{Adawiyah-2020}, the fault tolerant edge metric dimension \cite{Liu-2021}, and the fractional edge metric dimension \cite{Yi-2021}. We must remark that the concept of edge metric dimension from \cite{Kelenc-2018} is not the same as that one appeared in \cite{Nasir-2019} which is nothing more but the classical metric dimension of the line graph of a graph, that is, the notion of edges identifying edges. The parameter defined in \cite{Nasir-2019} was first called edge metric dimension, but after renamed as edge version of metric dimension in \cite{Liu-2018}.

\subsection{Edge metric dimension}

Formally, given a connected graph $G$, a vertex $v\in V$ and an edge $e=xy\in E(G)$, the \emph{distance} between $v$ and $e$ is $d_G(e,v)=\min\{d_G(x,v),d_G(y,v)\}$. Now, a given set of vertices $S\subset V(G)$ is an \emph{edge resolving set} of $G$ if for any pair of distinct edges $e,f\in E(G)$ there is a vertex $v\in S$ such that $d_G(e,v)\ne d_G(f,v)$. An edge resolving set having the smallest possible cardinality is called an edge metric basis and the edge metric dimension of $G$ is the cardinality of any edge metric basis, denoted as $\edim(G)$. Concepts above were first defined in \cite{Kelenc-2018} where edge resolving sets were called \emph{edge metric generators}. Figure \ref{Fig:edge-dim} shows an example of a graph with two edge resolving sets colored, one in red color, a second in blue. This sample graph was computationally found in \cite{Knor-2021}, which is indeed a subdivision of the graph $K_4$ (the complete graph $K_4$ with all the edges subdivided by a vertex). We moreover note that any two vertices of degree two having a common neighbor form an edge metric basis for such graph.

\begin{figure}[h]
\centering
\begin{tikzpicture}[scale=.7, transform shape]
\node [draw, shape=circle] (a1) at  (0,0) {};
\node [draw, shape=circle,fill=blue] (a5) at  (0,1.5) {};
\node [draw, shape=circle] (a9) at  (0,3) {};
\node [draw, shape=circle,fill=red] (a2) at  (4.5,0) {};
\node [draw, shape=circle] (a6) at  (2.5,1.5) {};
\node [draw, shape=circle,fill=blue] (a10) at  (4.5,3) {};
\node [draw, shape=circle] (a7) at  (6.5,1.5) {};
\node [draw, shape=circle] (a4) at  (9,0) {};
\node [draw, shape=circle,fill=red] (a8) at  (9,1.5) {};
\node [draw, shape=circle] (a12) at  (9,3) {};

\draw(a1)--(a2)--(a4)--(a8);
\draw(a1)--(a5)--(a9)--(a10)--(a12)--(a8);
\draw(a9)--(a6)--(a4);
\draw(a1)--(a7)--(a12);
\end{tikzpicture}
\caption{Two edge resolving sets of the smallest possible cardinality are drawn: one red colored and a second one blue colored.}\label{Fig:edge-dim}
\end{figure}
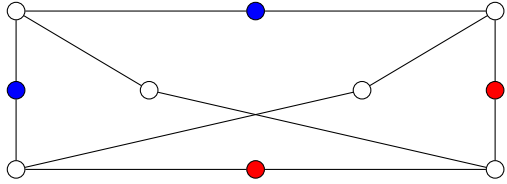

The decision problem concerning computing the edge metric dimension of graphs was proved to be NP-complete in \cite{Kelenc-2018}. The proof of this was highly influenced by the NP-completeness reduction from \cite{Khuller-1996} which uses the 3-SAT problem to achieve the result. On the other hand, it was also proved in \cite{Kelenc-2018}, that the problem of computing the edge metric dimension of graphs can be approximated within a factor of $O(\log m)$ in polynomial time, where $m$ is the number of edges of the graph, by reducing such problem to an instance of the set cover problem. Other computational aspects of computing the edge metric dimension of graphs are given in \cite{Huang-2021}. In such work, authors established a potential function and gave a corresponding greedy algorithm with approximation ratio $1 + \ln n + \ln(\log_2 n)$ for approximating $\edim(G)$, where $n$ is the number of vertices of the graph $G$.

\subsubsection{Bounding the edge metric dimension}

Based on the NP-hardness of finding the edge metric dimension of graphs, it is natural to consider bounding it for general graphs, as well as, for particular classes, or by introducing some restrictions on the graphs. The first trivial bounds are clearly as follows (given in \cite{Kelenc-2018}). For any graph $G$ of order $n$,
$$1\le \edim(G)\le n-1.$$
Similarly to the case of the classical metric dimension, it is easy to deduce that $\edim(G)=1$ if and only if $G$ is a path of order at least two. In contrast to this fact, characterizing the graphs $G$ for which $\edim(G)=n-1$ turned out to be a challenging problem (note that graphs of order $n$ with metric dimension $n-1$ are only the complete graphs - see \cite{Chartrand-2000-a}). In \cite{Kelenc-2018}, only a few independent necessary or sufficient conditions for having $\edim(G)=n-1$ were given. For instance, there was proved that if $G$ is a connected graph of order $n$ and $\edim(G)=n-1$, then for every $u, v \in V(G)$, $u \neq v$, it holds $N_G(u) \cap N_G(v) \neq \emptyset$. Also, for the graph $G$, if there is a vertex $v \in V(G) $ of degree $n-1$, then either $\edim(G) = n-1$ or $\edim(G)=n-2$. Moreover, if there are two distinct vertices $u, v \in V(G) $ of degree $n-1$, then $\edim(G) = n-1$ as well. In the same direction, the following was shown in \cite{Zubrilina-2018}.

\begin{theorem}\emph{\cite{Zubrilina-2018}}
Let $G$ be a graph of order $n$. Then $\edim(G)=n-1$ if and only if for any distinct $u,v\in V(G)$ there exists $w\in V(G)$ such that $uw,vw\in E(G)$ and $w$ is adjacent to all non-mutual neighbors of $u,v$.
\end{theorem}

Despite the characterization is true, it lacks of practical usefulness since it does not give a clear idea on the structure of the graphs achieving these properties. Another approach to the same problem is presented in \cite{Zhu-2019}, where graphs $G$ with $\edim(G)=n-1$ are called \emph{topfull graphs}. For instance, it is first noted that a graph $G$ of order $n$ with $\edim(G)=n-1$ is $2$-connected, and they characterized all the $2$-connected graphs $G$ with $\edim(G)=n-1$ as follows.

\begin{theorem}\emph{\cite{Zhu-2019}}
Let $G$ be a 2-connected graph of order $n$ and let $\{u,v\}$ be a vertex cut of $G$. Then, $\edim(G)=n-1$ if and only if $uv\in E(G)$ and every vertex in $V(G)\setminus\{u,v\}$ is adjacent to both $u$ and $v$.
\end{theorem}

For a connected graph $G$ in general, it is noted in \cite{Zhu-2019} that if $\overline{G}$ has at least three components, then $\edim(G)=n-1$. In addition, some other results in the same style of all these characterizations above are given \cite{Zhu-2019}. Although, all of these characterizations are not exactly of the highest usefulness, there is some remarkable contribution in \cite{Zhu-2019}. The authors gave an algorithm which checks in polynomial time ($\mathcal{O}(n^3)$) whether a given graph of order $n$ is topfull or not. In addition to all the results above, another similar characterization was given in \cite{Geneson-2020}, where it is also noted that a connected graph $G$ such that $\edim(G)=n-2$ satisfies that it has diameter at most 5. With respect to graphs achieving such equality, in \cite{Wei-2020} was given a characterization of all connected bipartite graphs with $\edim=n-2$, which partially answers an open problem from \cite{Zubrilina-2018}. That is next stated.

\begin{theorem}\emph{\cite{Wei-2020}}
Let $G$ be a connected bipartite graph of order $n\ge 3$. Then $\edim(G) = n-2$ if and only if $G$ is a complete bipartite graph
\end{theorem}

Moreover, it was also presented in \cite{Wei-2020} a sufficient and necessary condition for $\edim(G)=n-2$, whether $G$ is a graph with maximum degree $n-1$.

Now, with respect to bounding the edge metric dimension of graphs a few contributions are known. We recall some of them in the next table.

\begin{table}[ht]
  \centering
  \begin{tabular}{|c|c|c|c|}
    \hline
    $\geq \left\lceil \log_2{\Delta(G)} \right\rceil$ & \cite{Kelenc-2018} & $\ge 1 + \left\lceil \log_2\delta(G)\right\rceil$ & \cite{Filipovic-2019} \\ \hline
    $\edim(Q_n) \leq n$ & \cite{Kelenc-2018} & $\edim(\prod_{i=1}^{n}P_{r_i})\le n$ & \cite{Geneson-2020} \\
    \hline
  \end{tabular}
  \caption{Bounds for the edge metric dimension of graphs}\label{tab:edge-dim-bound}
\end{table}

Another style of bound relating the edge metric dimension was given in \cite{Kelenc-2018} as follows: If $\edim(G)=k$ and $G$ has diameter $D$, then $|E(G)| \leq (D+1)^k$. Such bound for $|E(G)|$ was improved to $\binom{k}{2}+kD^{k-1}+D$ in \cite{Zubrilina-2018}, and in turn, this latter one was also improved to $(\left\lfloor \frac{2D}{3} \right\rfloor + 1)^k + k \sum_{i = 1}^{\left\lceil \frac{D}{3} \right\rceil} (2i)^{k-1}$ in \cite{Geneson-2020}.

With respect to this style of result, bounding the number of edges in a graph $G$ with a given diameter and edge metric dimension, we can find some related results in \cite{Geneson-2020,Geneson-2021}, where a few other properties of the graph are bounded in relation to the edge metric dimension, when some specific patterns are avoided. For instance, we remark the following ones.

\begin{theorem}\emph{\cite{Geneson-2020}}
The maximum possible value of $n$ for which some graph of edge metric dimension $\le k$ contains $K_{1,n}$ as a subgraph is $n=2k$.
\end{theorem}

\begin{theorem}\emph{\cite{Geneson-2020}}
The maximum possible value of $n$ for which some graph of edge metric dimension $\le k$ contains $K_{n,n}$ as a subgraph is between $2^{\lfloor k/2\rfloor}$ and $3^{k/2}$.
\end{theorem}

\begin{theorem}\emph{\cite{Geneson-2020}}
The maximum possible degeneracy of any graph of edge metric dimension $\le k$ is  between $2^{\lfloor k/2\rfloor-1}$ and $2^k$
\end{theorem}

\begin{theorem}\emph{\cite{Geneson-2021}}
The maximum possible clique number of a graph of edge metric dimension at most $k$ is $2^{\Theta(k)}$.
\end{theorem}

We close this subsubsection with some asymptotical bounds for the edge metric dimension of the Erd\"os-R\'enyi random graph $G(n, p)$ with constant $p$ given in \cite{Zubrilina-2021}.

\begin{theorem}{\em \cite{Zubrilina-2021}}
Let $G(n, p)$ be the Erd\"os-R\'enyi random graph with constant $p$. Then
$$\edim(G(n, p))=(1+o(1))\frac{4\log n}{\log(1/q)},$$ where $q = 1-2p(1-p)^2(2-p)$.
\end{theorem}

The proof of such result uses several probabilistic techniques and is highly influenced by a similar result for the classical metric dimension of random graphs given in \cite{Bollobas-2012}.

\subsubsection{Comparing $\dim(G)$ and $\edim(G)$}

Although metric dimension and edge metric dimension of graphs are in general not comparable, one could consider bounding one of them in terms of the other one by using some linear dependence. However, such kind of dependence is in general not possible as it was first proved in \cite{Zubrilina-2018} with the next interesting result.

\begin{theorem}\emph{\cite{Zubrilina-2018}}
\label{th:edim-dim-unbound}
The ratio $\frac{\edim(G)}{\dim(G)}$ is not bounded from above.
\end{theorem}

To prove the result above, the author constructed a graph $F_k$ with $\edim(F_k)=k+2^k-2$, and $\dim(F_k)=k$. Such graph is defined as follows. For a positive integer $k$, let $F_k$ be the graph on the vertex set $A\cup B$, where $B=\{b_1,\dots, b_k\}$ and $A=\{a_S\,|\, S\subseteq B\}$. Let $b_i,b_j$ be adjacent for all $b_i,b_j\in B$ with $b_i\ne b_j$, and let $a_S,a_T$ be adjacent for all $a_S,a_T\in A$ with $a_S\ne a_T$. For any $b_i\in B$, $a_S\in A$ let $b_i,a_S$ be adjacent if and only if $b_i\in S$.

Based on Theorem \ref{th:edim-dim-unbound}, the author of \cite{Zubrilina-2018} raised up a question on the existence of graph $G$ for which $\edim(G)\gg 2^{\dim(G)}$. Such questions was positively answered in \cite{Geneson-2021}.

On the reciprocal situation (with respect to the ratio of Theorem \ref{th:edim-dim-unbound}, the existence of graphs for which $\edim(G)\ll \dim(G)$ seemed to be harder to settle till recently, since only one example of a family of graphs with $\edim(G)< \dim(G)$ was known, \emph{i.e.}, the class of torus graphs $C_{4r}\Box C_{4t}$ with $r,t\ge 1$, for which $3=\edim(C_{4r}\Box C_{4t})< \dim(C_{4r}\Box C_{4t})=4$, already proved in \cite{Kelenc-2018}. Very recently, in \cite{Knor-2021}, all the smallest graphs $G$ (which turned up to have 10 vertices - one of them appears in Figure \ref{Fig:edge-dim}) for which $\edim(G)< \dim(G)$ were computationally found. Moreover, it was also proved there, a symmetrical result to that in \cite{Zubrilina-2018}, which is next appearing.

\begin{theorem}\emph{\cite{Knor-2021}}
The ratio $\frac{\dim(G)}{\edim(G)}$ is not bounded from above.
\end{theorem}

To prove this result authors were required first to prove the following realization result. We may recall that from this result, it can be also deduced the result of \cite{Zubrilina-2018} concerning the unboundness of the ratio $\frac{\edim(G)}{\dim(G)}$.

\begin{theorem}\emph{\cite{Knor-2021}}
\label{th:realization-dim-edim}
Let $k_1,k_2\ge 2$ and $k_1\ne k_2$. Then there is an integer $n_0$ such that for every $n\ge n_0$ there exists a graph on
$n$ vertices with $\dim(G)=k_1$ and $\edim(G)=k_2$.
\end{theorem}

The result above required to construct the following family of graphs. We begin with a cycle $C$ on $n_1$ vertices, where $n_1\ge 5$.
We denote the vertices of $C$ consecutively by $a_1,a_2,\dots, a_{n_1}$. Further, take a path $P$ on $n_2$ vertices denoted consecutively by
$b_1,b_2,\dots,b_{n_2}$, where $n_2\ge 1$, and join $P$ to $C$ by the edge $a_2b_1$. Then take vertices $c$ and $i$ and connect them by edges to $a_{n_1}$ and $a_1$, respectively. Finally, take $n_3$ vertices $j_1,j_2,\dots,j_{n_3}$, where $n_3\ge 2$, and join them by edges to the vertex $i$. We denote the resulting graph by $G_{n_1,n_2,n_3}$. With ``several'' of such graphs, authors of \cite{Knor-2021} used a ``kind of concatenation'' of them, by adding a few edges between them, in order to finally construct a graph with the necessary requirements of the result.

We must remark that the graphs constructed in \cite{Knor-2021} have cut vertices. In this sense, the work \cite{Knor-2021+} was then centered into making some construction of graphs $G$ that also realize metric dimension and edge metric dimension (like in Theorem \ref{th:realization-dim-edim}) but for $2$-connected graphs. Along the way, the (edge) metric dimension of subdivisions graphs\footnote{A subdivision graph $S(G)$ is a graph obtained from $G$ by subdividing all its edges once.} $S(G)$, with emphasis in the subdivision graph of complete graphs minus a matching was also studied in \cite{Knor-2021+}. For instance, the next bound was given there.

\begin{theorem}\emph{\cite{Knor-2021+}}
Let $G$ be a graph on $n$ vertices.
If $G$ contains $\lfloor\frac{n-1}3\rfloor$ vertex-disjoint paths of length
$2$, then $\edim(S(G))\le\lceil\frac{2n-2}3\rceil$.
\end{theorem}

An interesting relationship between metric dimension and edge metric dimension was given in \cite{Kelenc-2021}. There was proved that if $G$ is a connected bipartite graph, then every resolving set for $G$ is also an edge resolving set for $G$, and where a crucial property that allowed to make the proof is based on the non existence of closed walk of odd length in a bipartite graph. As a consequence of this interesting property, it is deduced that for any connected bipartite graph $G$,
\begin{equation}
\label{eq:edim-dim-bipartite}
  \edim(G)\le \dim(G).
\end{equation}

We end this subsubsection by including some examples of graphs $G$ for which $\edim(G)=\dim(G)$. Table \ref{tab:dim-edim-equal} contains them, together with the references where they were published.

\begin{table}[ht]
  \centering
  \small{
  \begin{tabular}{|c|c|c|c|}
    \hline
    Cycles & \cite{Kelenc-2018} & Complete graphs & \cite{Kelenc-2018} \\ \hline
    Trees & \cite{Kelenc-2018} & Grid graphs $P_r\Box P_t$ & \cite{Kelenc-2018} \\ \hline
    Complete bipartite graphs & \cite{Kelenc-2018} & Hypercubes $Q_{2t+1}$, $t\ge 1$ & \cite{Kelenc-2021} \\ \hline
    Circulant graphs $C_n(1,3)$, $n\equiv 0,3,4$ (mod 6) & \cite{Ahsan-2020} & Some families of unicyclic graphs & \cite{Sedlar-2021-c} \\
    \hline
  \end{tabular}}
  \caption{Graphs having equal values in their metric and edge metric dimensions.}\label{tab:dim-edim-equal}
\end{table}

With respect to the hypercube $Q_d$, as stated in Table \ref{tab:dim-edim-equal}, there are cases in which $\dim(Q_d)=\edim(Q_d)$ (when $d$ is odd). For the remaining cases ($d$ is even), it was noted in \cite{Kelenc-2021} that
\begin{equation}\label{eq:dim-edim-Q_d}
  \edim(Q_d)\le\dim(Q_d)\le\edim(Q_d)+1.
\end{equation}
However, it is only known one case in which $\dim(Q_d)=\edim(Q_d)+1$, that is whether $d=4$. Small even values of $d$ (6 and 8) satisfy also that $\edim(Q_d)=\dim(Q_d)$, which has been computationally checked.

\subsubsection{Edge metric dimension of some graphs}

We now center our attention into those investigations centered into computing the value of the edge metric dimension of some graphs. This is summarized in Table \ref{tab:edim-values}, where we include the references where the corresponding graphs have been studied.

\begin{table}[ht]
  \centering
  \small{
  \begin{tabular}{|c|c|c|}
    \hline
    \textbf{Graphs} $\mathbf{G}$ & $\mathbf{\edim(G)}$ & \textbf{Reference} \\ \hline
    Complete bipartite graph $K_{r,t}$ & $r+t-2$ & \cite{Kelenc-2018} \\ \hline
    Tree $T$ & $l(T)-ex(T)$ & \cite{Kelenc-2018} \\ \hline
    Grid graph $P_r\Box P_t$ & 2 & \cite{Kelenc-2018} \\ \hline
    $d$-dimensional grid $P_n^d$, $n\ge d^{d-1}$ & d & \cite{Geneson-2021} \\ \hline
    Torus graph $C_{4r}\Box C_{4t}$ & 3 & \cite{Kelenc-2018} \\ \hline
    Web graph $\mathbb{W}_n$ & 3 & \cite{Zhang-2020} \\ \hline
    Convex polytope $\mathbb{D}_n$ & 3 & \cite{Zhang-2020} \\ \hline
    Other convex polytope related graphs & Several formulas & \cite{Ahsan-2021} \\ \hline
    M\"obius ladder network $MB_{\Psi}$, $\Psi\ge 3$ & 4 & \cite{Deng-2021} \\ \hline
    Hexagonal M\"obius ladder network $HMB_{\Psi}$, $\Psi\ge 2$ & 3 & \cite{Deng-2021} \\ \hline
    Jahangir graph $\mathcal{J}_{2n}$ & $\left\lfloor\frac{2n}{3}\right\rfloor$ & \cite{Yang-2019} \\ \hline
    Helm, Sunflower and Friedship graphs & Several formulas  & \cite{Yang-2019} \\ \hline
    Wheel graph $W_{1,n}$ & $\begin{array}{ll} n, & n = 3, 4,\\ n-1, & n \geq 5.  \end{array}$ & \cite{Kelenc-2018} \\ \hline
    Graph $G+K_1$, $|V(G)|=n$ & $\begin{array}{l} \le n, \\ \ge n-1  \end{array}$ & \cite{Zubrilina-2018} \\ \hline
    Join graph $G+H$ & $\begin{array}{l} \le |V(G)|+|V(H)|-1, \\ \ge |V(G)|+|V(H)|-2  \end{array}$ & \cite{Peterin-2020} \\ \hline
    Complete multipartite graph $K_{r_1,\ldots,r_t}$ & $\begin{array}{ll}
r_1+r_2-2, & \mbox{if $t=2$} \\
\sum_{i=1}^tr_i-1, & \mbox{if $t>2$} \\ %
\end{array}$ & \cite{Peterin-2020} \\ \hline
    Corona graph $G\odot H$, $|V(H)|\geq 2$ & $|V(G)|\cdot(|V(H)|-1)$ & \cite{Peterin-2020} \\ \hline
    Antiprism graph $A_n$, $n\ge 3$ & $\begin{array}{ll}
4, & \mbox{if $n$ is even} \\
5, & \mbox{if $n$ is odd} \\ %
\end{array}$ & \cite{Zhang-2020} \\ \hline
    Prism related graph $D^*_n$, $n\ge 3$ & $\begin{array}{ll}
4, & \mbox{if $n=3$ or $n=4$} \\
\left\lceil n/2\right\rceil+1, & \mbox{otherwise} \\ %
\end{array}$ & \cite{Zhang-2020} \\ \hline
    Generalized Petersen graph $G(n,1)$ & 3 & \cite{Filipovic-2019} \\ \hline
    Generalized Petersen graph $G(n,2)$ & $\begin{array}{ll} 3, & n = 8\mbox{ or }n \ge 10\\ 4, & n \in \{5, 6, 7, 9\}  \end{array}$ & \cite{Filipovic-2019} \\ \hline

  \end{tabular}
  }
  \caption{Graphs for which their edge metric dimension has been studied.}\label{tab:edim-values}
\end{table}

\subsubsection{Miscellaneous results}

Another topic of interest concerns the existence of a linear programming model for the edge metric dimension. This was first presented in the Ph. D. dissertation \cite{Kelenc-2019}, and rediscovered in \cite{Klavzar-2021++}. A similar model for the metric dimension is known from \cite{Chartrand-2000-a}.

Let $G$ be a graph of order $n$ and size $m$ with vertex set $V=\{v_1,\dots,v_n\}$ and edge set $E=\{e_1,\dots,e_m\}$. We consider the matrix $D=[d_{ij}]$ of order $m\times n$ such that $d_{ij}=d_G(x_i,x_j)$, $x_i\in V$ and $x_j\in E$. Now, given the variables $y_j\in \{0,1\}$ with $j\in \{1,2,\dots,n\}$ we define the following function:
$$\mathcal{F}(y_1,y_2,\dots,y_{n})=y_1+y_2+\dots+y_{n}.$$
Clearly, minimizing the function $\mathcal{F}$ subject to the following constraints
$$\sum_{i=1}^{n}|d_{ji}-d_{li}|y_i\ge 1\;\;\mbox{for every $1\le j<l\le m$},$$
is equivalent to finding an edge metric basis of $G$, since the solution for $y_{1}, y_{2},\dots, y_{n}$ represents a set of values for which the function $\mathcal{F}$ achieves the minimum possible.

The model above was applied in \cite{Klavzar-2021++} for computing the edge metric dimension of some interesting graphs, like for instance the bridge-cycle graphs, and some chemical graphs, like some fullerenes.

Some other interesting results on the edge metric dimension of graphs concerns that one on product graphs. Although, there is not much on this direction, we recall the following ones from \cite{Peterin-2020} for the lexicographic product graphs $G\circ H$.

\begin{theorem}\emph{\cite{Peterin-2020}}
Let $G$ be any graph with at least three vertices in every component and let $H\ncong K_1$ be a graph. Then
$$\edim(G\circ H)\ge |V(G)|(|V(H)|-1)+f'(G)+t'(G)+q'(G').$$
Moreover, if $H\notin\mathcal G$, then
$$\edim(G\circ H)= |V(G)|(|V(H)|-1)+f'(G)+t'(G)+q'(G').$$
\end{theorem}

The notations $f'(G)$, $t'(G)$, and $q'(G')$ are parameters which depends on false and true twins of $G$. Since their definition required a little much effort (which is a little beyond the goal of this survey), we suggest the reader to directly check them in \cite{Peterin-2020}, as well as, for the definition of the family $\mathcal G$.

Other examples of product graphs, like corona, join and grid graphs, that have been studied in the literature, already appear in Table \ref{tab:edim-values}. However, there still remain the case of corona graph $G\odot K_1$, for which is known from \cite{Peterin-2020} the following contributions.

\begin{theorem}\emph{\cite{Peterin-2020}}
For any graph $G$, $\edim(G\odot K_1)\ge \edim(G)$, and this bound is sharp.
\end{theorem}

For the sharpness of the bound, it is used the graph $K_n\odot K_1$. Although the bound above is tight, one can see that the difference $\edim(G\odot K_1)-\edim(G)$ can be arbitrarily large. To observe this, in \cite{Peterin-2020} was given the following. Let $T$ be a tree of order $n\ge 3$ different from a path. Clearly, $l(T\odot K_1)=n$ and $ex(T\odot K_1)=n-l(T)$. By using the formula for the edge metric dimension of trees given in Table \ref{tab:edim-values}, $\edim(T\odot K_1)=l(T\odot K_1)-ex(T\odot K_1)=n-(n-l(T))=l(T)$. Thus, $\edim(T\odot K_1)-\edim(T) = l(T)-(l(T)-ex(T))=ex(T)$, which can be as large as desired.

We close this subsection by mentioning some results concerning the edge metric dimension of unicyclic graphs obtained in \cite{Sedlar-2021-c}. Among them, we remark that one which states that for any unicyclic graph $G$, the difference $\dim(G)-\edim(G)$ can only take the values $-1,0,1$. The unicyclic graphs for which such difference equals $-1$, $0$ or $1$ are described based on some possible configurations that can occur in a unicyclic graph.

\subsection{Mixed metric dimension}

The mixed metric dimension appeared in \cite{Kelenc-2017} as an approach for the (edge) metric dimension aimed to uniquely and indistinctly recognizing all the elements (vertices or edges) of a graph by means of distances to a given set of vertices.

A set $S$ of vertices of a connected graph $G$ is a \emph{mixed resolving set} (or \emph{mixed metric generator}) if any two elements (vertices or edges) of $G$ are distinguished by some vertex of $S$. A mixed resolving set of smallest possible cardinality is a \emph{mixed metric basis} of  $G$, and its cardinality is the \emph{mixed metric dimension}, denoted by $\mdim(G)$. Figure shows the example of a grid graph, where a mixed metric basis appears in red color. This is a particular case of a general case studied in \cite{Kelenc-2017}.

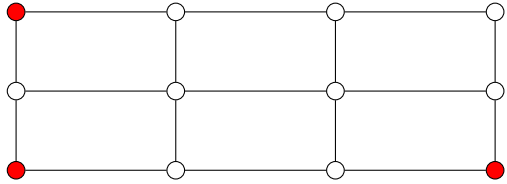
\begin{figure}[h]
\centering
\begin{tikzpicture}[scale=.7, transform shape]
\node [draw, shape=circle, fill=red] (a1) at  (0,0) {};
\node [draw, shape=circle] (a5) at  (0,1.5) {};
\node [draw, shape=circle, fill=red] (a9) at  (0,3) {};
\node [draw, shape=circle] (a2) at  (3,0) {};
\node [draw, shape=circle] (a6) at  (3,1.5) {};
\node [draw, shape=circle] (a10) at  (3,3) {};
\node [draw, shape=circle] (a3) at  (6,0) {};
\node [draw, shape=circle] (a7) at  (6,1.5) {};
\node [draw, shape=circle] (a11) at  (6,3) {};
\node [draw, shape=circle, fill=red] (a4) at  (9,0) {};
\node [draw, shape=circle] (a8) at  (9,1.5) {};
\node [draw, shape=circle] (a12) at  (9,3) {};

\draw(a1)--(a2)--(a3)--(a4)--(a8)--(a7)--(a6)--(a5)--(a1);
\draw(a5)--(a9)--(a10)--(a11)--(a12)--(a8);
\draw(a2)--(a6)--(a10);
\draw(a3)--(a7)--(a11);
\end{tikzpicture}
\caption{A mixed resolving set of the smallest possible cardinality appears in red.}\label{Fig_resolving-basis}
\end{figure}

In concordance with the NP-completeness of the decision problems concerning computing the metric and edge metric dimensions of connected graphs, it is not surprising that an analogous problem for the mixed metric dimension is of the same complexity class. This was precisely proved in \cite{Kelenc-2017}. The approach of the proof is also using a reduction from 3-SAT, and therefore, highly influenced by the related proofs for the metric and edge metric dimension complexity results.

\subsubsection{Mixed metric dimension versus metric and edge metric dimension}

It clearly happens that any mixed resolving set is also a resolving set as well as an edge resolving set. In this sense, the following relationship immediately follows. For any graph $G$,
\begin{equation}\label{mdim-dim-edim}
\mdim(G)\ge \max\{\dim(G),\edim(G)\}.
\end{equation}
On the other hand, one could think that the union of a resolving with an edge resolving set might be a mixed resolving set. However, this is indeed far from being true. From \cite{Khuller-1996} and \cite{Kelenc-2018} it is known that
$$\dim(T)=l(T)-ex(T)=\edim(T).$$
Moreover, a set having all but one terminal vertex of each exterior major vertex forms a metric basis (an edge metric) for $T$. Also, from \cite{Kelenc-2017} we know that for any tree $T$,
$$\mdim(T)=l(T).$$
In consequence, the following result can be easily deduced.

\begin{proposition}
\label{prop:dim-m-all-q}
For any integer $q$, there exists a graph $G$ for which $$\mdim(G)=\dim(G)+\edim(G)+q.$$
\end{proposition}

\begin{proof}
To see this, for a given integer $q$, we only need to consider a tree $T$ such that $q=2ex(T)-l(T)$. This means that
\begin{align*}
  \mdim(T) & =l(T) \\
            & =l(T)+2l(T)-2ex(T)-2l(T)+2ex(T) \\
            & =2l(T)-2ex(T)+2ex(T)-l(T)\\
            & =\dim(T)+\edim(T)+q,
\end{align*}
which is the required value for $\mdim(T)$.
\end{proof}

Notice that the value $q$ from the result above can be zero, or positive or negative. Moreover, trees for which $q=2ex(T)-l(T)$ can be easily constructed. Figure \ref{fig:tree} shows an example whether $q=6$ and other when $q=-2$.

\begin{figure}[h]
  \centering
\begin{tikzpicture}[scale=.5, transform shape]
\node [draw, shape=circle] (a0) at  (0,0) {};
\node [draw, shape=circle] (a1) at  (3,0) {};
\node [draw, shape=circle] (a2) at  (6,0) {};
\node [draw, shape=circle] (a3) at  (9,0) {};
\node [draw, shape=circle] (a4) at  (12,0) {};
\node [draw, shape=circle] (b1) at  (-3,0) {};
\node [draw, shape=circle] (b2) at  (-6,0) {};
\node [draw, shape=circle] (b3) at  (-9,0) {};
\node [draw, shape=circle] (b4) at  (-12,0) {};
\node [draw, shape=circle] (x00) at  (1,3) {};
\node [draw, shape=circle] (y00) at  (-1,3) {};
\node [draw, shape=circle] (a11) at  (3,3) {};
\node [draw, shape=circle] (a22) at  (6,3) {};
\node [draw, shape=circle] (a33) at  (9,3) {};
\node [draw, shape=circle] (a44) at  (12,3) {};
\node [draw, shape=circle] (a444) at  (13.5,3) {};
\node [draw, shape=circle] (b11) at  (-3,3) {};
\node [draw, shape=circle] (b22) at  (-6,3) {};
\node [draw, shape=circle] (b33) at  (-9,3) {};
\node [draw, shape=circle] (b44) at  (-12,3) {};
\node [draw, shape=circle] (b444) at  (-13.5,3) {};

\node at (-14,1.5) {\huge $T$};

\draw(a0)--(a1)--(a2)--(a3)--(a4);
\draw(a0)--(b1)--(b2)--(b3)--(b4);
\draw(x00)--(a0)--(y00);
\draw(a44)--(a4)--(a444);
\draw(b44)--(b4)--(b444);
\draw(a1)--(a11);
\draw(a2)--(a22);
\draw(a3)--(a33);
\draw(b1)--(b11);
\draw(b2)--(b22);
\draw(b3)--(b33);

\end{tikzpicture}

\vspace*{0.3cm}
\begin{tikzpicture}[scale=.5, transform shape]
\node [draw, shape=circle] (a0) at  (0,0) {};
\node [draw, shape=circle] (a2) at  (6,0) {};
\node [draw, shape=circle] (a4) at  (12,0) {};
\node [draw, shape=circle] (b2) at  (-6,0) {};
\node [draw, shape=circle] (b4) at  (-12,0) {};
\node [draw, shape=circle] (x00) at  (1,3) {};
\node [draw, shape=circle] (y00) at  (-1,3) {};
\node [draw, shape=circle] (a11) at  (3,3) {};
\node [draw, shape=circle] (a22) at  (6,3) {};
\node [draw, shape=circle] (a33) at  (9,3) {};
\node [draw, shape=circle] (a44) at  (12,3) {};
\node [draw, shape=circle] (a444) at  (13.5,3) {};
\node [draw, shape=circle] (b11) at  (-3,3) {};
\node [draw, shape=circle] (b22) at  (-6,3) {};
\node [draw, shape=circle] (b33) at  (-9,3) {};
\node [draw, shape=circle] (b44) at  (-12,3) {};
\node [draw, shape=circle] (b444) at  (-13.5,3) {};

\node at (-14.3,1.5) {\huge $T'$};

\draw(a0)--(a2)--(a4);
\draw(a0)--(b2)--(b4);
\draw(x00)--(a0)--(y00);
\draw(a44)--(a4)--(a444);
\draw(b44)--(b4)--(b444);
\draw(a2)--(a11);
\draw(a2)--(a22);
\draw(a2)--(a33);
\draw(b2)--(b11);
\draw(b2)--(b22);
\draw(b2)--(b33);
\end{tikzpicture}
\caption{A tree $T$ with $\mdim(T)=\dim(T)+\edim(T)+6$, and other one $T'$ with $\mdim(T')=\dim(T')+\edim(T')-2$.}\label{fig:tree}
\end{figure}
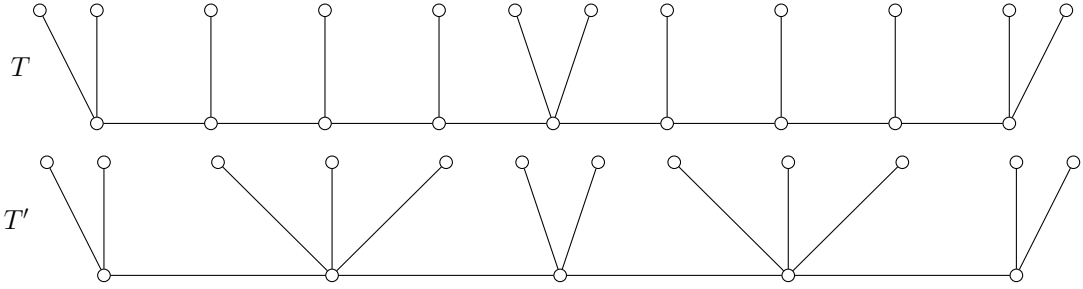

Proposition \ref{prop:dim-m-all-q} leads to claim the following result, which indeed means that the mixed metric dimension of a graph $G$ cannot in general be bounded from above by a constant factor of the sum $\dim(G)+\edim(G)$.

\begin{theorem}
The ratio $\frac{\mdim(G)}{\dim(G)+\edim(G)}$ cannot be bounded from above.
\end{theorem}

\begin{proof}
By using Proposition \ref{prop:dim-m-all-q}, for any integer $n$, one can construct a large enough tree $T$ such that $\dim(T)+\edim(T)=k$ for some integer $k$, and also $\mdim(T)=nk$ (by taking $q=(n-1)k$ in Proposition \ref{prop:dim-m-all-q}). Thus, it follows that $\frac{\mdim(G)}{\dim(G)+\edim(G)}=\frac{nk}{k}=n$.
\end{proof}

In view of Proposition \ref{prop:dim-m-all-q}, one would consider the problem of characterizing the graphs $G$ for which either $\mdim(G)<\dim(G)+\edim(G)$, or $\mdim(G)=\dim(G)+\edim(G)$, or $\mdim(G)>\dim(G)+\edim(G)$. By using the formulas $\dim(T)$ and $\edim(T)$ of any tree $T$, the following observations are easy to deduce.

\begin{remark}
Let $T$ be any tree. Then,
\begin{itemize}
  \item $\mdim(T)=\dim(T)+\edim(T)$ if and only if $\dim(T)=\edim(T)=l(T)/2$,
  \item $\mdim(T)>\dim(T)+\edim(T)$ if and only if $\dim(T)=\edim(T)<l(T)/2$, and
  \item $\mdim(T)<\dim(T)+\edim(T)$ if and only if $\dim(T)=\edim(T)>l(T)/2$.
\end{itemize}
\end{remark}

An interesting connection between the metric, edge metric and mixed metric dimension of graphs appears whether we consider hypercubes. For instance, it was proved in \cite{Kelenc-2021} the following result.

\begin{theorem}\emph{\cite{Kelenc-2021}}
For every $d\ge 3$ it holds,
            $$\dim(Q_d) = \mdim(Q_d).$$
\end{theorem}

This together with inequalities in \eqref{eq:dim-edim-Q_d} leads to the following. For every $d\ge 3$ it holds
$$\mdim(Q_d)-1=\dim(Q_d) -1 \le  \edim(Q_d) \le \dim(Q_d) = \mdim(Q_d),$$
and moreover, by using the asymptotical result for the metric dimension of hypercubes known from \cite{Cantor-1966} it is deduced that for every $d\ge 2$,
$$\mdim(Q_d) \sim \edim(Q_d) \sim \dim(Q_d) \sim \frac{2d}{\log_2 d}.$$

\subsubsection{Bounding the mixed metric dimension of graphs}

It is not difficult to see that the whole vertex set of any graph $G$ forms a mixed resolving set. Also, any vertex of $G$ and any incident edge with it, have the same distance to the vertex itself. In this sense, a vertex alone cannot form a mixed resolving set in $G$. As a consequence of these situations, the following remark from \cite{Kelenc-2017} is readily seen to be true.

\begin{remark}\emph{\cite{Kelenc-2017}}
For any graph $G$ of order $n$, $2\le \mdim(G)\le n$.
\end{remark}

It is then natural to consider characterizing the graphs attaining the limit values in the remark above. Such issues were settled in \cite{Kelenc-2017}, as the next two results show. To this end, we need the following terminology. Let $v$ be a vertex of a graph $G$. A vertex $u \in N_G(v)$ is said to be a \emph{maximal neighbor} of $v$ if all neighbors of $v$ (and $v$ itself) are also in the closed neighbourhood of $u$.

\begin{theorem}\emph{\cite{Kelenc-2017}}
Let $G$ be any graph of order $n$. Then $\mdim(G)=2$ if and only if $G$ is a path.
\end{theorem}

\begin{theorem}\emph{\cite{Kelenc-2017}}
Let $G$ be a graph of order $n$. Then $\mdim(G)=n$ if and only if every vertex of the graph $G$ has a maximal neighbor.
\end{theorem}

Other bounds for the mixed metric dimension of graphs are less common than for the edge metric dimension. However, a few of them them can be remarked, as the next one from \cite{Kelenc-2017}, in terms of the \emph{girth} $g(G)$ of the graph $G$.

\begin{theorem}\emph{\cite{Kelenc-2017}}
Let $G$ be a graph of order $n$. If $G$ has a cycle, then $\mdim(G)\le n-g(G)+3$.
\end{theorem}

Some other general bound for $\mdim(G)$ was given in \cite{Danas-2020-a} in terms of the minimum degree $\delta(G)$ as follows. The proof of it relies in some similar bounds that exist for the metric and edge metric dimensions of graphs (see Table \ref{tab:edge-dim-bound} for instance).

\begin{theorem}\emph{\cite{Danas-2020-a}}
If $G$ is a connected graph, then  $\mdim(G) \ge 1+\left\lceil\log_2(\delta(G)+11)\right\rceil$.
\end{theorem}

Another style of bound that was given in \cite{Danas-2020-a} involving the order, the size, the maximum degree, and the diameter of graphs is as follows.

\begin{theorem}\emph{\cite{Danas-2020-a}}
Let $G$ be a connected graph with mixed metric dimension $M$ and let $D$ be the diameter of graph $G$. Then
$$|V(G)|+|E(G)|\leq D^{M}+M(\Delta(G)+1).$$
\end{theorem}

We again must remark that the proof of such bound also relies in other similar bounds that exist for the metric and edge metric dimensions of graphs. The work \cite{Danas-2020-a} also contains an interesting computational contribution while comparing the quality of their new lower bounds with respect to four known lower bounds from the literature. We indeed write their conclusions: ``\emph{Testing was performed on two groups: all
21 connected graphs of 5 vertices and 12 well-known graphs with 10 up to 36 vertices. For the first group, one of proposed lower bounds in 20 out of 21 cases reached value of mixed metric dimension. For the second group, situation is quite opposite, so only in 2 cases proposed lower bound reached value of mixed metric dimension.}''

Other bounds for the mixed metric dimension of graphs are centered into special graph classes. For instances we can find some of them in \cite{Sedlar-2021-b,Sedlar-2021-d} for graphs with vertex disjoint cycles, with some emphasis on unicyclic graphs and cactus graphs. We remark the following ones.

\begin{theorem}{\em\cite{Sedlar-2021-d}}
Let $G \ne C_n$ be a cactus graph with $c$ cycles. Then $\mdim(G)\le l(G) + 2c$, and the upper bound is attained if and only if every cycle in $G$ has exactly one vertex of degree larger or equal to $3$.
\end{theorem}

\begin{proposition}{\em\cite{Sedlar-2021-b}}
Let $G$ be a 3-connected graph with cyclomatic number $c(G)$. Then $\mdim(G) < 2c(G)$.
\end{proposition}

\subsubsection{Mixed metric dimension of some graphs}

This subsubsection is devoted to research works aimed to computing the value of the mixed metric dimension of some families of graphs. This is summarized in Table \ref{tab:mdim-values}, where we also include the references where the corresponding values have been computed.

\begin{table}[ht]
  \centering
  \small{
  \begin{tabular}{|c|c|c|}
    \hline
    \textbf{Graphs} $\mathbf{G}$ & $\mathbf{\mdim(G)}$ & \textbf{Reference} \\ \hline
    Cycle graph $C_n$ & $3$ & \cite{Kelenc-2017} \\ \hline
    Unicyclic graphs & Several formulas and bounds & \cite{Sedlar-2021-b,Sedlar-2021-d} \\ \hline
    Complete bipartite graph $K_{r,t}$ & $\begin{array}{ll}
                                                                    r+t-1, & \text{if } r=2  \text{ or } t=2 \\
                                                                    r+t-2, & \text{otherwise}
                                                                  \end{array}$ & \cite{Kelenc-2017} \\ \hline
    Tree $T$ & $l(T)$ & \cite{Kelenc-2017} \\ \hline
    Torus graph $C_r\Box C_t$, $r,t\ge 3$ & $4$ & \cite{Danas-2020-a} \\ \hline
    Grid graph $P_r\Box P_t$ & 3 & \cite{Kelenc-2017} \\ \hline
    Prism graph $C_n\Box P_2$, $n\ge 5$ & $\begin{array}{ll}
                                                                    3, & \text{if } n  \text{ is even} \\
                                                                    4, & \text{if $n$ is odd}
                                                                  \end{array}$ & \cite{Raza-2020-a} \\ \hline
    Antiprism graph $\mathcal{A}_n$, $n\ge 3$ & $\begin{array}{ll}
                                                                    4, & \text{if } n  \text{ is even} \\
                                                                    5, & \text{if $n$ is odd}
                                                                  \end{array}$ & \cite{Raza-2020-a} \\ \hline
    Generalized Petersen graph $P(n,2)$ & $\begin{array}{ll}
                                                                    4, & \text{$n\cong\; 0, 2$ (mod 4)} \\
                                                                    5, & \text{$n\cong\; 1, 3$ (mod 4)}
                                                                  \end{array}$ & \cite{Raza-2020-b} \\ \hline
    Flower snark $\mathcal{J}_n$, $n\ge 5$ odd & $\begin{array}{ll}
                                                                    5, & n=5 \\
                                                                    4, & n\ge 7
                                                                  \end{array}$ & \cite{Danas-2020} \\ \hline
  \end{tabular}
  }
  \caption{Graphs for which their mixed metric dimension has been studied.}\label{tab:mdim-values}
\end{table}

\subsubsection{Miscellaneous results}

In a similar manner as for the edge metric dimension, the following model can be stated. Let $G$ be a graph of order $n$ and size $m$ with vertex set $V=\{v_1,\dots,v_n\}$ and edge set $E=\{e_1,\dots,e_m\}$. We consider the matrix $D=[d_{ij}]$ of order $(n+m)\times n$ such that $d_{ij}=d_G(x_i,x_j)$, $x_i\in V\cup E$ and $x_j\in V$. Now, given the variables $y_j\in \{0,1\}$ with $j\in \{1,2,\dots,n\}$ we define the following function:
$$\mathcal{F}(y_1,y_2,\dots,y_{n})=y_1+y_2+\dots+y_{n}.$$
Clearly, minimizing the function $\mathcal{F}$ subject to the following constraints
$$\sum_{i=1}^{n}|d_{ji}-d_{li}|y_i\ge 1\;\;\mbox{for every $1\le j<l\le n+m$},$$
is equivalent to finding a mixed metric basis of $G$, since the solution for $y_{1}, y_{2},\dots, y_{n}$ represents a set of values for which the function $\mathcal{F}$ achieves the minimum possible. The model above was first presented in \cite{Yero-2016-a}.

\subsection{Edge version of metric dimension}

It is very common in graph theory studying the dual version of a vertex parameter in a graph for an edge situation. This is indeed nothing more but studying the given parameter in the line graph of the graph. The metric dimension has not of course escaped to this. In \cite{Nasir-2019}, authors introduced a parameter called edge metric dimension defined as follows. For a connected graph $G$, a set of edges $W\subset E(G)$ is an \emph{edge resolving set} for $G$ if for every two distinct edges $e,f\in E(G)$ there is an edge $h\in W$ such that $d_G(e,h)\ne d_G(f,h)$. We understand the distance between two edges $xy,uv$ as $d_G(xy,uv)=\min\{d_G(x,u),d_G(x,v),d_G(y,u),d_G(y,v)\}$. Moreover, the cardinality of a smallest possible edge resolving set was called the \emph{edge metric dimension} of $G$, and denoted by $\dim_E(G)$. Since the concept of edge metric dimension was indeed known from before for some other structure, and based on what this other version indeed represents, the authors of \cite{Liu-2018} decided to rename the parameter in \cite{Nasir-2019} as the \emph{edge version of metric dimension}.

It must be remarked that in fact, this parameter had already been studied before in \cite{Feng-2013-a,Eroh-2014} for instances, although not from the perspective of a parameter, but of that of studying the classical metric dimension of line graphs.

In concordance with the recent survey \cite{Tillquist-2021}, we shall not consider exposing more results on this parameter here.

\subsection{Some open problems}

\begin{itemize}
  \item Infinite families of $k$-connected graphs with $k=1$ and $k=2$ are known for which $\edim(G)< \dim(G)$. Can you find the largest value $k$ for which a graph $G$ is $k$-connected and $\edim(G)< \dim(G)$?
  \item Compute the edge metric dimension of other grid like graphs: cylinders, toruses, etc.
  \item Based on inequality \eqref{eq:edim-dim-bipartite}, characterize all the bipartite graphs $G$ with $\edim(G)=\dim(G)$.
  \item Is it true that $\dim(Q_d)=\edim(Q_d)$ for every $d\ne 4$?
  \item Characterizing the graphs $G$ for which $\mdim(G)=\dim(G)$, $\mdim(G)=\edim(G)$ or $\mdim(G)=\dim(G)=\edim(G)$.
  \item Characterizing the graphs $G$ for which $\mdim(G)=\dim(G)+\edim(G)$.
  \item The problems of computing the metric and edge metric dimensions of graphs can approximated within a constant factor. Can the problem of computing the mixed metric dimension also be approximated within a constant factor?
  \item Prove the following conjecture from \cite{Sedlar-2021-d}: Let $G\ne C_n$ be a graph, $c(G)$ its cyclomatic number, and $l(G)$ the number of leaves in $G$. Then $\mdim(G) \le l(G) + 2c(G)$.
  \item Which is the complexity of computing the edge version of metric dimension?
\end{itemize}

%% file: fractional.tex
\section{ILP models and the fractional versions of metric dimension}

General Integer Linear Programming models for the metric dimension of graphs are already known in the literature from \cite{Chartrand-2000-a}. Since then, some variations (and specifications for special families of graphs) of this model have appeared here and there. One interesting contribution concerning such ILP model from \cite{Chartrand-2000-a} is one that standing over the ILP model, consider not only integers solutions for it, but also real solutions. Hence, the metric dimension concept can be ``fractionalized'', and it is understood from a different point of view. Such ideas were first described in \cite{Arumugam-2012}.

Consider a function $f$ defined on the vertex set $V(G)$. For a set $S \subseteq V(G)$, the \emph{weight} of $f$ is $f(S)=\sum_{s \in S}f(s)$. A real-valued function $f: V(G) \rightarrow [0,1]$ is a \emph{resolving function} of $G$ if $f(R\{x, y\}) \ge 1$ for any two distinct vertices $x, y \in V(G)$. The \emph{fractional metric dimension}, $\dim_f(G)$, of $G$ equals $\min\{f(V(G)): f \mbox{ is a resolving function of } G\}$. Figure \ref{Fig:fractional} shows a graph with a labeling that produces a resolving function $f'$ of minimum weight in such graph. Notice that, in this case, $|R\{x, y\}|\in \{6,7\}$ for any two distinct vertices $x, y \in V(C_8)$, and thus, $f'(R\{x, y\}) \ge 1$ for any two distinct vertices $x, y \in V(C_8)$.

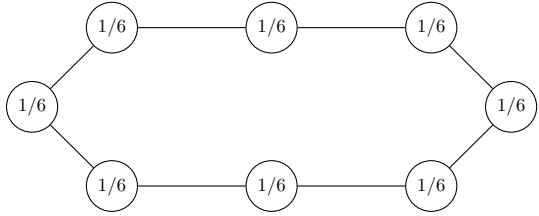
\begin{figure}[h]
\centering
\begin{tikzpicture}[scale=.7, transform shape]
\node [draw, shape=circle] (a1) at  (1.5,0) {1/6};
\node [draw, shape=circle] (a5) at  (0,1.5) {1/6};
\node [draw, shape=circle] (a9) at  (1.5,3) {1/6};
\node [draw, shape=circle] (a2) at  (4.5,0) {1/6};
\node [draw, shape=circle] (a10) at  (4.5,3) {1/6};
\node [draw, shape=circle] (a4) at  (7.5,0) {1/6};
\node [draw, shape=circle] (a8) at  (9,1.5) {1/6};
\node [draw, shape=circle] (a12) at  (7.5,3) {1/6};

\draw(a1)--(a2)--(a4)--(a8);
\draw(a1)--(a5)--(a9)--(a10)--(a12)--(a8);
\end{tikzpicture}
\caption{A labeling in the cycle $C_8$ that produces a resolving function of minimum weight.}\label{Fig:fractional}
\end{figure}

Assume $G$ is a graph of order $n$ with vertex set $V=\{v_1,\dots,v_n\}$, and consider the matrix $D=[d_{ij}]$ of order $n\times n$ such that $d_{ij}=d_G(x_i,x_j)$, $x_i,x_j\in V$. Now, given the real valued variables $y_j\in [0,1]$ with $j\in \{1,2,\dots,n\}$, let
$$\mathcal{F}(y_1,y_2,\dots,y_{n})=y_1+y_2+\dots+y_{n}.$$
Clearly, minimizing the function $\mathcal{F}$ subject to the following constraints
$$\sum_{i=1}^{n}|d_{ji}-d_{li}|y_i\ge 1\;\;\mbox{for every $1\le j<l\le n$},$$
leads to finding a fractional resolving function of minimum weight in $G$, since the solution for $y_{1}, y_{2},\dots, y_{n}$ represents a set of values for which the function $\mathcal{F}$ achieves the minimum possible, and viceversa.

Once this concept was made public, some investigations on it have been developed. Among them, we remark \cite{Arumugam-2013,Feng-2014,Feng-2013,Krismanto-2015,Liu-2019,Yi-2015}. For some extra information on fractionalization of graph parameters we suggest the book \cite{Scheinerman-2011}.

One first comment on the fractional metric dimension of graphs is that $\dim_f(G)$ reduces to $\dim(G)$, if the codomain of the resolving functions is restricted to the set $\{0,1\}$. In this sense, it is clear that $\dim_f(G)\le \dim(G)$ for any graph $G$. It is known that $1\le \dim(G)\le n-1$ for every graph $G$ of order $n$, and the limit values of such trivial bounds are attained when $G$ is a path ($\dim(G)=1$) or $G$ is a complete graph ($\dim(G)=n-1$). However, for the fractional version, the upper bound can be significantly decreased as proved in \cite{Arumugam-2012}. They have obtained that for any graph $G$ of order $n$,
$$1\le \dim_f(G)\le n/2.$$
The proof of the upper bound relies on the following fact. If $k=\min\{|R\{u,v\}|\,:\,u,v\in V(G),\,u\ne v\}$, then the constant function $f$ defined on $V(G)$ as $f(v)=1/k$ for all $v\in V(G)$ is readily seen to be a resolving function. Thus, $\dim_f(G)\le n/k\le n/2$ since $|R\{u,v\}|\ge 2$ for any pair of distinct vertices $u,v$. A theoretical characterization of graphs achieving the equality in such upper bound was given also in \cite{Arumugam-2012}, and examples of graphs for which this is satisfied are complete graphs; complete graphs minus one edge; complete graphs of even order minus a perfect matching and complete multipartite graphs whose each partition set has at least two vertices. Further on, a complete characterization of the graphs $G$ for which $\dim_f(G)=n/2$ was presented in \cite{Arumugam-2013}.

In \cite{Feng-2014}, the upper bound $\dim_f(G)\le n/2$ was indeed improved to $\dim_f(G)\le n/k$ by using the same argument as before. The interesting fact in this paper was that they were able to prove the following result.

\begin{theorem}{\em\cite{Feng-2014}}
If $G$ is a vertex-transitive graph and $k=\min\{|R\{u,v\}|\,:\,u,v\in V(G),\,u\ne v\}$, then $\dim_f(G)=|V(G)|/k$.
\end{theorem}

Table \ref{tab:fractional} shows some other values for the fractional metric dimension of graphs which are already known, together with the corresponding references where they were obtained.

\begin{table}[h]
  \centering
  \begin{tabular}{|c|c|c|}
  \hline
  \textbf{Graph} & \textbf{Value} & \textbf{Ref.} \\ \hline
  Petersen graph & 5/3 & \cite{Arumugam-2012} \\ \hline
  Cycle $C_n$ & $\begin{array}{ll}
            \frac{n}{n-1}, & \mbox{$n$ odd}\\ [0.15cm]
            \frac{n}{n-2}, & \mbox{$n$ even}
          \end{array}$
   & \cite{Arumugam-2012} \\ \hline
  Hypercube $Q_n$ & 2 & \cite{Arumugam-2012} \\ \hline
  Wheel $W_n$, $n\ge 7$ & $\frac{n-1}{4}$ & \cite{Arumugam-2012} \\
  \hline
  Grid graph $P_r\Box P_s$ & 2 & \cite{Arumugam-2012} \\ \hline
  Bouquet of cycles $B_m$, $m\ge 2$ & $m$ & \cite{Kang-2013} \\ \hline
  $K_k\Box K_n$, $2\le k\le n$, $n\ge 3$ & $\frac{n}{2}$ & \cite{Arumugam-2013} \\ \hline
  Hamming graph $H_{n,k}$, $k\ge 3$ & $\frac{k}{2}$ & \cite{Feng-2014} \\ \hline
  $\begin{array}{c}
     \mbox{Johnson graph $J(n,k)$} \\
     4\le 2k\le n
   \end{array}
  $ & $\begin{array}{ll}
                                              3, & \mbox{$(n,k)=(4,2)$} \\ [0.15cm]
                                              35/17, & \mbox{$(n,k)=(8,4)$} \\ [0.15cm]
                                              \frac{n^2-n}{2kn-2k^2}, & \mbox{otherwise} \\ [0.15cm]
                                            \end{array}
  $ & \cite{Feng-2014} \\ \hline
  Generalized Jahangir graph $J_{5,k}$, $k\ge 4$ & $\begin{array}{ll}
            \frac{5k}{k+4}, & \mbox{$k\equiv 0$ (mod $2$)}\\[0.15cm]
            \frac{5k}{k+3}, & \mbox{$k\equiv 1$ (mod $2$)}
          \end{array}$ & \cite{Liu-2019} \\ \hline
\end{tabular}
  \caption{Fractional metric dimension of some graphs}\label{tab:fractional}
\end{table}

It is necessary to remark that the equality $\dim_f(G)=\dim(G)$ can occur in several situations. Examples of this are the friendship graphs and grid graphs, as shown in \cite{Arumugam-2012}. In connection with such equality, another interesting relationship between $\dim_f(G)$ and $\dim(G)$ was given in \cite{Feng-2014}, involving also the order of $G$.

\begin{theorem}{\em \cite{Feng-2014}}
\label{th:bound-dim_f-dim}
Let $G$ be a graph of order $n$. Then $\dim_f(G)\ge \frac{n}{n-\dim(G)+1}$. Moreover, the equality holds if and only if $G$ is isomorphic to a path, a complete graph, or an odd cycle.
\end{theorem}

In addition to the commented results, some bounds for the fractional metric dimension of Cartesian product graph were given in \cite{Feng-2014}. For instance, if $G$ and $H$ are two graphs, then
$$\dim_f(G)\le \dim_f(G\Box H)\le \max\{\dim_f(G),|V(H)|\}.$$
Other particular cases of Cartesian product graphs were studied in \cite{Arumugam-2013,Feng-2014}. Regarding this, the fractional metric dimension of the hierarchical product of graphs was studied in \cite{Feng-2013}, where it was indeed used as a tool, a kind of variation of the fractional metric dimension, which authors called as rooted fractional metric dimension. For more details, we indeed suggest to precisely see \cite{Feng-2013}. Other results concerning product graphs have appeared in the unpublished manuscript \cite{Feng-2012}, which is dedicated to the corona and the lexicographic products of graphs.

Finally, in \cite{Yi-2015}, some other studies on the fractional metric dimension of graphs were given. For instance, we remark the next formula for the case of trees. Some bounds and closed formulae for some families of permutation graphs were also proved in \cite{Yi-2015}.

\begin{theorem}{\em \cite{Yi-2015}}
If $T$ is a tree, then $\dim_f(T)=\frac{1}{2}(|L(T)|-ex_1(T))$.
\end{theorem}

The notion of fractionalizing the metric dimension has been also extended to other related parameters, and one can find now variants like fractional local metric dimension (\cite{Benish-2018}), fractional strong metric dimension (\cite{Kang-2013}), fractional edge metric dimension (\cite{Yi-2021}), and fractional $k$-metric dimension (\cite{Kang-2019}). We next shortly review a couple of them.

\subsection{Fractional strong metric dimension}

We consider $S\{x,y\}$ represents the set of vertices $z$ such that either $x$ lies on a shortest $y-z$ path, or $y$ lies on a shortest $x-z$ path. A real valued function $f: V(G)\rightarrow [0,1]$ is a \emph{strong resolving function} of $G$ if $f(S\{x,y\})=\sum_{z\in S\{x,y\}}f(z)\ge 1$ for any two distinct vertices $x,y\in V(G)$. The \emph{fractional strong metric dimension} of $G$ equals $\min\{f(V(G))\,:\,f \mbox{ is a strong resolving function for $G$}\}$, and is denoted as $\sdim_f(G)$. Figure \ref{Fig:fractional-strong} shows a labeling of the graph $C_8$ that produces a strong resolving function of minimum weight. The proof of this comes from the article 

\begin{figure}[h]
\centering
\begin{tikzpicture}[scale=.7, transform shape]
\node [draw, shape=circle] (a1) at  (1.5,0) {1/2};
\node [draw, shape=circle] (a5) at  (0,1.5) {1/2};
\node [draw, shape=circle] (a9) at  (1.5,3) {1/2};
\node [draw, shape=circle] (a2) at  (4.5,0) {1/2};
\node [draw, shape=circle] (a10) at  (4.5,3) {1/2};
\node [draw, shape=circle] (a4) at  (7.5,0) {1/2};
\node [draw, shape=circle] (a8) at  (9,1.5) {1/2};
\node [draw, shape=circle] (a12) at  (7.5,3) {1/2};

\draw(a1)--(a2)--(a4)--(a8);
\draw(a1)--(a5)--(a9)--(a10)--(a12)--(a8);
\end{tikzpicture}
\caption{A labeling in the cycle $C_8$ that produces a strong resolving function of minimum weight.}\label{Fig:fractional-strong}
\end{figure}
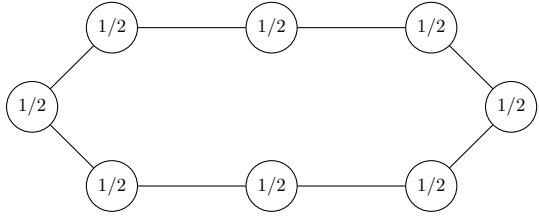

Notice that $\sdim_f(G)$ reduces to $\sdim(G)$ if the codomain of the strong resolving functions is restricted to the set $\{0,1\}$. The concepts above were first presented in \cite{Kang-2013}. In connection with the other fractional versions, it must be mentioned that the problem of finding the fractional strong metric dimension of graphs can be formulated as a linear programming problem as detailed in \cite{Fehr-2006-a}.

The main contributions on this parameter are divided into three styles of results: bounds, formulas for specific families, and the behaviour of $\sdim_f(G)$ in the case of product graphs. With respect to some interesting bounds, we recall the following ones.

\begin{theorem}{\em \cite{Kang-2013}}
Let $G$ be a connected graph of order $n\ge 2$. Then $1\le \sdim_f(G)\le \frac{n}{2}$. Further,
\begin{itemize}
  \item $\sdim_f(G)=1$ if and only if $G$ is a path.
  \item $\sdim_f(G) = \frac{n}{2}$ if and only if there exists a bijection $\alpha$ on $V(G)$ such that $\alpha(v) \ne v$ and $S\{v, \alpha(v)\} = \{v, \alpha(v)\}$ for every $v\in V(G)$.
\end{itemize}
\end{theorem}

An interesting connection between $\sdim_f(G)$ and $\sdim(G)$ for connected graphs $G$ was found in \cite{Kang-2016}, which is a result influenced by an analogous one for $\dim_f(G)$ (see Theorem \ref{th:bound-dim_f-dim}).

\begin{theorem}{\em \cite{Kang-2016}}
Let $G$ be a graph of order $n$. Then $\sdim_f(G)\ge \frac{n}{n-\sdim(G)+1}$. Moreover, the equality holds exactly when $G$ is a path or a complete graph.
\end{theorem}

\begin{theorem}{\em \cite{Kang-2013}}
For any connected graph $G$, $\sdim_f(G)\ge \frac{l(G)}{2}$.
\end{theorem}

For the case of trees, the following connection between their fractional strong metric dimension and that of the unicyclic graph obtained from a tree by adding any edge is known from \cite{Kang-2013}.

\begin{theorem}{\em \cite{Kang-2013}}
For a tree $T$, $\sdim_f(T + e) \ge \sdim_f(T)$, where $e\in E(\overline{T})$.
\end{theorem}

We recall that this style of bound presented in the next theorem is known as a Nordhauss-Gaddum result, which is a typical contribution in several areas of graph theory.

\begin{theorem}{\em \cite{Kang-2013}}
Let $G$ and $\overline{G}$ be connected graphs of order $n\ge 4$. Then $2\le \sdim_f(G) + \sdim_f(G)\le n$. Moreover,
\begin{itemize}
  \item $\sdim_f(G) + \sdim_f(G)=2$ if and only if $n = 4$.
  \item If $n\ge 5$ and $G$ is a unicyclic graph, then $\sdim_f(G) + \sdim_f(G) = n$ if and only if $G$ is the cycle $C_n$.
\end{itemize}
\end{theorem}

Now, with respect to formulas for $\sdim_f(G)$ for specific families of graph $G$, we summarize them in Table \ref{tab:sdim_f-values}.

\begin{table}[ht]
  \centering
  \small{
  \begin{tabular}{|c|c|c|}
    \hline
    \textbf{Graphs} $\mathbf{G}$ & $\mathbf{\sdim_f(G)}$ & \textbf{Reference} \\ \hline
    Petersen graph & $5$ & \cite{Kang-2013} \\ \hline
    Cycle graph $C_n$, $n\ge 3$ & $\frac{n}{2}$ & \cite{Kang-2013} \\ \hline
    Grid graph $P_r\Box P_t$, $r,t\ge 2$ & $2$ & \cite{Kang-2013} \\ \hline
    Tree $T$ with $l(T)$ leaves & $\frac{l(T)}{2}$ & \cite{Kang-2013} \\ \hline
    Wheel graph $W_{1,n}$, $n\ge 3$ & $\begin{array}{ll} 2, & n = 3\\ \frac{n-1}{2}, & n \geq 4  \end{array}$ & \cite{Kang-2013} \\ \hline
    Vertex-transitive graph $G$ & $\frac{|V(G)|}{2}$ & \cite{Kang-2016} \\ \hline
    $\begin{array}{c}
      \mbox{Complete $t$-partite graph $K_{r_1,\ldots,r_t}$,} \\
      \mbox{order $n=\sum_{i=1}^{t}r_i$}
    \end{array}$
     & $\begin{array}{cl}
\frac{n-1}{2}, & \mbox{if $r_i=1$ for only one $i\in\{1,2,\dots,t\}$} \\
\frac{n}{2}, & \mbox{otherwise} \\ %
\end{array}$ & \cite{Kang-2013} \\ \hline
  \end{tabular}
  }
  \caption{The fractional strong metric dimension of some families of graphs.}\label{tab:sdim_f-values}
\end{table}

The last point of interest in this subsection concerns contributions on the fractional strong metric dimension of product graphs, specifically on the corona product, the lexicographic product, and the Cartesian product of graphs. For this topic, it is remarkable the use of the strong resolving graph of a graph, already defined in Section \ref{Sec:strong-dim}. For instance, the following general results are worth of mentioning. By $\nu(G)$ we mean the \emph{matching number} of $G$.

\begin{theorem}{\em \cite{Kang-2018}}
For any connected graph $G$, $\sdim_f(G)\ge \nu(G_{SR})$.
\end{theorem}

\begin{theorem}{\em \cite{Kang-2018}}
Let $G$ be a connected graph. If each connected component of $G_{SR}$ is a regular graph, then $\sdim_f(G) = \frac{|\partial(G)|}{2}$.
\end{theorem}

Now, by using these contributions above, as well as, the tool of the strong resolving graph together with some other techniques and arguments, the following results for product graphs were given in \cite{Kang-2018}.

\begin{proposition}\emph{\cite{Kang-2018}}
Let $G$ be a connected graph of order $n \ge 2$, and let $H$ be a graph of order $m \ge 1$. Then $\sdim_f(G \odot H)=\frac{nm}{2}$.
\end{proposition}

\begin{proposition}\emph{\cite{Kang-2018}}
If $H$ is a connected graph, then $\sdim_f(H) \le \sdim_f(K_1 \odot H) \le \frac{1}{2}(1+|V(H)|)$ and both bounds are sharp.
\end{proposition}

We recall that the case for which a first glance indicates will be easier to manage ($G=K_1$) is the one that is more problematic to deal with.

For the case of lexicographic product graphs $G\circ H$, the results are highly dependent on the diameter of $H$, as well as, in the existence of true and false twins in $G$. We mention here, one of the main contributions, and suggest the reader to check \cite{Kang-2018} for several particular situations, and some other more general ones.

\begin{theorem}\emph{\cite{Kang-2018}}
Let $G$ be a connected graph of order $n \ge 2$ without true twin vertices, $H$ be a graph of order $m \ge 2$, and let
$|\partial(G)|=n'$ and $|\partial(H)|=m'$. If $D(H) \le 2$, then
$$\sdim_f(G\circ H) \ge \max\{n \sdim_f(H)+(m-m') \sdim_f(G),(n-n') \sdim_f(H)+m \sdim_f(G)\}$$
and
$$\sdim_f(G\circ H) \le \frac{1}{2}(nm'+mn'-n'm'),$$
where both bounds are sharp.
\end{theorem}

Finally, for the case of Cartesian product graphs, the contributions are highly based on the interesting result from \cite{Rodriguez-Velazquez-2014-a} which relates the strong resolving graph of $G\Box H$ and that ones of the direct product of the factors. Namely, the one which states that for any two connected graphs $G$ and $H$,
$$(G \square H)_{SR} \cong G_{SR} \times H_{SR}.$$
This fact together with other contributions in the topic allowed to deduce several results like the following ones for instances.

\begin{corollary}\emph{\cite{Kang-2018}}
Let $n\ge 2$ be an integer.
\begin{itemize}
\item[(a)] If $G_{SR}$ is a Hamiltonian graph, then $\sdim_f(G \square K_n)=\frac{n |\partial(G)|}{2}$.
\item [(b)] For any tree $T$ of order at least two, $\sdim_f(T \square K_n)=\frac{n\cdot l(T)}{2}$.
\item [(c)] Let $K_{r_1,\ldots,r_k}$ be a complete $k$-partite graph, where $k \ge 2$. If $r_i \ge 2$ for each $i\in \{1,2,\ldots,k\}$, or $r_j=1$ for at least two different $j \in \{1, 2, \ldots, k\}$, then $\sdim_f(K_{r_1,\ldots,r_k} \square K_n)=\frac{n}{2}\sum_{i=1}^k r_i$.
\end{itemize}
\end{corollary}

On the other hand, the following general bounds were proved also in \cite{Kang-2018}.

\begin{theorem}\emph{\cite{Kang-2018}}
Let $G$ and $H$ be connected graphs of order at least two. Then
$$\max\{2\sdim_f(G), 2\sdim_f(H)\} \le \sdim_f(G \square H) \le \min\{|\partial(G)|\sdim_f(H), |\partial(H)|\sdim_f(G)\},$$
and both bounds are sharp.
\end{theorem}

\subsection{Fractional $k$-metric dimension}

This parameter introduced in \cite{Kang-2019}, emerges from two possible directions: a first one as the fractionalization of the $k$-metric dimension, and a second one, as an extension from the fractional metric dimension. In formal way, for a real number $k\ge 1$, a real-valued function $h: V(G) \rightarrow [0,1]$ is a \emph{$k$-resolving function} of $G$ if $h(R_G\{x, y\})=\sum_{v\in R_G\{x, y\}}h(v) \ge k$ for any distinct vertices $x, y \in V(G)$. The \emph{fractional $k$-metric dimension}, $\dim_f^k(G)$, of $G$ is $\min\{h(V(G))\!:\! h \mbox{ is a $k$-resolving function of } G\}$. Note  $\dim^1_f(G)=\dim_f(G)$ and that $\dim_f^k(G)$ reduces to $\dim^k(G)$ when the codomain of $k$-resolving functions is restricted to $\{0,1\}$. Figure \ref{Fig:k-fractional} shows a labeling of the cycle $C_8$ which produces a $k$-resolving function of minimum weight for every $k\in [1,n-2]$ (note that $[1,n-2]$ represents the interval of real numbers between $1$ and $n-2$).

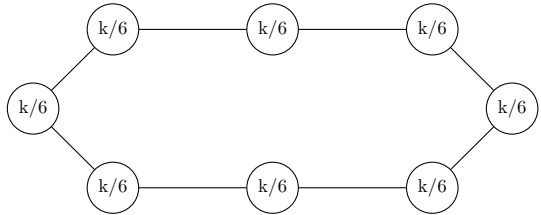
\begin{figure}[h]
\centering
\begin{tikzpicture}[scale=.7, transform shape]
\node [draw, shape=circle] (a1) at  (1.5,0) {k/6};
\node [draw, shape=circle] (a5) at  (0,1.5) {k/6};
\node [draw, shape=circle] (a9) at  (1.5,3) {k/6};
\node [draw, shape=circle] (a2) at  (4.5,0) {k/6};
\node [draw, shape=circle] (a10) at  (4.5,3) {k/6};
\node [draw, shape=circle] (a4) at  (7.5,0) {k/6};
\node [draw, shape=circle] (a8) at  (9,1.5) {k/6};
\node [draw, shape=circle] (a12) at  (7.5,3) {k/6};

\draw(a1)--(a2)--(a4)--(a8);
\draw(a1)--(a5)--(a9)--(a10)--(a12)--(a8);
\end{tikzpicture}
\caption{A labeling in the cycle $C_8$ that produces a $k$-resolving function of minimum weight for any $k\in [1,n-2]$.}\label{Fig:k-fractional}
\end{figure}

Clearly, similarly to the case of the $k$-metric dimension, there is a limit value for $k$ for which there are not $k'$-resolving functions for every $k'>k$. Indeed, the value of $k$ for which a graph $G$ is $k$-metric dimensional is precisely this mentioned limit value. In concordance, from now on, we shall write that $\kappa(G)$ is the value $k$ for which $G$ is $k$-metric dimensional. This means that $k$-resolving functions and fractional $k$-metric dimension are defined for every real number in the interval $[1,\kappa(G)]$.

Based on the fact that there is only one work, \cite{Kang-2019}, on the topic of fractional $k$-metric dimension of graphs, we shall only remark a couple of interesting results on the topic, and will suggest the reader to check such work for more information on the topic.

One significant remark from \cite{Kang-2019} is as follows. For any connected graph $G$ and for any $k \in [1,\kappa(G)]$, $\dim_f^k(G) \ge k \dim_f(G)$. This bound together with the technique of constructing some $k$-resolving function for $G$ allowed several times to compute the fractional $k$-metric dimension of graphs. In order to certify this we consider the following. Let $G$ be a connected graph and let $k \in[1, \kappa(G)]$. If there exists a minimum resolving function $g: V(G) \rightarrow [0,1]$ such that $g(v) \le \frac{1}{k}$ for each $v \in V(G)$, then $\dim_f^k(G)=k\dim_f(G)$ for any $k \in [1, \kappa(G)]$.

Another interesting fact from \cite{Kang-2019} is also as follows.

\begin{proposition}\emph{\cite{Kang-2019}}
Let $G$ be a connected graph of order $n$. For any $k \in [1, \kappa(G)]$, $\displaystyle k \le \dim_f^k(G) \le \frac{kn}{\kappa(G)},$ where both bounds are sharp.
\end{proposition}

Clearly, an interesting problem in then to characterizing the graphs achieving the bounds above. To this end, consider $\mathcal{R}_{\kappa}(G)= \displaystyle\bigcup_{x,y \in V(G), x \neq y, |R\{x,y\}|=\kappa} R\{x,y\}$, where $\kappa=\kappa(G)$. Then, for a connected graph $G$ of order $n \ge 2$ and for $k \in [1, \kappa(G)]$,
\begin{itemize}\setlength\itemsep{0em}
\item $\dim_f^k(G)=k$ if and only if $G \cong P_n$ and $k \in [1,2]$,
\item $\dim_f^k(G)=n$ if and only if $k=\kappa(G)=\kappa$ and $V(G)=\mathcal{R}_{\kappa}(G)$.
\end{itemize}

Note that from the first item, we deduce that if $G$ is not a path, or $k\ne 1,2$, then $\dim_f^k(G)>k$. However, it can be proved that for any real number $\epsilon>0$ and any integer $k$, there exists a graph $G$ such that $dim_f^k(G)\le k+\epsilon$. To see this, we consider for instance a cycle $C_n$ of even order for which $\dim_f^k(C_n)=\frac{kn}{n-2}$. That is, the function $g:V(C_n) \rightarrow [0,1]$ such that $g(u)=\frac{k}{n-2}$, for each $u \in V(C_n)$, is a minimum resolving function of $C_n$. Thus, if $n$ tends to be infinite, then the value of $\dim_f^k(C_n)$ tends to be $k$. It is then readily seen the claimed fact.

We end this subsection with some contribution on the fractional $k$-metric dimension of trees. We first consider the case of paths.

\begin{proposition}\emph{\cite{Kang-2019}}
Let $P_n$ be an $n$-path, where $n \ge 2$. Then $\dim_f^k(P_2)=k$ for $k \in [1,2]$ and, for $n \ge 3$,
\begin{equation*}
\dim_f^k(P_n)=\left\{
\begin{array}{ll}
k & \mbox{ if } k \in [1,2],\\
2+(k-2)\frac{n-2}{n-3} & \mbox{ if } k \in (2,n-1].
\end{array}\right.
\end{equation*}
\end{proposition}

For the case of trees different from a path, the study needs to be divided into two cases related to the number of exterior major vertices that the tree contains.

\begin{proposition}\emph{\cite{Kang-2019}}
Let $T$ be a tree with $ex(T)=1$. Let $v$ be the exterior major vertex of $T$ and let $\ell_1, \ell_2, \ldots, \ell_a$ be the terminal vertices of $v$ in $T$ (note that $a \ge 3$). Suppose that $d(v, \ell_1) \le d(v, \ell_2) \le \ldots \le d(v, \ell_a)$. Then,
\begin{itemize}\setlength\itemsep{0em}
\item[(a)] if $d(v, \ell_1)=d(v, \ell_2)$, then $\dim_f^k(T)=k\dim_f(T)=\frac{ka}{2}$ for $k \in [1, \kappa(T)]$;
\item[(b)] if $d(v, \ell_1) < d(v, \ell_2)$, then
\begin{equation*}
\dim_f^k(T)=\left\{
\begin{array}{ll}
\frac{ka}{2} & \mbox{ for } k\in [1, 2d(v, \ell_1)],\\
(a-1)k-(a-2)d(v, \ell_1) & \mbox{ for } k \in (2d(v, \ell_1), \kappa(T)].
\end{array}\right.
\end{equation*}
\end{itemize}
\end{proposition}

By using this result, we can then deduce the value of $\dim_f^k(T)$ for any tree $T$ with more than one exterior major vertex, where the following notations are necessary. Let $M_1(T)=\{w \in M(T): ter_T(w)=1\}$, $M_2(T)=\{w \in M(T): ter_T(w)=2\}$, and $M_3(T)=\{w \in M(T): ter_T(w) \ge 3\}$

\begin{theorem}\emph{\cite{Kang-2019}}
Let $T$ be a tree with $ex(T) \ge 2$. Then, for any $k \in [1,\kappa(T)]$,
\begin{equation*}
\dim_f^k(T)=k|M_2(T)|+\displaystyle\sum_{v\in M_3(T)}\dim_f^k(T_v).
\end{equation*}
\end{theorem}

Some other families of graphs like cycles, wheels, bouquets of cycles, complete multipartite graphs, and grid graphs are studied in \cite{Kang-2019}.

\subsection{Some open problems}

\begin{itemize}
  \item Can the graphs $G$ for which $\sdim_f(G) = \frac{n}{2}$ be characterized?
  \item Which is the complexity of computing the fractional $k$-metric dimension of graphs?
  \item Study the fractional versions of metric dimension for the other remaining products of graphs.
\end{itemize}

%% file: k-antidim.tex
\section{$k$-metric antidimension}

This notion have raised up in connection with the problem of quantifying how secure can be a social network against active attacks to its privacy. That is, the $k$-metric antidimension of graphs allows to generate a privacy measure for social graphs, and thus the study of the theoretical properties of this parameter is worthwhile.

Given an integer $k\ge 1$, a set of vertices $S\subset V(G)$ is known to be $k$-\emph{antiresolving set} for $G$ if $k$ is the largest positive integer such that for every vertex $v\in V(G)\setminus S$, there exists $k-1$ distinct vertices $v_1\dots,v_{k-1}\in V(G)\setminus S$, other than $v$, such that $d_G(v,x)=d_G(v_1,x)=\cdots=d_G(v_{k-1},x)$ for every vertex $x\in S$, and $k$ is the largest possible. The smallest cardinality among all $k$-antiresolving set for $G$ is the $k$-\emph{metric antidimension} of graphs, and is denoted by $\adim_k(G)$. A $k$-antiresolving set of cardinality $\adim_k(G)$ is called a $k$-\emph{antiresolving basis}. Concepts above were first presented in \cite{Trujillo-Rasua-2016}. 

As an example, consider for instance the double star $S_{r,t}$\footnote{A double star $S_{r,t}$ is a tree of order $r+t+2$ in which there two adjacent support vertices and the remaining vertices are leaves. One support vertex has $r$ adjacent leaves and the other one has $t$ adjacent leaves.}, with $r\ge t\ge 3$. Let $S$ be the set of the two support vertices of $S_{r,t}$. One can notice that for every leaf $v$, there are either $r-1$ or $t-1$ other vertices (the remaining leaves being false twins with $v$) having the same distance, as $v$, to every vertex of $S$. Thus, $S$ is a $t$-antiresolving set, and moreover $\adim_t(S_{r,t})=2$. Figure \ref{Fig:k-antidim} shows two drawings of the double star $S_{3,3}$, where the red vertices in the left hand side graph form a $3$-antiresolving set, and the blue vertices in the right side graph form a $2$-antiresolving set.

\begin{figure}[h]
\centering
\begin{tikzpicture}[scale=.7, transform shape]
\node [draw, shape=circle] (a1) at  (0,0) {};
\node [draw, shape=circle,fill=red] (a2) at  (2,0) {};
\node [draw, shape=circle,fill=red] (a3) at  (4,0) {};
\node [draw, shape=circle] (a4) at  (6,0) {};
\node [draw, shape=circle] (a5) at  (0,-1) {};
\node [draw, shape=circle] (a6) at  (0,1) {};
\node [draw, shape=circle] (a7) at  (6,-1) {};
\node [draw, shape=circle] (a8) at  (6,1) {};

\draw(a1)--(a2)--(a3)--(a4);
\draw(a5)--(a2)--(a6);
\draw(a7)--(a3)--(a8);
\end{tikzpicture}
\hspace*{0.7cm}
\begin{tikzpicture}[scale=.7, transform shape]
\node [draw, shape=circle] (a1) at  (0,0) {};
\node [draw, shape=circle,fill=blue] (a2) at  (2,0) {};
\node [draw, shape=circle,fill=blue] (a3) at  (4,0) {};
\node [draw, shape=circle,fill=blue] (a4) at  (6,0) {};
\node [draw, shape=circle] (a5) at  (0,-1) {};
\node [draw, shape=circle] (a6) at  (0,1) {};
\node [draw, shape=circle] (a7) at  (6,-1) {};
\node [draw, shape=circle] (a8) at  (6,1) {};

\draw(a1)--(a2)--(a3)--(a4);
\draw(a5)--(a2)--(a6);
\draw(a7)--(a3)--(a8);
\end{tikzpicture}
\caption{The red vertices in the left hand side graph form a $3$-antiresolving set, and the blue vertices in the right side graph form a $2$-antiresolving set of the double star $S_{3,3}$.}\label{Fig:k-antidim}
\end{figure}
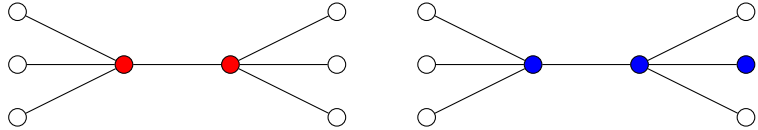

A first issue which one notes is that, there is a limit value for $k$ for which $k$-antiresolving sets are well defined. In this sense, a graph $G$ is called $k$-metric antidimensional if $k$ is the largest integer for which $G$ contains a $k$-antiresolving set. It is clear that the largest integer $k$ for which a graph could be $k$-metric antidimensional is the order of the graph minus one. On the other hand, there could be graphs for which such maximum value is just 1. For instance, if we consider a path of even order, then one can readily see that it does not contain any $k$-antiresolving set for every $k\ge 2$.

Once the first article on this topic was published, a few interesting investigations about it have been presented. There have been even appeared some variants of it which consider a different metric representation of the vertices in order to achieve the ``non-identification'' property. An example of this is for instance the $k$-adjacency antidimension version \cite{Mauw-2019}.

\subsection{The related privacy measure}

In order to define the privacy measure which uses the $k$-metric antidimension, one needs to make some considerations. First, it is understood that the metric representation of vertices with respect to a set of attacker nodes $S$ in a social graph $G$ is the adversary's background knowledge of a target (vertex). Through it, an attacker can uniquely recognize the vertices (or some vertices) of the graph.

The $(k, \ell)$-anonymity is a privacy measure generated from the adversary's background knowledge defined above, and it is based on the concept of \emph{$k$-antiresolving set} previously defined. It was said in \cite{Trujillo-Rasua-2016}, that a graph $G$ meets $(k, \ell)$-\emph{anonymity} with respect to active attacks if $k$ is the smallest positive integer such that the $k$-metric antidimension of $G$ is lower than or equal to $\ell$. This allows to claim that for a social graph $G$ that satisfies $(k, \ell)$-\emph{anonymity}, it is guaranteed that a user cannot be re-identified with probability higher than $\frac{1}{k}$ by an adversary controlling at most $\ell$ attacker nodes.

The value $k$ is used as a privacy threshold, while the value $\ell$ can be understood as an upper bound on the expected number of attacker
nodes that could exist in a network. In general, $\ell$ can be estimated through statistical analysis, and it is usually assumed that the number of attacker nodes is significantly lower than the order of the graph. For instance, it can be noted that for every $n > 0$ and $0 < \ell < n$, the complete graph $K_n$ meets $(n - \ell,\ell)$-anonymity, or equivalently, the complete graph $K_n$ guarantees that a user cannot be re-identified
with probability higher than $\frac{1}{n - \ell}$ by an adversary controlling $\ell$ attacker nodes.

These ideas clearly show that computing the $k$-metric antidimension of graph could be of interest for measuring how much secure the privacy of a graph is with respect to active attacks to its privacy.

Implementations and evaluations of the privacy measure in random graphs and in real social networks were made in the articles \cite{DasGupta-2019,Trujillo-Rasua-2016}. Indeed, \cite{DasGupta-2019} is more centered in these issues than the other work. As a conclusion of both studies, nothing really surprising for society was obtained, since precisely, the studies confirmed what we all know nowadays: that social networks are highly insecure with respect to privacy. However, at least a probabilistic value can be associated with any network that quantifies how much secure (or maybe we should write how much insecure) a network is, with respect to active attacks to its privacy.

A situation that frequently happens while evaluating the privacy achieved by a network is that the largest values for $k$ and $\ell$ in the $(k, \ell)$-anonymity measure are just $k=\ell=1$. In this sense, in \cite{Mauw-2018,Mauw-2016}, authors used such privacy measure and provided an efficient method to transform a graph $G$ into another graph $G'$ such that $G'$ will not be $(1, 1)$-anonymous. In other words, they obtained a graph $G'$  that satisfies  $(k,\ell)$-anonymity with $k>1$  or $\ell>1$. The anonymization methods used in \cite{Mauw-2018,Mauw-2016} are based on edge addition operations only, which means the graph $G$ is a subgraph of the new graph $G'$. In the work, it is provided a theoretical bound on the number of edges that are required to be added in order to transform a graph into another one that is not $(1, 1)$-anonymous. Several experimental results were also made in \cite{Mauw-2018}, on real-life graphs and a large collection of randomly generated graphs. This has shown that the described methods effectively prevent attacks from active adversaries with the capability of adding one node to the network, and additionally provide some level of protection against more capable attackers.

\subsection{$k$-metric antidimensional graphs}

Natural upper bound for $k$, which makes that a graph $G$ would be $k$-metric antidimensional is clearly the maximum degree of the graph, since the number of vertices at distance one from any vertex is at most the maximum degree of the graph. A particular case of this is as follows. Since the maximum degree of a graph is at most the order of the graph minus
one, it follows that if $G$ is any connected $k$-metric antidimensional graph of order $n$, then
$1\le k\le n-1$. Moreover, $G$ is $(n-1)$-metric antidimensional if and only if $G$ has maximum
degree $n-1$. 

In contrast with this latter situation, characterizing the graphs that are $1$-metric antidimensional seems to be a highly challenging problem, although knowing such graphs will be much worthy, since $1$-metric antidimensional graphs are those graphs which are not satisfying any privacy feature. That is, by controlling any set of vertices in a $1$-metric antidimensional graph, an attacker can always uniquely identify some elements of the graph. $1$-metric antidimensional trees and unicyclic graphs were already characterized in \cite{Trujillo-Rasua-2016-a}. In such a work, theoretical characterizations of $1$-metric antidimensional trees and unicyclic graphs were presented, and based on these characterizations, polynomial algorithms were developed in order to test whether a given tree or a unicyclic graph is $1$-metric antidimensional.

Some other lower bounds for the value $k$ such that a graph $G$ is $k$-metric antidimensional were given in \cite{Trujillo-Rasua-2016}, as well as, there were computed such values for some families of graphs. This is next presented.

\begin{proposition}\emph{\cite{Trujillo-Rasua-2016}}
\label{prop:k-antidimensional}
\begin{itemize}
  \item If the center of a graph $G$ is only one vertex, then $G$ is $k$-metric antidimensional for some $k\ge 2$.
  \item If a path $P_n$ has odd order, then it is $2$-metric antidimensional, and if it is even, then is $1$-metric antidimensional.
  \item Any cycle graph $C_n$ is $2$-metric antidimensional.
  \item Any complete bipartite graph $K_{r,t}$ with $r\ge t$ is $r$-metric antidimensional.
  \item Any tree $T$ is $k$-metric antidimensional for some $k\ge \max\{\phi(T),\xi(T)\}$\footnote{See \cite{Trujillo-Rasua-2016} for the formal definitions of $\phi(T)$ and $\xi(T)$}.
\end{itemize}
\end{proposition}

One detail that is of interest here is as follows. Assume $G$ is a $k'$-metric antidimensional graph. Hence, clearly $G$ does
not contain any $k''$-antiresolving set for every $k''> k'$.
In contrast, it is not \emph{a priori} clear if $G$ contains $k$-antiresolving sets for every $k<k'$. As an example, a complete graph $K_n$ is $(n-1)$-metric antidimensional and for every $1\le k\le n-1$, there exists a set of nodes in $K_n$ which is a $k$-antiresolving set.
Notwithstanding, if we consider the wheel graph $W_{1,n}$ with $n\ge 5$, then we can readily see that the central node $v_n$ is the unique $n$-antiresolving set. Also, $1$-antiresolving sets, $2$-antiresolving sets, and $3$-antiresolving sets exist as well. However, there are no
$k$-antiresolving sets for every $k\in\{4,\dots, n-1\}$. This motivates the following research question.\\

\noindent
\textbf{Open question:} For a given class of $k'$-metric antidimensional networks, can we decide if they also have $k$-antiresolving sets for all $1\le k\le k'-1$?\\

This question was answered positively for the case of trees in \cite{DasGupta-2019}. That is, there was proved that if $T$ is a $k'$-metric antidimensional tree, then for every $1\le k\le k'$ there exists a $k$-antiresolving set for $T$. The question remains open for any other family of graphs.

\subsection{Computational aspects of $\adim_k(G)$}

With respect to computational issues, the first contributions on this direction were presented in the seminal article \cite{Trujillo-Rasua-2016}. There was presented a not randomized true-biased algorithm for approximating the value of the $k$-metric antidimension of graphs, which was indeed used to make some empirical evaluations of the privacy features that achieves some synthetic as well as some real networks. Such algorithm is clearly exponential, but its complexity becomes significantly lower than the computational complexity of a brute force algorithm in a lot of situations.

As we could suspect the problem of computing the $k$-metric antidimension of graphs is NP-hard, and this result was proved independently in the two works \cite{Chatterjee-2019,Zhang-2017}. Both works are using a reduction from the Exact Cover by 3 Sets (X3C) problem. In \cite{Zhang-2017}, there are also established three bounds on the size of $k$-antiresolving sets in Erd\"os-R\'enyi random graphs. On the other hand, the work \cite{Chatterjee-2019} contains several other more results on computational aspects related to the $k$-metric antidimension of graphs. Specifically, there were studied the following optimization problems.\\

\noindent
\textbf{Problem ADIM:} Finding the largest value $k$ for which a given graph is $k$-metric antidimensional.

\noindent
\textbf{Problem ADIM$_{\ge k}$:} Given an integer $k$, finding a set of vertices which is a $k'$-antiresolving set for some $k'\ge k$.

\noindent
\textbf{Problem ADIM$_{=k}$:} Given an integer $k$, finding a set of vertices which is a $k$-antiresolving set.\\

For each of these problems above some complexity and approximation results were given in \cite{Chatterjee-2019}. For instance, we next remark the main results presented there.

\begin{theorem}\label{thm1}\emph{\cite{Chatterjee-2019}}\\
\noindent
{\bf (a)}
Both ADIM and ADIM$_{\ge k}$ can be solved in $O\left(n^4\right)$ time.

\smallskip
\noindent
{\bf (b)}
Both ADIM and ADIM$_{\ge k}$ can also be solved in $O\left(\frac{n^4 \,\log n}{k}\right)$ time ``with high
probability'' $($i.e., with a probability of at least $1-n^{-c}$ for some constant $c>0$$)$.
\end{theorem}

\begin{theorem}\label{k-large}\emph{\cite{Chatterjee-2019}}\\
\noindent
{\bf (a)}
ADIM$_{=k}$ is NP-complete for any integer $k$ in the range $1\leq k \leq n^\epsilon$, where
$0\leq \epsilon<\frac{1}{2}$ is any arbitrary constant, even if the diameter of the input graph is $2$.

\medskip
\noindent
{\bf (b)}
Assuming NP$\not\subseteq$\emph{DTIME}$\,(n^{\log\log n})$, there exists a universal constant $\delta>0$ such that
ADIM$_{=k}$ does not admit a $\left(\frac{1}{\delta}\ln n\right)$-approximation for any integer
$k$ in the range $1\leq k \leq n^\epsilon$, where
$0\leq \epsilon <\frac{1}{2}$ is any arbitrary constant, even if the diameter of the input graph is $2$.

\medskip
\noindent
{\bf (c)}
If $k=n-c$ for some constant $c$, then $\mathcal{L}_{opt}^{=k}=c$ if a solution exists and
ADIM$_{=k}$ can be solved in polynomial time.
\end{theorem}

The particular case $k=1$ was separately studied and the main contributions on this directions appear next.

\begin{theorem}\label{k-one}\emph{\cite{Chatterjee-2019}}\\
\noindent
{\bf (a)}
ADIM$_{=1}$ admits a $(1+\ln (n-1)\,)$-approximation in $O\left(n^3\right)$ time.

\medskip
\noindent
{\bf (b)}
If $G$ has at least one node of degree $1$, then
$\mathcal{L}_{opt}^{=1}=1$ and thus ADIM$_{=k}$ can be solved in $O\left(n^3\right)$ time.

\medskip
\noindent
{\bf (c)}
If
$G$ does not contain a cycle of $4$ edges, then
$\mathcal{L}_{opt}^{=1}\leq 2$ and thus ADIM$_{=k}$ can be solved in $O\left(n^3\right)$ time.
\end{theorem}

\subsection{Combinatorial aspects of $\adim_k(G)$}

In this direction, the investigations are centered into bounding or finding the $k$-metric antidimension of graphs. We now first correct some wrong results which were published in \cite{Trujillo-Rasua-2016}, concerning complete bipartite graphs. It was presented there the following result.

\begin{proposition}\emph{\cite{Trujillo-Rasua-2016}}
Let $r,t$ be two positive integers with $r\ge t$.
\begin{enumerate}
\item If $t< k\le r$, then $\adim_k(K_{r,t})=r+t-k$.
\item If $1< k\le t$, then $\adim_k(K_{r,t})=r+t-2k$.
\end{enumerate}
\end{proposition}

However, some parts of this result are not correct, and the corrected version of it is as follows.

\begin{proposition}
Let $r,t$ be two positive integers with $r\ge t$.
\begin{enumerate}
\item If $1< k< t$, then $\adim_k(K_{r,t})=t-k$.
\item If $k=t$, then $\adim_k(K_{r,t})=\left\{\begin{array}{ll}
                                         1, & \mbox{if $r>t$}, \\
                                         r, & \mbox{if $r=t$}.
                                       \end{array}\right.
$
\item If $t< k\le r$, then $\adim_k(K_{r,t})=r+t-k$.
\end{enumerate}
\end{proposition}

\begin{proof}
From Proposition \ref{prop:k-antidimensional} we know that $K_{r,t}$ is $r$-metric
antidimensional. Let $U$ and $V$ be the two partite sets of $K_{r,t}$
with $|U|=r$ and $|V|=t$. We first assume that $t< k\le r$. Let $A\subseteq U$
with $|A|=k$ and let be $S=(V\cup U)-A$. Notice that if $k=r$, then $A=U$ and
so, $S=V$. Since any vertex $v\notin S$ (or equivalently $v\in A$) is adjacent
to every vertex of $V$ and it has distance two to every vertex in $U-A$, we
have that all the vertices of $A$ have the same metric representation with
respect to $S$. As $|A|=k$, it follows that $S$ is a $k$-antiresolving set and
$\adim_k(K_{r,t})\le r+t-k$. Now, suppose $\adim_k(K_{r,t})<r+t-k$ and let $S'$
be a
$k$-antiresolving set for $K_{r,t}$. So, we have either one of the following
situations.
\begin{itemize}
\item There exist more than $k$ vertices of $U$ not in $S'$. Hence, for any
vertex $u\in U-S'$ there exist at least $k$ vertices not in $S'$ which,
together with $u$, have the same metric representation with respect to $S'$.
So, $S'$ is not a $k$-antiresolving set, but a $k'$-antiresolving set for some
$k'\ge k+1$, a contradiction.
\item There exists at least one vertex of $V$ not in $S'$. It is a direct
contradiction, since $|V|=t<k$.
\end{itemize}
Therefore, we obtain that $\adim_k(K_{r,t})=r+t-k$.

Now assume $k=t$. If $r=t$, then clearly $U$ and $V$ are the only two $k$-antiresolving sets of $K_{r,t}$. Thus, $\adim_k(K_{r,t})=r$. On the contrary, if $r>t$, then for any vertex $z\in V$ the $t=k$ vertices in $U$ have the same metric representation with respect to $\{z\}$ and for any vertex $w\in V-\{z\}$ there are at least $k-1$ vertices having the same metric representation as $w$ with respect to $\{z\}$. Thus, $\{z\}$ is a $k$-antiresolving sets of $K_{r,t}$ of minimum cardinality, or equivalently, $\adim_k(K_{r,t})=1$.

Finally, we assume that $1< k\le t$. Let $Y\subseteq V$ with $|Y|=k$ and let $Q=V-Y$. Hence,
for any vertex $v\in Y$, there exist
exactly $k-1$ vertices, such that all of them, together with $v$, have the same
metric representation with respect to $Q$. Moreover, for any vertex $u\in U$, there exist
at least $k$ vertices having the same
metric representation as $u$ with respect to $Q$. Thus, $Q$ is a $k$-antiresolving set and
$\adim_k(K_{r,t})\le t-k$.

Now, suppose that $\adim_k(G)<t-k$ and let $Q'$ be a $k$-antiresolving set
in $K_{r,t}$. Hence, there exist more than $k$ vertices of $U$ not in
$Q'$ or there exist more than $k$ vertices of $V$ not in $Q'$. Thus, in any
of both possibilities we obtain that $Q'$ is not $k$-antiresolving set, but a
$k'$-antiresolving set for some $k'\ge k+1$, a contradiction.
As a consequence, $\adim_k(K_{r,t})= t-k$.
\end{proof}

Other contributions on this direction are summarized in Table \ref{tab:k-antidim}.

\begin{table}
  \centering
  \small{\begin{tabular}{|c|c|c|}
    \hline
    \textbf{Graphs} $\mathbf{G}$ & $\mathbf{\adim_{k}(G)}$ & \textbf{Reference} \\ \hline
    Cycle $C_n$ ($k=2$)  & $\begin{array}{cc}
                                               =1, & \mbox{if $n$ is odd} \\
                                               =2, & \mbox{if $n$ is even}
                                             \end{array}$ & \cite{Trujillo-Rasua-2016} \\ \hline
    Tree $T$ & Some bounds and formulas & \cite{Trujillo-Rasua-2016} \\ \hline
    Some generalized Petersen graphs & Some formulas & \cite{Kratica-2019} \\ \hline
  \end{tabular}}
  \caption{Bounds and formulas for the $k$-metric antidimension of some families of graphs.}\label{tab:k-antidim}
\end{table}

\subsection{Some open problems}

\begin{itemize}
  \item Characterizing the family of all $1$-metric antidimensional graphs or at least some classes of such graphs.
  \item Is it the case that every Cartesian, strong and lexicographic products graphs are $k$-metric antidimensional for some $k\ge 2$?
  \item For a given class of $k'$-metric antidimensional networks, can we decide if they also have $k$-antiresolving sets for all $1\le k\le k'-1$?
  \item Find the $k$-metric antidimension of trees and unicyclic graphs.
  \item Find relationships between the $k$-metric antidimension of product graphs and that of its factors.
\end{itemize}